\RequirePackage{filecontents}

\documentclass[11pt]{amsart}
\usepackage[numbers]{natbib}
\usepackage{amssymb,mathrsfs,graphicx,enumerate}
\usepackage{amsthm}
\usepackage{colortbl}
\definecolor{black}{rgb}{0.0, 0.0, 0.0}
\definecolor{red}{rgb}{1.0, 0.5, 0.5}

\title[   ]{A new family of transportation costs with applications to  reaction-diffusion and parabolic equations 
with boundary conditions}
\usepackage{mathtools}
\DeclarePairedDelimiter{\ceil}{\lceil}{\rceil}

\author[J. Morales]{Javier Morales}
\address[Javier Morales]{\newline Department of Mathematics, \newline The University of Texas at Austin, Austin, TX 78712, USA}
\email{jmorales@math.utexas.edu}

\newtheorem{lemma}{Lemma}[section]
\newtheorem{theorem}[lemma]{Theorem}
\newtheorem{corollary}[lemma]{Corollary}
\newtheorem{proposition}[lemma]{Proposition}
\newtheorem{remark}[lemma]{Remark}

\numberwithin{figure}{section}
\newcommand{\beq}{\begin{equation}}
\newcommand{\eeq}{\end{equation}}
\newcommand{\bsp}{\begin{split}}
\newcommand{\esp}{\end{split}}

\newcommand\adots{\mathinner{\mkern2mu\raise1pt\hbox{.}
\mkern3mu\raise4pt\hbox{.}\mkern1mu\raise7pt\hbox{.}}}

\renewcommand{\div}{{\rm div}}

\def\charf {\mbox{{\text 1}\kern-.30em {\text l}}}
\begin{document}
%%%%%%%%%%%%%%%%

\date{\today}

\begin{abstract}This paper introduces a family of transportation costs between
non-negative measures. This family is
used to obtain parabolic and reaction-diffusion equations with drift, subject to
Dirichlet boundary condition, as the gradient flow of the entropy functional
$\int_{\Omega}\rho\log\rho+V\rho+1\hspace{1mm}dx$. In \cite{Figalli-Gigli}, Figalli and Gigli
study a transportation cost that can be used to obtain
parabolic equations with drift subject to Dirichlet boundary condition. 
However, the drift and the boundary condition are coupled 
in that work. The costs
in this paper allow the drift and the boundary condition 
to be detached. 
\end{abstract}

\maketitle \centerline{\date}


%\tableofcontents
\section{Introduction}

The use of optimal transport for the study of evolutionary equations
has proven to be a powerful method in recent years. More precisely,
one of the most surprising achievements of 
\cite{Jordan-Kin-Otto,Otto-microstructure,Otto-porous}
has been that many evolution equations of the form 
\[
\frac{d}{dt}\rho(t)={\rm{div}}\bigg(\nabla\rho(t)+\rho(t)\nabla V+\rho(t)(\nabla W\ast\rho(t)\big)\bigg),
\]
can be seen as gradient flows of some entropy functional on the space
of probability measures with respect to the Wasserstein distance:
\[
W_{2}(\mu,\nu)=\inf\bigg\{\sqrt{\int| x-y|^{2}\hspace{1mm}d\gamma(x,y)}\hspace{1em}:\hspace{1em}\pi_{1\#}\gamma=\mu,\pi_{2\#}\gamma=\nu\bigg\}.
\]
In addition to the fact that this interpretation allows one to prove entropy
estimates and functional inequalities (see \cite{V1,V2} for more details
on this area, which is still very active and in constant evolution),
this point of view provides a powerful variational method to prove
the existence of solutions to the above equations: given a time step $\tau>0$,
and an initial measure $\rho_{0}$, construct an approximate solution
by iteratively minimizing 
\[
\rho\rightarrow\frac{W_{2}(\rho,\rho_{n})}{2\tau}^{2}+\int\bigg[\rho\log\rho+\rho V+\frac{1}{2}\rho\big(W\ast\rho\big)\bigg]\hspace{1mm}dx=L[\rho|\rho_{n}],
\]
where $\rho_{n}$ is a minimum for $L[\rho|\rho_{n-1}]$. \\

 This approach will always produce solutions to parabolic equations
with Neumann boundary conditions. More recently, Figalli and Gigli
\cite{Figalli-Gigli} introduced a distance among positive measures
in an open domain $\Omega.$ Such a distance allows one to use this approach to build solutions
to the problem: 
\[
\begin{cases}
 & \frac{d}{dt}\rho(t)={\rm{div}}\bigg(\nabla\rho(t)+\rho(t)\nabla V\bigg)\hspace{1em}\mbox{in \ensuremath{\Omega},}\\
 & \rho=e^{-V}\hspace{1em}\mbox{on\hspace{1em}\ensuremath{\partial\Omega},}
\end{cases}
\]\\
in bounded domains. Note, however, that the boundary condition for $\rho$
is decided by the drift term appearing in the equation. 
Our goal here is to decouple the equation and the boundary condition. 
Also, we want to allow for the presence of a reaction term. Hence, inspired by \cite{Figalli-Gigli},
we introduce a new family of transportation
costs in a bounded open domain $\Omega.$ This family allows us to build weak solutions to
\begin{equation}
\begin{cases}
 & \partial_{t}\rho= {\rm div} \bigg(\nabla\rho(t)+\rho(t)\nabla V\bigg)-F_{x}^{\prime}(\rho)\hspace{1em} {\rm in}\hspace{1em}\Omega,\\
 & \rho=\rho_{D}\hspace{1em}\rm{on}\hspace{1em}\partial\Omega.
\end{cases}\label{Dirichlet}
\end{equation}
Here, $F$ is a function  on $[0,\infty)\times\overline{\Omega}.$  
We will use the notation
$F_{x}:=F(\hspace{1mm}\cdot\hspace{1mm},x).$
Also,  we denote  the first and second partial derivatives with respect to
the first variable by $F_{x}^{\prime}$ and 
$F_{x}^{\prime\prime}$. Our method works  for a wide class of reaction terms
$F_{x}^{\prime}.$ Some examples include
\[F_{x}^{\prime}(\rho)=W(x)\rho^{1+\beta}-Q(x),
\]
\[F_{x}^{\prime}(\rho)=W(x)\log\rho-Q(x),
\]
and
\[
F_{x}^{\prime}(\rho)=W(x)(\rho-1)|1-\rho|^{\alpha-1}-Q(x),  
\]
with $\alpha$ in $(0,1)$, $\beta\geq0,$ $W$ Lipschitz and strictly positive, 
and $Q$ Lipschitz and non-negative. (Note that when $V=W=1$ 
and $Q=0,$ the last example is equivalent to the equation 
$\partial_{t}u=\Delta u -u^{\alpha}$ via the change of variable $u=\rho-1,$
for non negative initial data.)\\
Now, we list sufficient conditions on $F:$
\begin{itemize}
\item[(F1)] 
$F_{x}$ is strictly convex for every $x$ in $\overline{\Omega}.$
\item[(F2)] 
For every $x$ in $\overline{\Omega},$ $F_{x}^{\prime}$
is a homeomorphism from $(0,\infty)$ to
$(\inf_{r>0} F_{x}^{\prime}(r),\infty).$ 
\item[(F3)] 
For every $r$ in $(0,\infty)$ the map $F_{x}^{\prime}$ is a continuous function
of $x.$
\item[(F4)] 
$
\lim_{r\rightarrow\infty}[F_{x}^{\prime}](r)=\infty\hspace{1mm}
{\rm{uniformly}}\hspace{1mm}{\rm{in}}\hspace{1mm} \ensuremath{x}.
$
\item[(F5)] There exist  positive constants $s,$ $s_{1},$ $B_{0},$ and $C_{0}$ such that,
\[
F_{x}^{\prime}(r)\leq C_{0}\hspace{1mm}r,
\]
for every $(r,x)$ in $(0,s)\times \overline{\Omega}$
and 
\[||\nabla_{x}[F_{x}^{\prime}](e^{p-V(x)})||_{L^{\infty}(\Omega)}\leq B_{0}, 
\]
for every $(p,x)$ in $(-\infty,-s_{1})\times \overline{\Omega}.$
\item[(F6)] The map
\[
 (h,x)\rightarrow \int_{[F_{x}^{\prime}]^{-1}(0)}^{[F_{x}^{\prime}]^{-1}(h)}(\log r+V)F_{x}^{\prime\prime}(r)dr,
\]
is Lipschitz on any compact subset of 
$\{(h,x) \in\mathbb{R}\times\overline{\Omega} 
\hspace{1mm}:\hspace{1mm}
[F^{\prime}_{x}]^{-1}(h(x))>0\}.$
\item[(F7)] For every $x$ in $\Omega,$ $F_{x}^{\prime} $ satisfies that either
\[
 \lim_{r\rightarrow0}F_{x}^{\prime}(r)=-\infty,
\]
or
\[
 \lim_{r\downarrow0}F_{x}^{\prime}(r)=F_{x}^{\prime}(0).
\]

\end{itemize}
We will assume that the drift, the  domain, and the boundary data satisfy:
\begin{itemize}
\item[(B1)] 
$V$ is Lipschitz.
\item[(B2)] 
$\Omega$ is Lipschitz, open, bounded,  and satisfies the interior ball
condition.
\item[(B3)] 
$\rho_{D}$ is Lipschitz and uniformly positive.
\end{itemize}
These transportation costs, that we shall define later, were found through a set of heuristic 
arguments (see Section 2). These arguments explore costs
that are related to a larger class of problems. Examples of these problems include:
\begin{equation}
\begin{cases}
 & \partial_{t}\rho={\rm div}\bigg(\nabla\rho(t)+\rho(t)\nabla V\bigg)-F_{x}^{\prime}(\rho)m(\rho)\hspace{1em}{\rm in}\hspace{1em}\Omega,\\
 & \rho=\rho_{D}\hspace{1em}{\rm {on}\hspace{1em}\partial\Omega,}
\end{cases}\label{Dirichlet2}
\end{equation}
and
\begin{equation}
\begin{cases}
 & \partial_{t}\rho={\rm div}\bigg(\nabla\rho(t)+\rho(t)\nabla V\bigg)-F_{x}^{\prime}(\rho)m(\rho)\hspace{1em}{\rm in}\hspace{1em}\Omega,\\
 & -\langle\nabla\rho-\nabla V\rho,\nu\rangle=g_{R}(\rho-\rho_{R})\hspace{1em}{\rm {on}\hspace{1em}\partial\Omega.}
\end{cases}\label{Robin}
\end{equation}
Here, the functions $g_{R},$ and $\rho_{R}$ are assumed to be uniformly positive. Also, $m\hspace{1mm}:\hspace{1mm}[0
,\infty)\rightarrow[0,\infty)$
is concave.\\

The author found this heuristic by combining several previous works. 
First, the work of Felix Otto on the formal Riemannian structure in
the space of probability measures \cite[Section 3]{Otto-Villani}. Second, the work of John Milnor \cite[Part III]{Milnor} on the formal 
Riemannian structure in the space of paths of a Riemannian manifold. 
Third, the work of Francesco Rossi  and Benedetto Piccoli \cite{Benedetto} on the generalization of 
the Benamou Brenier formula \cite{Benamou-Brenier} for positive measures. The last ingredient is 
a paper  by Figalli, Gangbo, and Yolcu \cite{Gangbo}, in which they successfully follow
the minimizing movement scheme for Lagrangian cost. The addition of 
nonlinear mobilities and the corresponding notion of generalized geodesics
has been studied in a diffrent context by J.A Carrillo, S. Lisini, G. Savare, and D. Slepcev \cite{mobility}.

The heuristic arguments are developed in the second section of this paper.
These are made rigorous only for the costs induced by Problem 1.1.
These costs produce
solutions to \eqref{Dirichlet} and \eqref{T2} via the minimizing movement scheme.
This is the main result of the paper: Theorem \ref{equation}.\\

Our family of costs depend on a positive number $\tau$ and two functions
\[
e\hspace{1mm}  :\hspace{1mm}\mathbb{R}\times\overline{\Omega}\rightarrow\mathbb{{R}}\cup\{\infty\},\\
\]
and
\[
\Psi\hspace{1mm}  :\hspace{1mm}\overline{\Omega}
\rightarrow \mathbb{R}.
\]

We will use the notation  
$e_{x}:=e(\hspace{1mm}\cdot\hspace{1mm},x).$
We will denote  the derivative of $e$ with respect to its first
entry by $e_{x}^{\prime}$. Additionally, for each fixed $x$, $[e_{x}^{\prime}]^{-1}$ denotes 
the inverse of such a derivative as a function of its own first entry. Analogous
notation will be used for $F$. We will denote the  interior of the set of points such that
$e$ is finite by $D(e)$ and the interior of the set of points $z$ such that $e(z,x)$ is finite by 
$D(e_{x}).$ We require that the functions
$\Psi$ and $e$ satisfy the 
following properties:
\begin{itemize}
\item[(C1)] 
 $\Psi$ is Lipschitz.
\item[(C2)] 
For each $x$ in  $\overline{\Omega}$, 
$e_{x}:=e(\hspace{1mm}\cdot\hspace{1mm},x)$ is  strictly
convex and lower semicontinuous.
\item[(C3)] 
For each  $L\in\mathbb{R},$ 
there exists $C(L)$ such that
\begin{equation*}
 e(z,x)\geq L| z|+C(L)\hspace{1em}
\forall (z,x)\in \mathbb{R}\times \overline{\Omega}.
\end{equation*}
\item[(C4)] 
The map $e$ is  Lipschitz in any compact subset of $D(e).$ 
(We regard $\overline{\Omega}$ as a topological space: Hence, the interior of 
any set of the form $\overline{A}\times\overline{\Omega},$ where $A$ is an open subset
of $\mathbb{R},$ is given by $A\times\overline{\Omega}$).
\item[(C5)] 
For each $x$ in $\Omega,$ the sets $D(e_{x})$ 
are of the form 
$(a(x),\infty),$ with $a(x)$ being either  a constant or negative infinity.
\item[(C6)]  
For each 
$x$, the map $e_{x}^{\prime}$ is a homeomorphism between $D(e_{x})$
and $\mathbb{R}.$
\item[(C7)] 
For each $r$ in $\mathbb{R},$ the map
$[e_{x}^{\prime}]^{-1}(r)$ is a continuous function
of $x$ and 
\[
 \lim_{r\rightarrow\infty } [e_{x}^{\prime}]^{-1}(r) =\infty,
\]
uniformly in $x.$
\item[(C8)] There exist  positive constants $s,$ $s_{1},$
$B_{0},$ and $C_{0}$ such that 
\begin{equation*}
[e_{x}^{\prime}]^{-1}\big(\log r+V(x)\big)\leq C_{0}\hspace{1mm}r,
\end{equation*}
for every  $(r,x)$ in  $(0,s)\times \overline{\Omega}$ and
\begin{equation*}
||\nabla_{x}[e^{\prime}_{x} ]^{-1}(p)||_{L^{\infty}(\Omega)}
\leq B_{0},
\end{equation*}
for every  $(p,x)$ in  $(-\infty,s_{1})\times \overline{\Omega}.$
\item[(C9)] The function $e$ statisfies that
\[
 \int_{\Omega} e(0,x) \hspace{1mm}dx=0.
\]
\end{itemize}
Item (C9) can be easily be relaxed by adding a constant to
$e;$ we have just assumed it for convenience.  The notations $e(h(x),x)$, $e(h)$, $e\circ h,$ and $e_{x}(h)$
will be used interchangeably. Similarly, we will freely interchange
$e^{\prime}(h(x),x)$, $e^{\prime}(h),$ 
$e_{x}^{\prime}(h),$ and $e^{\prime}\circ h.$

We will use $\Psi$ to obtain the desired boundary condition and
$e$ to control the reaction term. We define the cost
$Wb_{2}^{e,\Psi,\tau}$ on the set of positive measures with finite
mass $\mathcal{M}(\Omega)$, as a result of Problem 1.1, below.
\vspace{2mm}\\
\textbf{Problem 1.1 (A variant of the transportation problem).}
{\it Given $\mu,\rho\hspace{1mm}dx\in\mathcal{M}(\Omega),$ we consider the problem of minimizing 
\begin{equation}
C_{\tau}(\gamma,h):=\int_{\overline{\Omega}\times\overline{\Omega}\backslash \partial\Omega\times\partial\Omega}
\bigg( \frac{1}{2}\frac{| x-y|^{2}}{\tau}+
\Psi(y)1_{\Omega\times\partial\Omega}-\Psi(x)1_{\partial\Omega\times
\Omega}\bigg)\hspace{1mm}d\gamma+\tau\int_{\Omega}e(h)\hspace{1mm}dx,\label{cost}
\end{equation}
in the space $ADM(\mu,\rho)$ of admissible pairs $(\gamma,h).$ An admissible pair consists of
a positive measure $\gamma$ in $\overline{\Omega}\times\overline{\Omega}$
and a function $h$ in $L^{1}(\Omega)$. We require the pair to satisfy
\begin{equation}
\pi_{2\#}\gamma_{\overline{\Omega}}^{\Omega}=\rho\hspace{1mm}dx+\tau h \hspace{1mm}dx\hspace{1em}{\rm{\textit{and}}}\hspace{1em}
\pi_{1\#}\gamma_{\Omega}^{\overline{\Omega}}=\mu.\label{interior-mass}
\end{equation}
Here, the measure $\gamma_{A}^{B}$ denotes the restriction of $\gamma$ to 
$A\times B\subset\overline{\Omega}\times\overline{\Omega}.$
Also, the functions $\pi_{1}$ and $\pi_{2}$ are the canonical projections
of $\overline{\Omega}\times\overline{\Omega}$ into the first and second factor.\\}
Hence, \eqref{cost} provides a transportation cost between $\mu$ and $\rho$ given
by 
\[
Wb_{2}^{e,\Psi,\tau}(\mu,\rho):=\inf_{(\gamma,h)\in ADM}C_{\tau}
(\gamma,h).
\]
Additionally, we will denote by $Opt(\mu,\nu)$ the set of minimizers of Problem 1.1 with
$\mu$ and $\nu$ given.

The main objective is the following: given an initial measure $\rho_{0}$,
we build a family of curves $t\rightarrow\rho^{\tau}(t)$, indexed
by $\tau>0.$ We will do this by iteratively minimizing
\begin{equation}
\rho\rightarrow\int_{\Omega}[\rho\log\rho-\rho+V(x)\rho+1]\hspace{1mm}dx+Wb_{2}^{e,\Psi,\tau}(\rho_{n}^{\tau},\rho)=E^{\tau}[\rho|\rho_{n}^{\tau}],\label{step1}
\end{equation}
where $\rho_{n}^{\tau}$ is a minimum of $E^{\tau}[\rho|\rho_{n-1}^{\tau}]$
in $\mathcal{M}(\Omega)$. We define the discrete solutions by
\[
\rho^{\tau}(t):=\rho_{[t/\tau]}^{\tau}.
\]
We then show that as $\tau\downarrow0$, we can extract a subsequence
 converging to a weak solution to the problem:
\begin{equation}
\begin{cases}
 & \partial_{t}\rho={\rm{div}}\bigg(\nabla\rho+\rho\nabla V\bigg)-[e_{x}^{\prime}]^{-1}(\log\rho+V),\hspace{1em}{\rm{in}}\hspace{1em}\Omega,\\
 & \rho=e^{\Psi-V},\hspace{1em}{\rm{in}}\hspace{1em}\partial\Omega,\\
 & \rho(0)=\rho_{0}.
\end{cases}\label{T2}
\end{equation}
In particular when we set
\begin{equation}\label{Z1}
 e(z,x)  =
\begin{cases}
\int_{[F_{x}^{\prime}]^{-1}(0)}^{[F_{x}^{\prime}]^{-1}(z)}\big( \log r+V(x) \big)F_{x}^{\prime\prime}(r)\hspace{1mm}dr,
\hspace{1em}{\rm{if}} \hspace{1em} 
z>\inf_{r>0} F_{x}^{\prime}(r),\\
\liminf_{z\downarrow F^{\prime}_{x}(0)}\int_{[F_{x}^{\prime}]^{-1}(0)}^{[F_{x}^{\prime}]^{-1}(z)}\big( \log r+V(x) \big)F_{x}^{\prime\prime}(r)\hspace{1mm}dr,
\hspace{1em}{\rm{if}} \hspace{1em}z=\inf_{r>0} F_{x}^{\prime}(r),\\
+ \infty \hspace{1em}{\rm{otherwise},}
\end{cases}\\
\end{equation}
and
\begin{equation}\label{Z2}
  \Psi =\log\rho_{D}+V\hspace{1em}{\rm{on}}\hspace{1em}\partial\Omega,
\end{equation}
we obtain a weak solution to (\ref{Dirichlet}). \\
Whenever the reaction term satisfies (F1)-(F7), the drift, the boundary, and the boundary data satisfy (B1)-(B3), and $\Psi$
and $e$ are as above, then 
properties (C1)-(C9) are satisfied as well. 

We will require $\rho_{0}$ to be bounded and uniformly bounded away from zero. Using Proposition \ref{bar}, we will show the existence of positive constants $\lambda$ and $\Lambda$ such that  the weak solution satisfies
\[\frac{\lambda}{\sup e^{-V} }e^{-(C_{0}t+V)}\leq \rho(x,t)\leq \frac{\Lambda}{\inf e^{-V} }e^{-V},\]
for almost every $x$.\\
The paper is organized as follows: Section 2 introduces the heuristics used to
find the transportation costs. There, we explain the  process used to relate the cost
with the boundary conditions and reaction term. Section 3
is devoted to the study of Problem 1.1 and characterization of its
solutions in terms of convex functions. Section 4 is devoted to the proof
of the main result, Theorem \ref{equation}, which states the convergence
of the minimizing movement scheme to the weak solution. Section 5 is devoted to the
study of properties of the minimizers of $E^{\tau}[\rho|\rho_{0}]$
that we use to prove the main Theorem. Finally, Appendix A is used to prove
some technical properties of solutions to Problem 1.1 that are necessary in Section 3. 

\vspace{10em}
\textit{Acknowledgments: All the works used to find the new heuristics, with the exception of the work of Milnor, were introduced
to me by my PhD advisor Alessio Figalli. I want to express
my gratitude to him for his invaluable support and orientation. The work of Milnor was introduced 
to me by my undergraduate advisor Lazaro Recht. I also wish to express
my gratitude to him for his invaluable support and orientation}.
\section{Heuristics}
We define the cost $Wb_{2}^{e,m,\overline{e},\tau},$
as a result of Problem 2.1, below. \vspace{2mm}
\\
\textbf{Problem 2.1 (A variant of the transportation problem).}
\textit{Given $\mu,\nu\in\mathcal{M}(\Omega)$
we consider the problem of
minimizing
\begin{equation*}
\tilde{C}_{\tau}(V_{t},h_{t,},\overline{h_{t}})=
\int_{0}^{\tau}\bigg[\frac{1}{2}\int_{\Omega} 
|V_{t}|^{2}\rho_{t}\hspace{1mm}dx+\int_{\Omega} e(h_{t})m(\rho_{t})\hspace{1mm}dx
+\int_{\partial\Omega}\overline{e}(\overline{h}_{t})\hspace{1mm}d\mathcal
{H}^{d-1}\bigg]dt,
\end{equation*}
among all positive measured valued maps from $[0,\tau]$ to $\mathcal{M}(\Omega),$
satisfying $\rho_{0}\hspace{1mm}dx=\mu$ and $\rho_{\tau}\hspace{1mm}dx=\nu$. 
Here, the measures $\rho_{t}$
and the triplets
$(V_{t},h_{t,},\overline{h_{t}})$ are indexed by $t$ in $[0,\tau].$ We require them to  satisfy
the constraint }
\begin{equation}
\frac{d}{dt}\int_{\Omega}\zeta\rho_{t}\hspace{1mm}dx=
\int_{\Omega}\langle \nabla\zeta, V_{t} \rangle \rho_{t} \hspace{1mm}dx-
\int_{\Omega}\zeta h_{t}m(\rho)\hspace{1mm}dx
-\int_{\partial\Omega}\zeta\overline{h}_{t}
\hspace{1mm}d\mathcal{H}^{d-1},\hspace{1em}\forall t\in[0,\tau]\mbox{ and }\forall\zeta\in C^{\infty}(\overline{\Omega}).\label{cons}
\end{equation}
This provides a transportation cost between $\mu$ and $\nu$ given
by 
\[
Wb_{2}^{e,m,\overline{e},\tau}(\mu,\nu):=\inf \tilde{C}_{\tau}(V_{t},h_{t,},\overline{h_{t}}).
\]
Henceforth, a path is defined as a  measured valued map from $[0,\tau]$ to $\mathcal{M}(\Omega).$
We apply the minimizing movement scheme to this cost: given an initial
measure $\rho_{0}$, we build a family of curves $t\rightarrow\rho^{\tau}(t)$,
indexed by $\tau>0$, iterating the minimization of the map
\begin{equation*}
\rho\rightarrow\int_{\Omega}[\rho\log\rho-\rho+V(x)\rho+1]\hspace{1mm}dx+Wb_{2}^{e,m,\overline{e},\tau}(\rho_{n}^{\tau},\rho)=\tilde{E}^{\tau}[\rho|\rho_{n}^{\tau}],
\end{equation*}
where $\rho_{n}^{\tau}$ is a minimum of $\tilde{E}^{\tau}[\rho|\rho_{n-1}^{\tau}],$
in $\mathcal{M}(\Omega)$. We define the discrete solutions by 
\[
\rho^{\tau}(t):=\rho_{[t/\tau]}^{\tau}.
\]
Then, as $\tau\downarrow0$, we extract a subsequence converging
to a weak solution of the problem: 
\begin{equation*}
\begin{cases}
 & \partial_{t}\rho={\rm {div}}\bigg(\nabla\rho+\rho\nabla V\bigg)-[e_{x}^{\prime}]^{-1}(\log\rho+V)m(\rho),\hspace{1em}in\hspace{1em}\Omega,\\
 & -\langle\nabla\rho-\nabla V\rho,\nu\rangle=[\overline{e}_{x}^{\prime}]^{-1}(\log\rho+V)\hspace{1em}in\hspace{1em}\partial\Omega,\\
 & \rho(0)=\rho_{0}.
\end{cases}
\end{equation*}
In particular, when we set 
\begin{equation*}
e(h,x)  =
\begin{cases}
&\int_{[F_{x}^{\prime}]^{-1}(0)}^{[F_{x}^{\prime}]^{-1}(h)}(\log r+V)F_{x}^{\prime\prime}(r)\hspace{1mm}dr,
\hspace{1em} {\rm{if}} \hspace{1em}  [F_{x}^{\prime}]^{-1}(h)\geq0,\\
& +\infty,  \hspace{1em} {\rm{otherwise,}} \hspace{1em}
\end{cases}
\end{equation*}
and
\begin{equation*}
\overline{e}(\overline{h},x)=
\begin{cases}
&g_{R}\bigg(l(\overline{h})\log l(\overline{h})+(V-1)l(\overline{h})+1\bigg),
\hspace{1em} {\rm{if}} \hspace{1em}  l(\overline{h})\geq0,\\
& +\infty,  \hspace{1em} {\rm{otherwise,}} \hspace{1em}
\end{cases}
\end{equation*}
we obtain a weak solution for the problem  \eqref{Robin}. Here,
\[
l(r)=\frac{r}{g_{R}}+\rho_{R}.
\]
Also, we will show that
when we set
\begin{equation*}
e(h,x)  =
\begin{cases}
&\int_{[F_{x}^{\prime}]^{-1}(0)}^{[F_{x}^{\prime}]^{-1}(h)}(\log r+V)F_{x}^{\prime\prime}(r)\hspace{1mm}dr,
\hspace{1em} {\rm{if}} \hspace{1em}  [F_{x}^{\prime}]^{-1}(h)\geq0,\\
& +\infty,  \hspace{1em} {\rm{otherwise,}} \hspace{1em}
\end{cases}
\end{equation*}
and
\begin{equation*}
\overline{e}(\overline{h},x)  =(\log\rho_{D}+V)\overline{h},
\end{equation*}
we obtain a weak solution to \eqref{Dirichlet2}.\\

The heuristic is presented as follows. Section 2.1 characterizes optimal
triplets in terms of potentials. Section 2.2 describes a characterization
of minimal paths in terms of an equation for the potentials. Section
2.3 describes how the equation for minimal paths can be used to perform
the minimizing movement scheme. Section 2.4 describes
how to match the cost with the boundary conditions. Finally, section 2.5 describes
how to match the cost with the reaction term.   

\subsection{Optimal triplets}
In this section, we show an heuristic argument that characterizes minimizing
triplets for Problem 2.1. For such triplets, there exist functions $\varphi_{t}$
indexed in $[0,\tau]$, such that:
\begin{enumerate}[i]
\item[{(a)}] $\nabla\varphi_{t}=V_{t}$. 
\item[{(b)}] $\varphi_{t}=-\overline{e}^{\prime}(\overline{h}_{t})$ on $\partial\Omega$
and $\overline{h_{t}}=\langle\rho\nabla\varphi_{t},\nu\rangle$. 
\item[{(c)}] $\varphi_{t}=-e^{\prime}(h_{t})\mbox{\hspace{1em}in }$ $\Omega$. 
\end{enumerate}
In order to see this, we fix $t\in[0,\tau]$ and minimize
\[
\frac{1}{2}\int_{\Omega} |V_{t}|^{2}\rho_{t} \hspace{1mm} \hspace{1mm}dx+
\int_{\Omega} e(h_{t})m(\rho_{t})\hspace{1mm}dx+
\int_{\partial\Omega}{e}(\overline{h_{t}})\hspace{1mm}d\mathcal{H}^{d-1},
\]
under the constraint \eqref{cons}.\\ First, we prove (a). 

 Let us assume that we have a minimizer for Problem 2.1. Let 
 $(V_{t},h_{t},\overline{h}_{t})$ be the corresponding minimal 
 triplet at the given time. We proceed as in the classical case,
 \cite[Proposition 2.30]{users-guide}. Let $W$  be a  compactly supported
 vector field in the interior of $\Omega,$ with  ${\rm{div}}(\rho_{t}W)=0$.
 Then, $(V_{t}+sW,h_{t},\overline{h}_{t})$
still satisfies the constraint, for every $s$.
Hence, by minimality, we must have 
\[
\frac{d}{dt}\bigg|_{s=0}\frac{1}{2}\int_{\Omega}|V_{t}+sW|^{2}
\rho_{t}\hspace{1mm}dx+\int_{\Omega} e(h_{t})m(\rho_{t})\hspace{1mm}dx+
\int_{\partial\Omega}\overline{e}(\overline{h_{t}})\hspace{1mm}d\mathcal{H}^{d-1}=\int_{\Omega} 
\langle V_{t}, W \rangle \rho_{t}\hspace{1mm}dx=0.
\]

Since $W$ was an arbitrary vector field satisfying $\div(W\rho_{t})=0$, by the  Helmholtz-Hodge Theorem we obtain $\nabla\varphi_{t}=V_{t}$ for some 
$\varphi_{t}\hspace{1mm}:\hspace{1mm}\Omega\rightarrow\mathbb{R}.$\\
Second, we prove (b). \\
Let $\omega\hspace{1mm}:\hspace{1mm}\partial\Omega\rightarrow\mathbb{R}$
be a smooth function. Also, let $\alpha$ solve the elliptic problem 
\begin{equation}
\label{ep}
\begin{cases}
{\rm{div}}(\rho_{t}\nabla\alpha)=0\hspace{1em}in\hspace{1em}\Omega,\\
\langle\rho_{t}\nabla\alpha,\nu\rangle=\omega\hspace{1em}in\hspace{1em}\partial\Omega.
\end{cases}
\end{equation}
Then, $(V_{t}+s\nabla\alpha,,h_{t},\overline{h}_{t}+s\mbox{\ensuremath{\omega})}$
satisfies the constraint for any $s$. By minimality, we must have
\[
\frac{d}{ds}\bigg|_{s=0}\frac{1}{2}
\int_{\Omega}|\nabla\varphi_{t}+s\nabla\alpha|^{2}\rho_{t}\hspace{1mm}dx
+\int_{\Omega} e(h_{t})m(\rho_{t})\hspace{1mm}dx+
\int_{\partial\Omega}\overline{e}(\overline{h}_{t}+s\omega)\hspace{1mm}d\mathcal{H}^{d-1}=0.
\]
Hence, 
\[
\int_{\Omega}\langle \nabla\varphi_{t}, \nabla\alpha \rangle \rho_{t}\hspace{1mm}dx
+\int_{\partial\Omega}\omega\overline{e}^{\prime}(\overline{h_{t}})
\hspace{1mm}d\mathcal{H}^{d-1}=0.
\]
Integrating by parts and using \eqref{ep}, we obtain
\[
\int_{\partial\Omega}\omega(\overline{e}^{\prime}(\overline{h_{t}})+\varphi_{t})d\mathcal{H}^{d-1}=0.
\]
Since $\omega$ was arbitrary, we conclude
\[
\overline{e}^{\prime}(\overline{h}_{t})=-\varphi_{t}\hspace{1em}{\rm{on}}\hspace{1em}\partial\Omega.
\]
By \eqref{cons}, we must have
\begin{align*}\int_{\Omega}\zeta\partial_{t}\rho_{t}\hspace{1mm}dx & 
=\int_{\Omega} \langle \nabla\zeta , \nabla\varphi_{t}\rangle \rho_{t}  \hspace{1mm}dx
-\int_{\Omega} \zeta h_{t}\rho \hspace{1mm}dx-
\int_{\partial\Omega}\zeta\overline{h_{t}}\rho_{t}\hspace{1mm}d\mathcal{H}^{d-1}\\
 & =-\int_{\Omega}\zeta {\rm{div}}
 (\nabla\varphi\rho_{t})-\int_{\Omega}\zeta h_{t}\rho \hspace{1mm}dx+
 \int_{\partial\Omega}\zeta\bigg(\langle\nabla\varphi\rho_{t},\nu\rangle-\overline{h_{t}}\bigg)d\mathcal{H}^{d-1},\\
\end{align*}
for any $\zeta\hspace{1mm}:\hspace{1mm}\Omega\rightarrow\mathbb{R}.$\\
Thus, we conclude
\[
\langle\nabla\varphi\rho_{t},\nu\rangle=\overline{h}_{t}
\hspace{1em}{\rm{on}} \hspace{1em} \partial\Omega.
\]
Third, we show (c). \\
Let $\beta,$ $\eta\hspace{1mm}:\hspace{1mm}\Omega\rightarrow\mathbb{R}$
be smooth compactly supported functions satisfying
\begin{equation}
\label{ep2}
-{\rm{div}}(\nabla\beta\rho_{t})=m(\rho_{t})\eta.
\end{equation}
Then, for any $s,$ the triplet $(\nabla\varphi+s\nabla\beta,h+s\eta,\overline{h})$
is admissible. Consequently, we must have
\[
\frac{d}{ds}\bigg|_{s=0}\frac{1}{2}\int_{\Omega}|\nabla\varphi_{t}+
s\nabla\beta|^{2}\rho_{t}\hspace{1mm}dx+\int_{\Omega} e(h_{t}+s\eta)m(\rho_{t})\hspace{1mm}dx+
\int_{\partial\Omega}\overline{e}(\overline{h}_{t})\hspace{1mm}d\mathcal{H}^{d-1}=0.
\]
Hence,
\[
\int_{\Omega}\langle \nabla\varphi_{t}, \nabla\beta\rangle\rho_{t} \hspace{1mm}dx
+\int_{\Omega}\eta \hspace{1mm} e^{\prime}(h_{t})
m(\rho_{t})\hspace{1mm}dx=0.
\]
Integrating by parts and using \eqref{ep2}, we obtain
\[
\int_{\Omega}[\varphi_{t}+e^{\prime}(h_{t})]\eta m(\rho_{t})\hspace{1mm}dx=0.
\]
Since $\eta$ was arbitrary, we conclude
\[
e^{\prime}(h_{t})=-\varphi_{t}\hspace{1em}{\rm{in}}\hspace{1em}\Omega.
\]

\subsection{Optimal paths}
In this section, we will show an heuristic argument that characterizes
minimizers of Problem 2.1.

Let $\rho_{t}$, indexed in $[0,\tau],$ be a minimizer of Problem 2.1.
Also, for each $t$ in $[0,\tau],$ let  $\varphi_{t}$ be the potential generating the corresponding optimal triplet:
$\big(\nabla\varphi_{t},[e^{\prime}]^{-1}
(-\varphi_{t}),[\overline{e}^{\prime}]^{-1}(-\varphi_{t})\big).$
Then,
\begin{equation}
\partial_{t}\varphi+\frac{1}{2}|\nabla\varphi_{t}|^{2}-\big[\varphi_{t}[e^{\prime}]^{-1}(-\varphi_{t})+e([e^{\prime}]^{-1}(-\varphi_{t}))\big]m^{\prime}(\rho_{t})=0,\label{geo-eq}
\end{equation}
and
\[
\frac{d}{dt}\int_{\Omega}\zeta\rho_{t}\hspace{1mm}dx=
\int_{\Omega} \langle \nabla\zeta , \nabla\varphi_{t} \rangle \rho_{t}\hspace{1mm}dx
-\int_{\Omega}\zeta[e^{\prime}]^{-1}(-\varphi_{t})m(\rho_{t})\hspace{1mm}dx-
\int_{\partial\Omega}\zeta[\overline{e}^{\prime}]
^{-1}(-\varphi_{t})\hspace{1mm}d\mathcal{H}^{d-1},
\]
for every $\zeta$ in $C_{c}^{\infty}(\overline{\Omega}).$

In order to see this, we proceed by perturbing such minimizers. For each $t\in[0,\tau],$
we consider optimal triplets $\big(\nabla\omega_{t},[e^{\prime}]^{-1}(-\omega_{t}),[\overline{e}^{\prime}]^{-1}(-\omega_{t})\big).$
We require $\omega_{t}$ to be identically $0$ in the complement of a compact
subset of $(0,\tau)$.

Then, for each $s$, we let $t\rightarrow\rho_{t,s}$ and $t\rightarrow\big(\nabla\varphi_{t,s},[e^{\prime}]^{-1}(-\varphi_{t,s}),[\overline{e}^{\prime}]^{-1}(-\varphi_{t,s})\big)$
satisfy constraint \eqref{cons}. Additionally, for each $t,$ we require
the map $s\rightarrow\rho_{t,s}$ to satisfy
\[
\frac{d}{ds}\bigg|_{s=0}\int_{\Omega}\zeta\rho_{t,s}\hspace{1mm}dx=
\int_{\Omega}\langle \nabla\zeta,\nabla\omega_{t}\rangle \rho_{t,s}\hspace{1mm}dx
-\int_{\Omega}\zeta[e^{\prime}]^{-1}(-\omega_{t})m(\rho_{t,s})\hspace{1mm}dx
-\int_{\partial\Omega}\zeta[\overline{e}^{\prime}]^{-1}(-\omega_{t})\hspace{1mm}d\mathcal{H}^{d-1},
\]
and
\[
\rho_{t,0}=\rho_{t},\hspace{1em}\varphi_{t,0}=\varphi_{t}.
\]
Since $t\rightarrow\rho_{t}$ is a minimizer,  we must have
\begin{align*}
\frac{d}{ds}\bigg|_{s=0}\int_{0}^{\tau}\frac{1}{2}
\bigg[\int_{\Omega}
|\nabla\varphi_{t,s}|^{2}\rho_{t,s}\hspace{1mm}dx &+
\int_{\Omega} e([e^{\prime}]^{-1}(-\varphi_{t,s}))m(\rho_{t,s})\hspace{1mm}dx \\
&\hspace{10em}+\int_{\partial\Omega}\overline{e}([\overline{e}^{\prime}]^{-1}(-\varphi_{t,s}))
\hspace{1mm}d\mathcal{H}^{d-1}\bigg]dt=0.
\end{align*}
Consequently, 
\begin{align*}
\int_{0}^{\tau}\bigg[\int_{\Omega}\langle \nabla\varphi_{t,s},&\nabla\partial_{s}\varphi_{t,s}\rangle
\rho_{t,s}\hspace{1mm}dx  +\frac{1}{2}\int_{\Omega}|\nabla\varphi_{t,s}|^{2}\partial_{s}
\rho_{t,s}\hspace{1mm}dx-\int_{\Omega}\varphi_{t,s}\partial_{s}[e^{\prime}]^{-1}
(-\varphi_{t,s})m(\rho_{t,s})\hspace{1mm}dx\\
&+\int_{\Omega} e([e^{\prime}]^{-1}(-\varphi_{t,s}))m^{\prime}(\rho_{t,s})\partial_{s}\rho_{t,s}\hspace{1mm}dx
-\int_{\partial\Omega}\varphi_{t,s}\partial_{s}[\overline{e}^{\prime}]^{-1}(-\varphi_{t,s})\hspace{1mm}d\mathcal{H}^{d-1}\bigg]dt=0.
\end{align*}
Then,
\begin{multline*}
\int_{0}^{\tau} \bigg[ \frac{d}{ds}\bigg(\int_{\Omega}|\nabla\varphi_{t,s}|^{2}\rho_{t,s}\hspace{1mm}dx
-\int_{\Omega}\varphi_{t,s}[e^{\prime}]^{-1}(-\varphi_{t,s})m(\rho_{t,s})\hspace{1mm}dx
-\int_{\partial\Omega}\varphi_{t,s}[\overline{e}^{\prime}]^{-1}(-\varphi_{t,s})
\hspace{1mm}d\mathcal{H}^{d-1}\bigg)\\
 -\bigg(\int_{\Omega}\langle \nabla\partial_{s}\varphi_{t,s}, \nabla\varphi_{t,s}\rangle \rho \hspace{1mm}dx
 -\int_{\Omega}\partial_{s}\varphi_{t,s}[e^{\prime}]^{-1}(-\varphi_{t,s})m(\rho_{t,s})\hspace{1mm}dx
 \\
 -\int_{\partial\Omega}\partial_{s}\varphi_{t,s}[\overline{e}^{\prime}]^{-1}
 (-\varphi_{t,s})\hspace{1mm}d\mathcal{H}^{d-1}\bigg)
 -\frac{1}{2}\int_{\Omega}|\nabla\varphi_{t,s}|^{2}\partial_{s}\rho_{t,s}\hspace{1mm}dx\\
 +
 \int_{\Omega}\varphi_{t,s}[e^{\prime}]^{-1}(-\varphi_{t,s})m^{\prime}(\rho_{t,s})
 \partial_{s}\rho_{t,s}\hspace{1mm}dx
 \\+
 \int_{\Omega} e([e^{\prime}]^{-1}(-\varphi_{t,s}))m^{\prime}(\rho_{t,s})
 \partial_{s}\rho_{t,s}\hspace{1mm}dx\bigg]dt=0.
\end{multline*}
Recall that $h_{t,s}=[e^{\prime}]^{-1}(-\varphi_{t,s})$ and  $\overline{h}_{t,s}=[\overline{e}^{\prime}]^{-1}(-\varphi_{t,s})$. By \eqref{cons}, we get
\begin{equation}
\begin{aligned}
\int_{0}^{\tau}\bigg[\frac{d}{ds}&\int_{\Omega}\varphi_{t,s}\partial_{t}
\rho_{t,s}\hspace{1mm}dx - \int_{\Omega}\partial_{s}\varphi_{t,s}\partial_{t}
\rho_{t,s}\hspace{1mm}dx\\ 
&-\int_{\Omega}\bigg(\frac{1}{2}|\nabla\varphi_{t,s}|^{2}-
\big[\varphi_{t,s}[e^{\prime}]^{-1}(-\varphi_{t,s})
+e([e^{\prime}]^{-1}(-\varphi_{t,s}))\big]m^{\prime}(\rho_{t,s})\bigg)
\partial_{s}\rho_{t,s}\hspace{1mm}dx\bigg]dt=0.
\end{aligned}\label{pre-int-parts}
\end{equation}
By construction
$\partial_{s}\varphi_{\tau,s}=\partial_{s}
\rho_{\tau,s}=\partial_{s}\varphi_{0,s}=\partial_{s}\rho_{0,s}=0$. Hence,
if we integrate by parts in $t,$ we obtain
\[
-\int_{0}^{\tau}\bigg[\int_{\Omega}\bigg(\partial_{t}\varphi+\frac{1}{2}|\nabla\varphi_{t}|^{2}-\big[\varphi_{t}[e^{\prime}]^{-1}(-\varphi_{t})+e([e^{\prime}]^{-1}(-\varphi_{t}))\big]m^{\prime}(\rho_{t,s})\bigg)\partial_{s}\rho_{t,s}\hspace{1mm}dx\bigg]dt=0.
\]
This gives the desired result.
\subsection{The minimizing movement scheme.}
Given $\rho_{0}\in\mathcal{M}(\Omega)$ and $\tau>0,$ we provide heuristic
arguments to characterize the minimizers of
\begin{equation}
\begin{aligned} \{ \rho_{t} \}_{t\in[0,\tau]} \rightarrow\int_{0}^{\tau} \bigg(
\frac{1}{2}\int_{\Omega} |V_{t}|^{2}\rho_{t}\hspace{1mm}dx  +
\int_{\Omega} e(h_{t})m(\rho_{t})\hspace{1mm}dx
&+ \int_{\partial\Omega}\overline{e}(\overline{h}_{t})\hspace{1mm}d\mathcal{H}^{d-1} \bigg) dt\\
 & +\int_{\Omega}\big[\rho_{\tau}\log\rho_{\tau}+\big(V-1\big)\rho_{\tau}+1\big]\hspace{1mm}dx.
\end{aligned}
\label{scheme}
\end{equation}
Here, the
triplets $(V_{t},h_{t},\overline{h}_{t})$ satisfy \eqref{cons}.  Also, $\rho_{0}$ is fixed
and $\rho_{\tau}=\rho.$\\
In Section 2.1, we saw that for minimizing triplets we have for each
$t\in[0,\tau]$ a function $\varphi_{t},$ such that
\[
(V_{t},h_{t},\overline{h}_{t})=\big(\nabla\varphi_{t},[e^{\prime}]^{-1}(-\varphi_{t}),[\overline{e}^{\prime}]^{-1}(-\varphi_{t})\big).
\]
In Section 2.2, we found that optimal paths satisfy
\[
\partial_{t}\varphi_{t}+\frac{1}{2}|\nabla\varphi_{t}|^{2}-\big[\varphi_{t}[e^{\prime}]^{-1}(-\varphi_{t})+e([e^{\prime}]^{-1}(-\varphi_{t}))\big]m^{\prime}(\rho_{t})=0.
\]
In this section, we will show that minimizers of \eqref{scheme}
must satisfy
\[
\varphi_{\tau}=-\log\rho_{\tau}-V.
\]
In oder to see this, we suppose that we have a minimizer, $\rho_{\tau}$  and a path, 
$t\rightarrow\rho_{t},$
with corresponding triplets $t\rightarrow\big(\nabla\varphi_{t},[e^{\prime}]^{-1}(-\varphi_{t}),[\overline{e}^{\prime}]^{-1}(-\varphi_{t})\big)$.

We proceed by perturbing the path $t\rightarrow\rho_{t}$. For each $t$
we choose a function $\omega_{t}.$ We require these functions to be  identically $0$ in the
complement of a compact subset of $(0,\tau].$ This generates for each
$s$ a path, $t\rightarrow\rho_{t,s},$ as in Section 2.2, with the difference
that now the end point $\rho_{\tau,s}$ is free.\\
For a minimizer, we must have
\begin{align*}
\frac{d}{ds}\bigg|_{s=0}\bigg(\int_{0}^{\tau}\frac{1}{2}
&\int_{\Omega}|\nabla\varphi_{t,s}|^{2}\rho_{t,s}\hspace{1mm}dx 
+\int_{\Omega} e([e^{\prime}]^{-1}(-\varphi_{t,s}))m(\rho_{t,s})\hspace{1mm}dx\\&+\int_{\partial\Omega}\overline{e}([\overline{e}^{\prime}]^{-1}(-\varphi_{t,s}))\hspace{1mm}d\mathcal{H}^{d-1}dt+
\int_{\Omega} 
\big[\rho_{\tau,s}\log\rho_{\tau,s}+\big(V-1\big)\rho_{\tau,s}+1\big]\hspace{1mm}dx\bigg)=0.
\end{align*}
By \eqref{pre-int-parts} we have
\begin{align*}
\int_{0}^{\tau}
\bigg[ &\frac{d}{ds}\int_{\Omega}
\varphi_{t,s}\partial_{t}\rho_{t,s}\hspace{1mm}dx -
\int_{\Omega}\partial_{s}\varphi_{t,s}\partial_{t}
\rho_{t,s}\hspace{1mm}dx\\
&-\int_{\Omega}\bigg(\frac{1}{2}|\nabla\varphi_{t,s}|^{2}-
\big[\varphi_{t,s}[e^{\prime}]^{-1}(-\varphi_{t,s})
+e([e^{\prime}]^{-1}(-\varphi_{t,s}))\big]m^{\prime}
(\rho_{t,s})\bigg)\partial_{s}\rho_{t,s}\hspace{1mm}dx\bigg]dt\\
&\hspace{22em}+\int_{\Omega}\big[\log\rho_{\tau}+V\big]\partial_{s}\rho_{\tau,s} \\hspace{1mm}dx=0.\\
\end{align*}
Then, if we use \eqref{geo-eq} we obtain
\[
\int_{0}^{\tau}\bigg[\int_{\Omega}\varphi_{t,s}\partial_{t}\partial_{s}
\rho_{t,s}\hspace{1mm}dx+\int_{\Omega}\partial_{t}\varphi_{t,s}\partial_{s}\rho_{t,s}\hspace{1mm}dx+
\bigg]dt+\int_{\Omega}(\log\rho_{\tau}+V)\partial_{s}\rho_{\tau,s}\hspace{1mm}dx=0.
\]
Recall that by construction
$\partial_{s}\varphi_{0,s}=\partial_{s}\rho_{0,s}=0$. Integrating by parts  we get
\[
\int_{\Omega}\bigg(\varphi_{\tau}+\log\rho_{\tau}+V\bigg)
\partial_{s}\rho_{\tau,s}\hspace{1mm}dx =0.
\]
Thus, we obtain the desired result.

\subsection{The Boundary Conditions}
In Section 2.3, we showed that for minimizers of \eqref{scheme},
we have that $\varphi_{\tau}=-\log\rho_{\tau}-V$. In Section
2.1, we showed that for optimal triplets, 
\[
 -\varphi=\overline{e}^{\prime}(\overline{h}) \hspace{5mm} {\rm{on}}\hspace{5mm} \partial\Omega.
\]
Hence, if we set
$\overline{e}({\overline{h}})=\Psi \hspace{1mm}\overline{h}, $
we obtain the boundary condition
\[
 \rho_{\tau}=e^{\Psi-V} \hspace{5mm} {\rm{on}}\hspace{5mm} \partial\Omega. 
\]
This concludes the analysis for the boundary condition for 
\eqref{Dirichlet2}.\\In order to derive the boundary condition for \eqref{Robin}, we proceed as follows:
In Section 2.1, we showed that for minimizers of Problem 2.1, we must have
\[
\overline{h}_{t}=\langle\rho_{t}\nabla\varphi_{t},\nu\rangle.
\]
Hence, we expect the limit of the minimizing movement scheme to satisfy
the relation
\[
-\langle\nabla\rho,\nu\rangle-\langle\rho\nabla V,\nu\rangle=[\overline{e}^{\prime}]^{-1}(\log\rho+V).
\]
Our goal is to obtain the boundary condition
\[
-\langle\nabla\rho,\nu\rangle-\langle\rho\nabla V,\nu\rangle=g_{R}\big(\rho-\rho_{R}\big).
\]
For this purpose, we would need
\[
[\overline{e}^{\prime}]^{-1}(\log\rho+V)=g_{R}\big(\rho-\rho_{R}\big).
\]
Thus,
\[
V+\log\rho=[\overline{e}^{\prime}]\bigg(g_{R}\big(\rho-\rho_{R}\big)\bigg).
\]
Hence, if we set
\[
l(r)=\frac{r}{g_{R}}+\rho_{R},
\]
we obtain
\[
[\overline{e}^{\prime}](l^{-1}(\rho))=\log\rho+V.
\]
Then, it follows that
\[
[\overline{e}^{\prime}](l^{-1}(\rho))[l^{-1}(\rho)]^{\prime}=g_{R}(\log\rho+V).
\]
Integrating, we obtain
\[
\overline{e}(l^{-1}(\rho))=\int_{0}^{\rho}g_{R}\big(\log r+V\big)dr+C.
\]
Here, $C$ is a constant that will be chosen later. This implies
\[
\overline{e}(\rho)=g_{R}\int_{l(0)}^{l(\rho)}\big(\log r+V\big)dr+C.
\]
Thus, it suffices to set
\[
\overline{e}(\rho)=g_{R}\bigg(l(\rho)\log l(\rho)+\bigg(V-1\bigg)l(\rho)+1\bigg).
\]
\subsection{The reaction term}
In Section 2.1, we showed that optimal triplets satisfy
\[
e^{\prime}(h_{\tau})=-\varphi_{\tau}.
\]
In Section 2.3, we showed that minimizers of \eqref{scheme} satisfy
\[
\varphi_{\tau}=-\log\rho_{\tau}-V.
\]
Thus, in order to obtain
\[
h=F^{\prime}(\rho),
\]
we set
\[
e^{\prime}(F^{\prime}(\rho))=\log\rho+V.
\]
This implies
\[
e^{\prime}(F^{\prime}(\rho))F^{\prime\prime}(\rho)
=\big( \log\rho +V \big) F^{\prime\prime}(\rho).
\]
Integrating we obtain
\[
e(F^{\prime}(\rho))=\int_{0}^{\rho}\big( \log r+V \big) F^{\prime\prime}(r)\hspace{1mm}dr+C,
\]
for some constant $C$. 
Thus, it suffices to set 
\[
e(\rho)=\int_{[F^{\prime}]^{-1}(0)}^{[F^{\prime}]^{-1}(\rho)}\big( \log r+V \big) F^{\prime\prime}(r)\hspace{1mm}dr.
\]
This concludes the heuristic arguments. In the following sections we make these arguments rigorous for the case described in the introduction. 
\section{Properties of $Wb_{2}^{e,\Psi,\tau}$}

In this section we study minimizers of Problem 1.1. We begin by showing their existence.
\begin{lemma} \textbf{ (Existence of Optimal pairs)} Let $\mu$ and $\nu$ be 
absolutely continuous measures in $\mathcal{M}(\Omega)$. Then there exists
a minimizing pair for Problem 1.1.

\begin{proof} We claim the following:

\textit{Claim 1:} There exists a minimizing sequence of
admissible pairs $\{(\gamma_{n},h_{n})\}_{n=1}^{\infty}$ for which
the mass of $\{\gamma_{n}\}_{n=1}^{\infty}$ and $\{|h_{n}|\hspace{1mm}dx\}_{n=1}^{\infty}$
is equibounded and the plans in the sequence have no mass
concentrated on $\partial\Omega\times\partial\Omega.$

We assume this claim and
postpone its proof until the end of the argument. By (\ref{interior-mass})
the claim gives us a uniform bound in the total variation of $\{(\gamma_{n},h_{n}\hspace{1mm}dx)\}_{n=1}^{\infty}$.
Then, by compactness of $\overline{\Omega}$ and $\overline{\Omega}\times\overline{\Omega},$
for a subsequence $\{(\gamma_{n},h_{n})\}_{n=1}^{\infty}$,
not relabeled, we have weak convergence to regular Borel measures,
with finite total variation, $\gamma$ and $\mathbf{h}$. This 
converge is in duality with continuous
bounded functions in $\overline{\Omega}\times\overline{\Omega}$ and
$\overline{\Omega},$ respectively.

Assumption (C3) and the Dunford-Pettis Theorem allows 
us to conclude that $\mathbf{h}=h\hspace{1mm}dx,$
for some $h$ in $L^{1}(\overline{\Omega})$ and that $\{h_{n}\}_{n=1}^{\infty}$
converges to $h$ in duality with functions in $L^{\infty}(\overline{\Omega})$. Since $\pi_{2\#}\big(\gamma_{n}\big)_{\overline{\Omega}}^{\Omega}=\rho\hspace{1mm}dx+h_{n}\tau$,
we have that for any $\zeta\in 
C_{c}(\overline{\Omega}),$
\begin{multline*}
\int_{\overline{\Omega}\times\overline{\Omega}}
\zeta\circ\pi_{2}\hspace{1mm}d\gamma_{\overline{\Omega}}^{\Omega}
=\int_{\overline{\Omega}\times\overline{\Omega}}\zeta\circ\pi_{2}
\hspace{1mm}d\gamma=
\lim_{n\rightarrow\infty}\int_{\overline{\Omega}\times\overline{\Omega}}
\zeta\circ\pi_{2}\hspace{1mm}d\gamma_{n}\\ 
=\lim_{n\rightarrow\infty}\int_{\overline{\Omega}
\times\overline{\Omega}}\zeta\circ\pi_{2}
\hspace{1mm}d(\gamma_{n})_{\overline{\Omega}}^{\Omega}
=\lim_{n\rightarrow\infty}\int_{\overline{\Omega}}\zeta\rho\hspace{1mm}dx
+\tau\int_{\overline{\Omega}}\zeta h_{n}\hspace{1mm}dx=\int_{\overline{\Omega}}
\zeta\rho\hspace{1mm}dx+\tau\int_{\overline{\Omega}}\zeta h\hspace{1mm}dx.
\end{multline*}
Hence, $\pi_{2\#}\gamma_{\overline{\Omega}}^{\Omega}=\rho\hspace{1mm}dx+h\tau$.
It can also be shown that $\pi_{1\#}\gamma{}_{\Omega}^{\overline{\Omega}}=\mu$ in
an analogous way. This implies that $(\gamma,h)$ is in $ADM(\mu,\nu)$.

Since the sequence $\{ h_{n} \}_{n=1}^{\infty}$ converges weakly in $L^{1}(\Omega)$ to $h,$ using assumptions 
(C2)-(C6) and  \cite[Theorem 1]{Ioffe}, we get 
\[
 \liminf_{n\rightarrow\infty}\int_{\Omega}e(h_{n}(x),x)dx\geq\int e(h(x),x)dx.
\]
We also claim the following:

\textit{Claim 2:} there exists a further subsequence $\{\gamma_{n}\}_{n=1}^{\infty},$
not relabeled, with the property that $\{(\gamma_{n})_{\Omega}^{\Omega}\}_{n=1}^{\infty},$
$\{(\gamma_{n})_{\partial\Omega}^{\Omega}\}_{n=1}^{\infty},$ and $\{(\gamma_{n})_{\Omega}^{\partial\Omega}\}_{n=1}^{\infty}$
converge weakly to $\gamma_{\Omega}^{\Omega},\gamma_{\partial\Omega}^{\Omega},$ and
$\gamma_{\Omega}^{\partial\Omega}$ in duality with continuous and bounded
functions in $C(\overline{\Omega}\times\overline{\Omega})$, $C(\partial\Omega\times\overline{\Omega}),$
and $C(\overline{\Omega}\times\partial\Omega),$ respectively. We
will also postpone the proof of this claim until the end of the argument.

Since
$\Psi$ is bounded and continuous, this claim implies that
\begin{multline*}
\lim_{n\rightarrow\infty}\bigg[\int_{\overline{\Omega}\times\overline{\Omega}}
\frac{|x-y|^{2}}{2\tau}\hspace{1mm}d(\gamma_{n})_{\Omega}^{\Omega}
+\int_{\partial\Omega\times\overline{\Omega}}
\bigg(\frac{|x-y|^{2}}{2\tau}-\Psi(x)\bigg)
\hspace{1mm}d(\gamma_{n})_{\partial\Omega}^{\Omega}\\
+
\int_{\overline{\Omega}\times\partial\Omega}
\bigg(\frac{|x-y|^{2}}{2\tau}+\Psi(y)\bigg)
\hspace{1mm}d(\gamma_{n})_{\Omega}^{\partial\Omega}\bigg] 
=\int_{\overline{\Omega}\times\overline{\Omega}}\frac{|x-y|^{2}}{2\tau}
\hspace{1mm}d\gamma_{\Omega}^{\Omega}\\
+\int_{\partial\Omega\times\overline{\Omega}}
\bigg(\frac{|x-y|^{2}}{2\tau}-\Psi(x)\bigg)\hspace{1mm}d\gamma_{\partial\Omega}^{\Omega}
+
\int_{\overline{\Omega}\times\partial\Omega}
\bigg(\frac{|x-y|^{2}}{2\tau}+\Psi(y)\bigg)\hspace{1mm}d\gamma_{\Omega}^{\partial\Omega}.
\end{multline*}
Hence, this shows the existence of minimizers, provided we prove
the two claims. In order to prove the first one, we note that due to 
\eqref{cost} and \eqref{interior-mass}  we can assume, without loss of generality, that
the plans in the minimizing sequence have no mass concentrated on 
$\partial\Omega\times\partial\Omega$.
Also, due to (C3)
and \eqref{interior-mass}, 
\begin{align*}
C_{\tau}(\gamma,h) & \geq-||\Psi||_{\infty}\bigg(|\gamma_{\partial\Omega}^{\Omega}|+|\gamma_{\Omega}^{\partial\Omega}|\bigg)+K\int_{\Omega}|h|\hspace{1mm}dx+C(K)|\Omega|\\
 & \geq-||\Psi||_{\infty}\bigg(|\gamma_{\Omega}^{\partial\Omega}|+|\gamma_{\partial\Omega}^{\Omega}|\bigg)+K|h|(\Omega)+C(K)|\Omega|\\
 & \geq-||\Psi||_{\infty}\bigg(\mu(\Omega)+\nu(\Omega)+\tau|h|(\Omega)\bigg)+K|h|(\Omega)+C(K)|\Omega|,\\
 \end{align*}
for any $K$. Taking $K$ large enough, we obtain a uniform bound
on $|h|(\Omega)$ and consequently on $|\gamma|,$ for any minimizing
sequence. This proves the first claim.

As previously explained, this claim gives us a subsequence, not relabeled
$\{(\gamma_{n},h_{n})\}_{n=1}^{\infty}$, that converges weakly to
$(\gamma,h).$ To prove the second claim, we note that the measures
in the sequence $\{(\gamma_{n})_{\Omega}^{\Omega}$,$(\gamma_{n})_{\partial\Omega}^{\Omega},(\gamma_{n})_{\Omega}^{\partial\Omega}\}_{n=1}^{\infty}$
have uniformly bounded mass. Then, by compactness of 
$\overline{\Omega}\times\overline{\Omega},$ 
$\partial\Omega\times\overline{\Omega},$
and
$\overline{\Omega}\times\partial\Omega$ we can find a further
subsequence $\{(\gamma_{n})_{\Omega}^{\Omega}$,$(\gamma_{n})_{\partial\Omega}^{\Omega},(\gamma_{n})_{\Omega}^{\partial\Omega}\}_{n=1}^{\infty}$,
not relabeled, weakly converging to the measures $\sigma_{0},\sigma_{1},$
and $\sigma_{2}.$ This convergence is in duality with continuous and bounded functions
in $C(\overline{\Omega}\times\overline{\Omega})$, $C(\partial\Omega\times\overline{\Omega}),$
and $C(\overline{\Omega}\times\partial\Omega),$ respectively. Using
the definition of weak convergence, it is easy to verify that we must
have 
\begin{equation}
\label{decomp}
\gamma=\sigma_{0}+\sigma_{1}+\sigma_{2}. 
\end{equation}
We will prove the second claim by showing that
$\sigma_{0}=\gamma_{\Omega}^{\Omega},$ 
$\sigma_{1}=\gamma_{\partial\Omega}^{\Omega},$
and $\sigma_{2}=\gamma_{\Omega}^{\partial\Omega}$.
By \eqref{decomp}, this is a consequence
of the measures 
$\pi_{2\#}\sigma_{0},$ $\pi_{1\#}\sigma_{0},$ $\pi_{2\#}\sigma_{1},$
and $\pi_{1\#}\sigma_{2}$ having no mass concentrated in 
$\partial\Omega$. In
order to see that these measures have this property, we let $A\subset\partial\Omega$ be a compact set
and we take a sequence $\{\eta_{k}\}_{k=1}^{\infty}$ of uniformly bounded
functions in $C(\overline{\Omega})$ that decreases monotonically to $1_{A}.$
Additionally, we require that the sets $supp(\eta_{k})$ decrease
monotonically to $A.$ Since $\overline{\Omega}$ is bounded, by the
dominated convergence Theorem,
\[
\int_{A}\hspace{1mm}d\pi_{2\#}\sigma_{0}=\int_{\overline{\Omega}}1_{A}\circ\pi_{2}\hspace{1mm}d\sigma_{0}
=\lim_{k\rightarrow \infty}\int_{\overline{\Omega}}\eta_{k}\circ\pi_{2}\hspace{1mm}d\sigma_{0}.
\]
Also, by construction we have 
\begin{multline*}
\int_{\overline{\Omega}}\eta_{k}\circ\pi_{2}\hspace{1mm}d\sigma_{0}=
\lim_{n\rightarrow\infty}\int_{\overline{\Omega}}\eta_{k}\circ\pi_{2}
\hspace{1mm}d(\gamma_{n})_{\Omega}^{\Omega}\leq
\lim_{n\rightarrow\infty}\int_{\overline{\Omega}\times \overline{\Omega}}\eta_{k}\circ\pi_{2}
\hspace{1mm}d(\gamma_{n})_{\overline{\Omega}}^{\Omega}\\
=\lim_{n\rightarrow\infty}\int_{\overline{\Omega}}\eta_{k}\rho\hspace{1mm}dx
+\tau\int_{\overline{\Omega}}\eta_{k}h_{n}\hspace{1mm}dx=
\int_{\overline{\Omega}}\eta_{k}\rho\hspace{1mm}dx+\tau\int_{\overline{\Omega}}\eta_{k}h
\hspace{1mm}dx\\
=\int_{supp(\eta_{k})}\eta_{k}(\rho+\tau h)\hspace{1mm}dx\leq\sup(\eta_{k})\int_{supp(\eta_{k})}|\rho+\tau h| \hspace{1mm}dx.
\end{multline*}
Since $\{\eta_{k}\}_{k=0}^{\infty}$ is uniformly bounded and $supp(\eta_{k})$
converges monotonically to the set $A\subset\partial\Omega$ with
zero $\mathcal{L}^{d}$ measure, we have
\[
\int_{A}d\pi_{2\#}\sigma_{0}\leq\lim_{k\rightarrow\infty}
\sup(\eta_{k})\int_{supp(\eta_{k})}|\rho+\tau h| \hspace{1mm}dx=0.
\]
Thus, we conclude that $\pi_{2\#}\sigma_{0}(A)=0,$ for any measurable subset
$A$ of $\partial\Omega$; the proof for the measures $\pi_{1\#}\sigma_{0},\pi_{2\#}\sigma_{1},$
and $\pi_{1\#}\sigma_{2}$ is analogous. This establishes the second
claim. Consequently, the Lemma is proven.
\end{proof} \label{exists} \end{lemma}
We will use the following definitions: \vspace{1mm}

Given an admissible pair $(\gamma,h),$ we define
\begin{equation*}
d_{\Psi,\tau}(x)=
\begin{cases}
& \inf_{y\in\partial\Omega}\frac{| x-y|^{2}}{2\tau}+\Psi(y)\hspace{1em} {\rm{if}}
\hspace{1em} x\in\Omega,\\
& 0 \hspace{1em}{\rm{otherwise}},
\end{cases}
\end{equation*}
and
\begin{equation*}
d_{-\Psi,\tau}(y)=
\begin{cases}
&\inf_{x\in\partial\Omega}\frac{| x-y|^{2}}{2\tau}-\Psi(x)
\hspace{1em}{\rm{if}}\hspace{1em}y\in\Omega,\\
&0\hspace{1em}{\rm{otherwise}}.
\end{cases}
\end{equation*}
For any $x$ and $y$ in $\Omega$ we denote by $\mathcal{P}_
{\Psi,\tau}(x)$ and $\mathcal{P}_{-\Psi,\tau}(y)$
the sets where the infima are respectively attained.
Henceforth, $P_{\Psi,\tau}$ and $P_{-\Psi,\tau}$ will be measurable maps
from $\overline{\Omega}$ to $\partial\Omega$ such that
\[
 d_{\Psi,\tau}(x)=\frac{| x - P_{\Psi,\tau}(x)|^{2}}{2\tau} +
 1_{\Omega}(x)\Psi\big(
 P_{\Psi,\tau}(x)\big),
\]
and
\[
 d_{-\Psi,\tau}(y)=\frac{| y - P_{-\Psi,\tau}(y)|^{2}}{2\tau}
 -1_{\Omega}(y)\Psi\big(
 P_{-\Psi,\tau}(y)\big).
\]
It is well known that such  maps are uniquely defined on 
$\mathcal{L}^{d}$-a.e. in
$\Omega.$ (Indeed, $P_{\Psi,\tau}(x)$ and $P_{-\Psi,\tau}(y)$
are uniquely defined
whenever the Lipschitz functions 
$d_{\Psi,\tau_{|\Omega}}$ and
$d_{-\Psi,\tau_{|\Omega}}$
are differentiable
and they are given by 
$P_{\Psi,\tau}(x)=x-\nabla_{x} 
d_{\Psi,\tau}$
and  
$P_{-\Psi,\tau}(y)=y-\nabla_{y} 
d_{-\Psi,\tau}.$
 Here, we are
just defining them on the whole $\overline{\Omega}$ via a measurable selection argument (we omit the details).

Henceforth,
$P\hspace{1mm}:\hspace{1mm}\overline{\Omega}\rightarrow \partial\Omega$
will be a measurable map defined
in the whole $\overline{\Omega}$ with the property that 
\[
 |x-P(x)|=d(x,\partial\Omega) \hspace{1em} \forall x\in \overline{\Omega}.
\]
  We define the costs 
\[
\tilde{c}(x,y)=\frac{| x-y|^{2}}{2\tau}1_{(\partial\Omega\times
\partial\Omega)^{c}}-1_{\partial\Omega\times\Omega}\Psi(x)
+1_{\Omega\times\partial\Omega}\Psi(y),
\]
\[
c(x,y)=\frac{| x-y|^{2}}{2\tau},
\]
\[
c_{1}=c_{|\Omega\times\overline{\Omega}},
\]
and
\[
c_{2}=c_{|\overline{\Omega}\times\Omega}.
\]
Also, we define the set
\[
\mathcal{A}=\bigg\{(x,y)\in\overline{\Omega}\times\overline{\Omega}
\hspace{1em}:\hspace{1em}d_{\Psi,\tau}(x)+d_{-\Psi,\tau}(y)\geq\tilde{c}(x,y)\bigg\}.
\]
We will work with the topological space $(\overline{\Omega}\times\overline{\Omega},\mathcal{G})$.
The topology of this space built by considering the product topology, in the
spaces $\Omega\times\Omega,$ $\partial\Omega\times\Omega,$ $ \Omega\times\partial\Omega,$ and $\partial\Omega\times\partial\Omega,$
and then taking the disjoint union topology. In other words, the space $\overline{\Omega}\times\overline{\Omega}$
is equipped with the topology
\[
\partial\Omega\times\partial\Omega\coprod\partial\Omega\times\Omega\coprod\Omega
\times\partial\Omega\coprod\Omega\times\Omega.
\]
Hence, if we are given continuous functions $\{f_{i}\}_{i=1}^{4}$ from
the spaces $\Omega\times\Omega,$ $\partial\Omega\times\Omega,\Omega\times\partial\Omega,$ and $\partial\Omega\times\partial\Omega$ to
any other topological space $Y$, then there exists a unique continuous function
$f\hspace{1mm}:\hspace{1mm}\overline{\Omega}\times\overline{\Omega}\rightarrow Y$
such that
\[
f_{i}=f\circ\phi_{i}.
\]
Here, $\{\phi_{i}\}_{i=1}^{4}$ are the canonical injections of $\Omega\times\Omega,$ $\partial\Omega\times\Omega,\Omega\times\partial\Omega,$ 
and $\partial\Omega\times\partial\Omega$
into $\overline{\Omega}\times\overline{\Omega}.$
The support of the measures $\gamma$ in $\overline{\Omega}\times\overline{\Omega}$
will be taken with respect to this topology. Hence, given a positive $\gamma$
measure in $\overline{\Omega}\times\overline{\Omega},$ 
$supp(\gamma)$ is defined to be set of points  $(x,y)$ in
$\overline{\Omega}\times\overline{\Omega}$ such that for every 
$G$ in $\mathcal{G}$ containing $(x,y),$ we have $\gamma(G)>0.$

Additionally, we will use the notions of 
 $c$-cyclical monotonicity, 	
 $c$-transforms, 
 $c$-concavity, and 
 $c$-superdifferential. We refer the reader to
 \cite[Definitions 1.7 to 1.10]{users-guide}. We will only 
 use the superdifferential. Thus, for any cost $c,$ we will denote by $\partial^{c}\varphi$ the superdifferential
 of any $c$-concave function $\varphi$.
 
The following Proposition characterizes solutions of Problem 1.1 satisfying
some hypotheses. We remark that Proposition \ref{trans-mass} provides conditions under which  these hypotheses are satisfied.
 
\begin{proposition}\label{phi} \textbf{(Characterization of optimal
pairs) } Let $\mu$ and $\nu$ be 
absolutely continuous measures in $\mathcal{M}(\Omega)$.
Also, let $(\gamma,h)$ be in $ADM(\mu,\nu).$ 
Assume that  $\mu$ and $\nu+\tau h$ are strictly positive. Then, the following are
equivalent:
\begin{enumerate}[(i)]
\item $C_{\tau}(\gamma,h)$ is minimal among all pairs in $ADM\big(\mu,\nu\big)$
with $h$ fixed. 
\item $\gamma$ is concentrated on $\mathbb{\mathcal{{A}}}$ and
$supp(\gamma)\cup\partial\Omega\times\partial\Omega$
is $\tilde{c}$-cyclically monotone. 
\item There exist functions $\varphi$, $\varphi^{*}:\overline{\Omega}\rightarrow\mathbb{{R}}$
having the following properties: 
\begin{enumerate}
\item $\varphi_{|\Omega}$ is $c_{1}-concave$, $\varphi_{|\Omega}=(\varphi^{*})^{c_{1}}$,
$\varphi_{|\Omega}^{*}$ is $c_{2}-concave,$ and $\varphi_{|{\Omega}}^{*}=\varphi^{c_{2}}$
. 
\item $supp(\gamma_{\Omega}^{\overline{\Omega}})\subset\partial^{c_{1}}\varphi$
and $supp(\gamma_{\overline{\Omega}}^{\Omega})\subset\partial^{c_{2}}\varphi^{*}$. 
\item $\varphi_{|\partial\Omega}=\Psi$ and $\varphi_{|\partial\Omega}^{*}=-\Psi$. 
\end{enumerate}
\end{enumerate}
Moreover, $(\gamma,h)$ is optimal in $ADM\big(\mu,\nu\big)$ if and only if $\varphi_{|\Omega}^{*}=-e^{\prime}\circ h+\kappa,$ $\mathcal{L}^{d}$ a.e.,
for some constant $\kappa.$
 \end{proposition} \begin{proof} 
We start by proving that $(i)\implies (ii).$
Define the plan $\tilde{\gamma}$ by
\[
\tilde{\gamma}:=\gamma_{|\mathcal{A}}+(\pi^{1},
P_{\Psi,\tau}\circ\pi^{1})_{\#}\bigg(\gamma_{|\overline
{\Omega}\times\overline{\Omega}\backslash\mathcal{A}}\bigg)+
(P_{-\Psi,\tau}\circ\pi^{2},\pi^{2})_{\#}\bigg
(\gamma_{|\overline{\Omega}\times\overline{\Omega}
\backslash\mathcal{A}}\bigg).
\]
Observe that $\tilde{\gamma}\in ADM(\mu,\nu)$ and
\begin{align*}
C_{\tau}(\tilde{\gamma},h)&=\int_{\overline{\Omega}\times\overline{\Omega}}\bigg(\frac{| x-y|^{2}}{2\tau}1_{(\partial\Omega\times\partial\Omega)^{c}}+\Psi(y)1_{\Omega\times\partial\Omega}-\Psi(x)1_{\partial\Omega\times\Omega}\bigg)\hspace{1mm}d\tilde{\gamma}\\
&=\int_{\mathcal{A}}\tilde{c}\hspace{1mm}d\tilde{\gamma}+
\int_{\overline{\Omega}\times\overline{\Omega}\backslash\mathcal{A}}\bigg(d_{\Psi,\tau}(x)+d_{-\Psi,\tau}(y)\bigg)\hspace{1mm}d\gamma\\
&\leq\int_{\overline{\Omega}\times\overline{\Omega}}\bigg(\frac{| x-y|^{2}}{2\tau}1_{(\partial\Omega\times\partial\Omega)^{c}}+\Psi(y)1_{\Omega\times\partial\Omega}-\Psi(x)1_{\partial\Omega\times\Omega}\bigg)\hspace{1mm}d\gamma\\
&=C_{\tau}(\gamma,h),
\end{align*}
with strict inequality if $\gamma(\overline{\Omega}\times\overline{\Omega}\backslash\mathcal{A})>0.$
Thus, from the optimality of $\gamma$, we deduce that it is
concentrated on $\mathcal{A}.$

Now we have to prove the $\tilde{c}$-cyclical 
monotonicity of $supp(\gamma)\cup\partial\Omega\times
\partial\Omega.$ Note that
\[
 C_{\tau}(\gamma,h)=C_{\tau}\big(\gamma+\mathcal{H}^{d-1}_{|\partial\Omega}\otimes\mathcal{H}^{d-1}_{|\partial\Omega},h\big).
\]
Hence, we can assume without loss of generality that $\partial\Omega\times\partial\Omega\subset supp(\gamma).$
Let $\{(x_{i},y_{i})\}_{i=1}^{n}\in supp(\gamma)$. Our
objective is to show  that
\[
\sum_{i}\tilde{c}(x_{i},y_{\sigma(i)})-\tilde{c}(x_{i},y_{i})\geq0,
\hspace{1em} {\rm {for \hspace{1mm} all  \hspace{1mm} permutations
 \hspace{1mm} }} \sigma  {\rm {\hspace{1mm} of  \hspace{1mm} }} \{1,...,n\}. \]
We proceed by contradiction. For this purpose, we  assume that  the above inequality fails for some permutation
$\sigma$. Let 
\[X_{i}=\begin{cases}
\partial\Omega\hspace{1em}\mbox{if}\hspace{1mm}  x_{i}\in\partial\Omega, &\\ 
\Omega \hspace{1em} \mbox{otherwise,}\end{cases}\]
and
\[Y_{i}=\begin{cases}
\partial\Omega\hspace{1em}\mbox{if}\hspace{1mm}  y_{i}\in\partial\Omega, &\\ 
\Omega \hspace{1em} \mbox{otherwise.}\end{cases}\]
The cost $\tilde{c}$ is continuous in $X_{i}\times Y_{i},$ for any $i$
in $\{0,...,n\}.$ Hence, we can find neighborhoods $U_{i}\subset X_{i}$ and $V_{i}\subset Y_{i}$
of $x_{i}$ and $y_{i}$ such that
\[
\sum_{i=1}^{N}\tilde{c}(u_{i},v_{\sigma(i)})-\tilde{c}(u_{i},v_{i})<0\hspace{1em}\forall(u_{i},v_{i})
\in U_{i}\times V_{i}\hspace{1em} {\rm{and}}\hspace{1em}\forall i\in \{0,...,n\}.
\]
We will build a variation of
$\gamma$, $\tilde{\gamma}=\gamma+\eta,$  in such a way that its minimality is violated.
To this aim, we need a signed measure $\eta$ with:

(A) $\eta^{-}\leq\gamma$ (so that $\tilde{\gamma}$ is non-negative);

(B) $\pi_{\#}^{1}\eta{}_{|\Omega}=\pi_{\#}^{2}\eta{}_{|\Omega}=0$ 
(so that $(\tilde{\gamma},h)$ is admissible);

(C) $\int_{\overline{\Omega}\times\overline{\Omega}}\tilde{c}(x,y)\hspace{1mm}d\eta<0$ (so that $\gamma$ is not optimal).

Let $\mathcal{C}=\Pi_{i=1}^{N}U_{i}\times V_{i}$ and $P\in\mathbb{P}(\mathcal{C})$
be defined as the product of the measures $\frac{1}{m_{i}}\gamma_{| U_{i}\times V_{i}}.$
Here, $m_{i}:=\gamma(U_{i}\times V_{i}).$ Denote by $\pi^{U_{i}}$ and $\pi^{V_{i}}$
the natural projections of $\mathcal{\mathcal{C}}$ to $U_{i}$ and
$V_{i}$ respectively. Also, define 
\[
\eta:=\frac{\min_{i}m_{i}}{N}\sum_{i=1}^{N}(\pi^{U_{i}},\pi^{V_{\sigma(i)}}\big)_{\#}P-(\pi^{U_{i}},\pi^{V_{i}}\big)_{\#}P.
\]
Since $\eta$ satisfies $(A),(B)$, and $(C)$, the $\tilde{c}-$cyclical monotonicity
is proven.\\
Next, we prove that $(ii)\implies (iii).$ \\
 Arguing as Step 2 of \cite[Theorem 1.13]{users-guide}, we can produce
a $\tilde{c}-concave$ function $\tilde{\varphi}$ such that
$supp(\gamma)\cup\partial\Omega\times\partial\Omega\subset\partial^{\tilde{c}}\tilde{\varphi}$.
Then,
\begin{multline} \label{tilde-dif}\tilde{\varphi}(x)+\tilde{\varphi}^{\tilde{c}}(y)  =\frac{| x-y|^{2}}{2\tau}1_{(\partial\Omega\times\partial\Omega)^{c}}-\Psi(x)1_{\partial\Omega\times\Omega}+\Psi(y)1_{\Omega\times\partial\Omega}\hspace{1.5mm}\forall(x,y)\in supp(\gamma)\cup\partial\Omega\times\partial\Omega\\
\hspace{5em} {\rm{and}} \hspace{1em} \tilde{\varphi}(x)+\tilde{\varphi}^{\tilde{c}}(y)  \leq\frac{| x-y|^{2}}{2\tau}1_{(\partial\Omega\times\partial\Omega)^{c}}-\Psi(x)1_{\partial\Omega\times\Omega}+\Psi(y)1_{\Omega\times\partial\Omega}\hspace{1.5mm}\forall(x,y)\in\overline{\Omega}\times\overline{\Omega}.
\end{multline}
After adding a constant, we can assume $\tilde{\varphi}^{c}(y_{0})=0$
for some $y_{0}\in\partial\Omega$. Then, using \eqref{tilde-dif}
it is easy to show that $\tilde{\varphi}=0$ on $\partial\Omega$.
Consequently, $\tilde{\varphi}^{c}=0$ on $\partial\Omega$
as well.

Set $\varphi=\tilde{\varphi}+1_{\partial\Omega}\Psi$ and $\varphi^{*}=\tilde{\varphi}^{\tilde{c}}-\Psi1_{\partial\Omega}$. Since the measure $\mu$ is strictly positive, by \eqref{tilde-dif}  we have
\[
\inf_{y\in\overline{\Omega}}c(x,y)-\varphi^{*}(y)=\varphi(x)
\hspace{1em}\forall x\in\Omega.
\]
Similarly, since $\pi_{2\#}\gamma$ is strictly positive, we have 
\[
\inf_{x\in\overline{\Omega}}c(x,y)-\varphi(x)=\varphi^{*}(y)\hspace{1em}\forall y\in\Omega.
\]
Then, all the items in $(iii)$ can be verified using \eqref{tilde-dif} (see \cite[Definitions 1.7 to 1.10]{users-guide}).\\
We proceed to prove that $(iii)\implies (i).$\\
 Let $(\tilde{\gamma},h)$ be any admissible pair. 
 We set $\tilde{\varphi}=\varphi-\Psi1_{\partial\Omega}$
and $\tilde{\varphi}^{*}=\varphi^{*}+\Psi1_{\partial\Omega}.$ 
By item $(b)$ of $(iii)$, we have that
\eqref{tilde-dif} holds with $\tilde{\varphi}^{*}$ in place of $\tilde{\varphi}^{c}$.
Moreover, from $(c)$ we get $\tilde{\varphi}_{|\partial\Omega}=
\tilde{\varphi}_{|\partial\Omega}^{*}=0.$
From $(a),$ $(b),$ and $(B2)$, we obtain  that $\tilde{\varphi}_{|\Omega}$ and $\tilde{\varphi}_{|\Omega}^{*}$
are Lipschitz. Thus, they are integrable against any measure with finite
mass. As a consequence of these observations, we deduce
\begin{align*}\begin{aligned}C_{\tau}(\gamma,h) & =\int_{\overline{\Omega}\times\overline{\Omega}}\tilde{c}\hspace{1mm}d\gamma+\tau\int_{\Omega}e(h)\hspace{1mm}dx\\
 & =\int_{\overline{\Omega}\times\overline{\Omega}}\bigg(\tilde{\varphi}(x)+\tilde{\varphi}^{*}(y)\bigg)\hspace{1mm}d\gamma+\tau\int_{\Omega}e(h)\hspace{1mm}dx\\
 & =\int_{\Omega}\tilde{\varphi}(x)\hspace{1mm}d\mu+\int_{\Omega}\tilde{\varphi}^{*}(y)\hspace{1mm}d\nu+\tau\int_{\Omega}\tilde{\varphi}^{*}(y)\hspace{1mm}dh+\int_{\Omega}e(h)\hspace{1mm}dx\\
 & =\int_{\overline{\Omega}\times\overline{\Omega}}\bigg(\tilde{\varphi}(x)+\tilde{\varphi}^{*}(y)\bigg)\hspace{1mm}d\tilde{\gamma}+\tau\int_{\Omega}e(h)\hspace{1mm}dx\\
 & \leq\int_{\overline{\Omega}\times\overline{\Omega}}\tilde{c}\hspace{1mm}d\tilde{\gamma}+\tau\int_{\Omega}e(h)\hspace{1mm}dx\\
 & = C_{\tau}(\tilde{\gamma},h).
\end{aligned}
\end{align*}
In the third and fourth line above, we have used \eqref{interior-mass}. This gives us the desired implication.\\
To prove the last part of the Proposition, we suppose the pair $(\gamma,h)$
is optimal. Also, we claim that there exists a set $L\subset\overline{\Omega}$ of zero Lebesgue
measure such that for every $x$ in $\overline{\Omega}\backslash L$
there exists $y\in\overline{\Omega}\backslash L$ such that $(x,y)\in supp(\gamma)\cup\partial\Omega\times\partial\Omega$
and 
\begin{multline}
e^{\prime}\circ h(\tilde{y})1_{\Omega}(\tilde{y})+\frac{| x-\tilde{y}|^{2}}{2\tau}1_{(\partial\Omega\times\partial\Omega)^{c}}(x,\tilde{y})+\Psi(\tilde{y})1_{\Omega\times\partial\Omega}(x,\tilde{y})-\Psi(x)1_{\partial\Omega\times\Omega}(x,\tilde{y})\\
   \hspace{4em} \geq\label{subdifrencial}e^{\prime}\circ h(y)1_{\Omega}(y)+\frac{| x-y|^{2}}{2\tau}1_{(\partial\Omega\times\partial\Omega)^{c}}(x,y)+\Psi(y)1_{\Omega\times\partial\Omega}(x,y)-\Psi(x)1_{\partial\Omega\times\Omega}(x,y),
\end{multline}
holds for every $\tilde{y}$ in $\overline{\Omega}\backslash L$. We
also claim that this set $L$ can be taken such that for every $y$
in $\overline{\Omega}\backslash L$ there exists $x$ in $\overline{\Omega}\backslash L$
so that $(x,y)\in supp(\gamma)\cup\partial\Omega\times\partial\Omega$ and the above inequality holds for
almost every $\tilde{y}$ in $\overline{\Omega}\backslash L$. We
will show these claims at the end of the proof. Now, we show how the
result follows from them. Define the function
\[
(-e^{\prime}\circ h)^{\tilde{c}}(x)=
\inf_{y\in\overline{\Omega}\backslash L}\frac{| x-y|^{2}}{2\tau}1_{(\partial\Omega\times\partial\Omega)^{c}}+\Psi(y)1_{\Omega\times\partial\Omega}(x,y)-\Psi(x)1_{\partial\Omega\times\Omega}(x,y)+e^{\prime}\circ h(y)1_{\Omega}(y).
\]
for every $x$ in $\overline{\Omega}\backslash L.$
By \eqref{subdifrencial} for every $x$
in $\overline{\Omega}\backslash L$ there exists $y$ in $\overline{\Omega}\backslash L$
such that $(x,y)\in supp(\gamma)\cup\partial\Omega,$
\begin{multline*}
(-e^{\prime}\circ h(y)1_{\Omega}(y))+(-e^{\prime}\circ h)^{\tilde{c}}(x)=\Psi(y)1_{\Omega\times\partial\Omega}(x,y)-\Psi(x)1_{\partial\Omega\times\Omega}(x,y)+\frac{|x-y|^{2}}{2\tau}1_{(\partial\Omega\times\partial\Omega)^{c}},\\
{\rm{and}}\hspace{1em}
(-e^{\prime}\circ h(y)1_{\Omega}(y))+(-e^{\prime}\circ h)^{\tilde{c}}(\tilde{x})\leq\Psi(y)1_{\Omega\times\partial\Omega}(\tilde{x},y)-\Psi(\tilde{x})1_{\partial\Omega\times\Omega}(\tilde{x},y)\\
+\frac{|x-y|^{2}}{2\tau}1_{(\partial\Omega\times\partial\Omega)^{c}},
\end{multline*}
for almost every $\tilde{x}$ in $\overline{\Omega}\backslash L.$
Then, we have that
\[
(-e^{\prime}\circ h(y)1_{\Omega}(y)=\inf_{x\in\overline{\Omega}\backslash L}\Psi(y)1_{\Omega\times\partial\Omega}(x,y)-\Psi(x)1_{\partial\Omega\times\Omega}(x,y)+\frac{|x-y|^{2}}{2\tau}1_{(\partial\Omega\times\partial\Omega)^{c}}-(-e^{\prime}\circ h)^{\tilde{c}}(x).
\]Thus, it follows that the functions $-e^{\prime}\circ h_{|\Omega\backslash L}$
admits a Lipschitz extension to $\Omega$, which we will not relabel.
Consequently, for every $y$ in $\Omega\backslash L$ there exists $x\in\overline{\Omega}\backslash L$
such that $(x,y)\in supp(\gamma)$
and \eqref{subdifrencial} holds for every
$\tilde{y}\in\Omega.$ Then  by \eqref{subdifrencial}, for every $y$ in $\Omega\backslash L$
there exists $x\in\overline{\Omega}$ and a constant $A:=A(x,y)$ such
that $(x,y)$ is in $supp(\gamma)$ and 
\begin{equation}\label{one}
\tau e^{\prime}\circ h(y)+\frac{|y|^{2}}{2}+\langle x,\tilde{y}-y\rangle+A(x,y)\leq\frac{|\tilde{y}|^{2}}{2}+\tau e^{\prime}\circ\tilde{h}(\tilde{y}),
\end{equation}
for every $\tilde{y}\in\Omega.$ Let $\mathcal{P}$ be the set of
affine functions that are below $\tau e^{\prime}\circ h(y)+\frac{|y|^{2}}{2}$
in $\Omega.$ 
Then, it follows that
\[
\tau e^{\prime}\circ h+\frac{|y|^{2}}{2}=\sup_{p\in\mathcal{P}}p(y),
\]
for every $y$ in $\Omega\backslash L.$ This together with the Lipschitz
continuity of $-e^{\prime}\circ h_{|\Omega}$ implies that the
function $\tau e^{\prime}\circ h(y)+\frac{|y|^{2}}{2}$ is convex.
In a similar way from \eqref{tilde-dif} we can deduce
that $\varphi_{|\Omega}^{*}$ is Lipschitz, $\frac{|y|^{2}}{2}-\tau\varphi^{*}(y)$ is convex, and for a.e $y$
in $\Omega$ there exists a point $x\in\overline{\Omega}$ and a constant $B:=B(x,y)$
such that $(x,y)\in supp(\gamma)$ and 
\begin{equation}\label{two}
-\tau\varphi^{*}(y)+\frac{|y|^{2}}{2}+\langle x,\tilde{y}-y\rangle+B(x,y)\leq\frac{|\tilde{y}|^{2}}{2}+\tau\varphi^{*}(\tilde{y}),
\end{equation}
for every $\tilde{y}\in\Omega.$
Recall $\nu+\tau h$ is absolutely continuous and uniformly
bounded from below. Consequently, by Lemma \ref{cre-mass}
$\gamma_{\overline{\Omega}}^{\Omega}=(S,Id)_{\#}\nu+\tau h,$
for a map $S$ that is optimal in the classical sense and is uniquely
defined a.e. Thus, it follows from \eqref{one} and \eqref{two} that $\frac{|y|^{2}}{2}-\tau\varphi^{*}(y)$
and $\tau e^{\prime}\circ h(y)+\frac{|y|^{2}}{2}$ are Lipschitz,
and have the same derivative
a.e in $\Omega$. Therefore, we deduce that there exists a
constant $\kappa$ such that
\[
\varphi^{*}=-e^{\prime}\circ h+\kappa\hspace{1em}a.e\hspace{1em}in\hspace{1em}\Omega.
\]
In order to prove the opposite implication, suppose $\tilde{\varphi}^{*}=-e^{\prime}\circ h+\kappa$
and let $(\tilde{\gamma},\tilde{h})\in ADM(\mu,\nu)$. When we argue
as in $(iii)\implies (i),$ we obtain 
\begin{align*}
  C_{\tau}(\gamma,h) & =\int_{\overline{\Omega}\times\overline{\Omega}}\tilde{c}\hspace{1mm}d\gamma+\tau\int_{\Omega}e(h)\hspace{1mm}dx\\
 & =\int_{\overline{\Omega}\times\overline{\Omega}}\bigg([\tilde{\varphi}(x)+\kappa]+[\tilde{\varphi}^{*}(y)-\kappa]\bigg)\hspace{1mm}d\gamma+\tau\int_{\Omega}e(h)\hspace{1mm}dx\\
 & =\int_{\Omega}[\tilde{\varphi}(x)+\kappa]\hspace{1mm}d\mu+\int_{\Omega}[\tilde{\varphi}^{*}(y)-\kappa]\hspace{1mm}d\nu+\tau\int_{\Omega}[\tilde{\varphi}^{*}(y)-\kappa]h\hspace{1mm}dx+\tau\int_{\Omega}e(h)\hspace{1mm}dx\\
 & =\int_{\overline{\Omega}\times\overline{\Omega}}\bigg(\tilde{\varphi}(x)+\tilde{\varphi}^{*}(y)\bigg)\hspace{1mm}d\tilde{\gamma}+\tau\int_{\Omega}[\tilde{\varphi}^{*}(y)-\kappa](h-\tilde{h})\hspace{1mm}dx+\tau\int_{\Omega}e(h)\hspace{1mm}dx\\
 &\leq\int_{\overline{\Omega}\times\overline{\Omega}}\tilde{c}\hspace{1mm}d\tilde{\gamma}+\tau \int_{\Omega}e(h)\hspace{1mm}dx+\tau\int_{\Omega}e^{\prime}\circ h(\tilde{h}-h)\hspace{1mm}dx\\
 &\leq\int_{\overline{\Omega}\times\overline{\Omega}}\tilde{c}\hspace{1mm}d\tilde{\gamma}+\tau\int_{\Omega}e(\tilde{h})\hspace{1mm}dx.
\end{align*} 
Here, in the last inequality we used (C2). 
This completes the proof of the Theorem, provided we can prove the claim.

Finally, we show \eqref{subdifrencial}. The idea is to use 
Proposition \ref{perturba}
and the absolute continuity and uniform positivity of $\mu$ and
$\nu+\tau h.$ We only prove the statement holds for $x\in\overline{\Omega}\backslash L$; the corresponding statement for $y$ is analogous. In order
to do this we will use the same notation as in 
Proposition \ref{perturba}. 

Let $L_{1}$
be a set of zero Lebesgue measure such that every point in $\Omega\backslash L_{1}$
is a Lebesgue point for $S,$ $\nu+\tau h,$ and $h.$ Also let $L_{2}$
be a set of zero Lebesgue measure such that every point in 
$\Omega\backslash L_{2}$
is a Lebesgue point for $T$ and the density of $\mu.$ 
Let $A=\{y\in\Omega\backslash L_{1}\hspace{1mm}:\hspace{1mm}S(y)\in\Omega\}$
and $B=\{x\in\Omega\backslash L_{2}\hspace{1mm}:\hspace{1mm}T(x)\in\partial\Omega\}.$
Since $\pi_{1\#}(\gamma_{\Omega}^{\Omega}+\gamma_{\Omega}^{\partial\Omega})=\mu$
and $\nu+\tau h$ and $\mu$ are absolutely continuous and uniformly
positive, it follows that $L_{3}=\overline{\Omega}\backslash (S(A)\cup T(B))$ has
zero Lebesgue measure. Set $L=L_{1}\cup L_{2}\cup L_{3}$.
Then, for every $x\in\Omega\backslash L$ we have two possibilities:
Either there exists $y\in\Omega\backslash L$ such that $x=S(y),$
in which cases the claim follows from \eqref{eq:-14} and \eqref{eq:-16},
or $T(x)\in\partial\Omega,$
in which case the claim follows from
\eqref{eq:-25} and \eqref{eq:1-5-1}. It remains to consider
the case when $x\in\partial\Omega \backslash L$. In such case 
the statement follows from \eqref{eq:-16-1} and \eqref{37}. This concludes
the proof of the Proposition.
\end{proof}
 The following result is the analogue in our setting
of Brenier's Theorem on existence and uniqueness of optimal transport
maps. \begin{corollary}\textbf{ (On uniqueness of optimal pairs)}
Let $\mu,\nu\in \mathcal{M}(\Omega)$ and fix $(\gamma,h)\in Opt(\mu,\nu)$ 
satisfying the hypotheses of the previous Proposition.  Additionally, let $\varphi$
and $\varphi^{*}$ be the functions given by Proposition \ref{phi}. Then
\begin{enumerate}[(i)]
\item The function $h$ is unique $\mathcal{L}^{d}$ a.e.
\item The plan $\gamma_{\Omega}^{\overline{\Omega}}$
is unique and it is given by $(Id,T)_{\#}\mu.$ Also, $T\hspace{1mm}:\hspace{1mm}\Omega\rightarrow\overline{\Omega}$
is the gradient of a convex function and
\[
 -\nabla \varphi = \frac{T-Id}{\tau}\hspace{1em}{\rm{a.e.}}
 \hspace{3mm}{\rm{in}} \hspace{1em}\Omega.
\]
\item The plan $\gamma_{\overline{\Omega}}^{\Omega}$
is unique and it is given by $(S,Id)_{\#}\nu.$ Also, $S\hspace{1mm}:\hspace{1mm}\Omega\rightarrow\overline{\Omega}$ is
the gradient of a convex function and 
\[
 -\nabla \varphi^{*} = \frac{S-Id}{\tau}\hspace{1em}
 {\rm{a.e.}}
 \hspace{3mm}{\rm{in}} \hspace{1em}\Omega.
\]
\item If $\gamma$ has no mass concentrated on 
$\partial\Omega\times\partial\Omega,$ then 
$\gamma$ is unique. 
\end{enumerate}
\label{uniqueness} \begin{proof} By linearity of the constraint \eqref{interior-mass} in $ADM(\mu,\nu)$, the uniqueness of $h$ follows by (C2). Due to the equivalence $(i)\iff(iii)$
of the previous Theorem, using $(a)$ and $(b)$ we get that the functions
$\tau\varphi$ and $\tau\varphi^{*}$ are $\frac{d^{2}}{2}-concave$. Here, $\hspace{1mm}d(x,y)=| x -y|.$
Thus, the result follows exactly as in the classical transportation
problem (see for example \cite[Theorem 6.2.4 and Remark 6.2.11]{Ambrosio-Gigli-Savare}).
\end{proof} \end{corollary} 
Henceforth we will assume, without loss of generality, that the transportation plans $\gamma$ 
have no mass concentrated on $\partial\Omega\times
\partial\Omega$. 
\section{The weak solution }
In this section we follow the minimizing movement scheme described in the introduction.
This method yields
a map, $t\rightarrow\rho(t),$ that belongs to $L_{loc}^{2}([0,\infty),W^{1,2}(\Omega))$.
Such a map is a weak solution to (\ref{T2}). By this, we mean
that the
map $t\rightarrow\rho(t)-e^{\Psi-V}$ belongs to $L_{loc}^{2}([0,\infty),W_{0}^{1,2}(\Omega))$, 
\[
 \rho(0)=\rho_{0}\hspace{1em}{\rm{in}}\hspace{1em}\Omega,
\]
and
\[
\int_{\Omega}\zeta \rho(s)\hspace{1mm} \hspace{1mm}dx-\int_{\Omega}\zeta \rho(t)\hspace{1mm}
\hspace{1mm}dx=\int_{t}^{s}\bigg(\int_{\Omega}\big[\Delta\zeta-
\langle \nabla V , \nabla\zeta \rangle \big]\rho(r)\hspace{1mm}dx-\int_{\Omega}\zeta[e_{x}^{\prime}]^{-1}\big(\log(\rho(r))+V\big)\hspace{1mm}dx\bigg)dr,\hspace{2mm}
\]
for all $0\leq t<s$ 
and $\zeta$ in $C_{c}^{\infty}(\Omega)$. 
\\
Similarly, we will say that a map $t\rightarrow\rho(t)$ in $L_{loc}^{2}([0,\infty),W^{1,2}(\Omega))$
is a weak solution of \eqref{Dirichlet}, if there exists a Lipschitz function $\tilde{\rho}$ such that
$t\rightarrow\rho(t)-\tilde{\rho}$ belongs to $L_{loc}^{2}([0,\infty),W_{0}^{1,2}(\Omega)),$
\[
\tilde{\rho} =\rho_{D}\hspace{1em}{\rm{on}}\hspace{1em}\partial\Omega,  
\]
\[
 \rho(0)=\rho_{0}\hspace{1em}{\rm{in}}\hspace{1em}\Omega,
\]
and
\[
\int_{\Omega}\zeta \rho(s)\hspace{1mm} \hspace{1mm}dx-\int_{\Omega}\zeta \rho(t)\hspace{1mm}
\hspace{1mm}dx=\int_{t}^{s}\bigg(\int_{\Omega}\big[\Delta\zeta-
\langle \nabla V , \nabla\zeta \rangle \big]\rho(r)\hspace{1mm}dx-\int_{\Omega}\zeta F_{x}^{\prime}(\rho(r))\hspace{1mm}dx\bigg)dr,\hspace{2mm}
\]
for all $0\leq t<s$ and $\zeta$ in $ C_{c}^{\infty}(\Omega)$.
\[
E(\mu):=\begin{cases}
 & \int_{\Omega}\mathcal{E}(\rho(x),x)\hspace{1mm}dx\hspace{1em}if\hspace{1em}\mu=\rho \hspace{1mm}
 \mathcal{L}_{|\Omega,}^{d}\\
 & +\infty\hspace{1em}otherwise,
\end{cases}
\]
where $\mathcal{E}:[0,\infty)\times{\Omega}\rightarrow[0,\infty)$ is given by
\[
\mathcal{E}(z,x):=z\log z-z+V(x)z+1.
\]
We will denote by $\mathcal{E^\prime}$ the derivative of $\mathcal{E}$ with respect
to its first variable and by $D(\mathcal{E})$ the interior of the sets of 
points where $\mathcal{E}$ is finite. The notations $\mathcal{E}(\rho(x),x)$
and  $\mathcal{E}(\rho)$ will be used interchangeably. 
Also, we will freely interchange $\mathcal{E^{\prime}}(\rho(x),x)$ and
$\mathcal{E^{\prime}}(\rho).$ 

The main result is the following:
\begin{theorem} With the notation from the introduction and assumptions 
(B1) and (B2), for any pair of functions $e$ and $\Psi$ satisfying (C1)-(C9), any uniformly positive and bounded initial data $\rho_{0},$ 
and any sequence $\tau_{k}\downarrow0$ there exists a subsequence, not relabeled, such that $\rho^{\tau_{k}}(t)$
converges to $\rho(t)$ in $L^{2}(0;t_{f},L_{loc}^{2}(\Omega)),$ for
any $t_{f}>0.$                            
The map $t\rightarrow\rho(t)$ belongs to $L_{loc}^{2}([0,\infty),W^{1,2}(\Omega))$
and is a weak solution of \eqref{T2}. Moreover, there exist positive constants
$\lambda$ and $\Lambda$ such that
\begin{equation}\label{Z3}
\frac{\lambda}{\sup\{e^{-V}\}}e^{-(C_{0}t+V)}\leq \rho(x,t)\leq \frac{\Lambda}{\inf\{e^{-V}\}}e^{-V},
\end{equation}
for almost every  $x$. 
\label{equation} \end{theorem}
\begin{remark}
When assumptions (B1)-(B3) and (F1)-(F7) hold, and  $e$ and $\Psi$
are as in \eqref{Z1} and \eqref{Z2}, properties (C1)-(C9) hold as well and the map $t\rightarrow\rho(t)$
given by the previous Theorem is a weak solution of \eqref{Dirichlet}.
\end{remark}
The proof of Theorem \ref{equation} is involved. We begin with a technical result.
\begin{proposition}\textbf{ (A step of the minimizing movement)}
Let $\mu$ be a measure in $\mathcal{{M}}(\Omega)$ with the property
that $E(\mu)<\infty$. Also,
assume that its density is uniformly positive and bounded. Additionally, let
$\tau$ be a positive number.
Then, there exists a minimum $\mu_{\tau}\in\mathcal{{M}}(\Omega)$
of 
\begin{equation}
\rho\rightarrow E(\rho)+Wb_{2}^{e,\Psi,\tau}(\mu,\rho).\label{minmov}
\end{equation}
Moreover, there exists $\delta>0$ such that 
if $\tau<\delta,$ then the corresponding optimal pair 
$(\gamma,h)\in ADM(\mu,\mu_{\tau})$ satisfies:
\begin{enumerate}[(i)]
\item $\mu_{\tau}=\rho_{\tau}\hspace{1mm}
\mathcal{L}_{|\Omega}^{d}.$
\item $e^\prime\circ h=\log\rho_{\tau}+ V.$
\item The restriction of $\gamma$ to $\overline{\Omega}\times\Omega$  
is given by $(T,Id)_{\#}\mu_{\tau}.$ The map $T$ satisfies
\begin{equation}
\frac{T(y)-y}{\tau}=\nabla\log\rho_{\tau}(y)+\nabla V(y),\hspace{1em}\mathcal{L}^{d}-a.e.x.\label{eq:-9}
\end{equation}
\item $\rho_{\tau}\in W^{1,2}(\Omega)$ and $||Tr\hspace{1mm}[\rho_{\tau}]-e^{\Psi-V}||_{L^{\infty}(\partial\Omega)}\leq C\sqrt{\tau}$. 

\end{enumerate}
Here, $C$ is a positive constant that depends only on $\Psi$. Also,
$Tr\hspace{1mm}:\hspace{1mm}W^{1,2}(\Omega)\rightarrow L^{2}(\partial\Omega)$
denotes the trace operator. \label{step} \begin{proof} 
Consider a minimizing sequence of measures 
$\{ \rho^{n} \}_{n=1}^{\infty},$ with corresponding optimal pairs 
$\{(\gamma^{n},h^{n})\}_{n=1}^{\infty}$ in $ADM(\mu, \rho^{n}).$
We claim such sequences of measures and optimal pairs have the property that the  mass the elements of 
$\{(\rho^{n},\gamma^{n})\}_{n=1}^{\infty}$ and the norm in $L^{1}(\Omega)$ of the members of $\{h^{n}\}_{n=1}^{\infty}$
are uniformly  bounded. Since $\Omega$ is bounded, the claim allows
us to obtain compactness and produce subsequences
weakly converging to $\gamma,$ $h,$ and $\rho_{\tau}$ .
The previous convergence takes place as described 
in the proof of Lemma \ref{exists}. We will not relabel these subsequences. The absolute
continuity of $h$
and $\rho_{\tau}$ is guaranteed by
the superlinearity of $e$ and $\mathcal{E}$.\\
The inequality
\[
\liminf_{n\rightarrow \infty}E(\rho^{n})\geq E(\rho_{\tau}),
\]
is a consequence of the weak conergence, 
$\rho_{n} \rightharpoonup \rho,$  and
the convexity  and superlinearity of  the maps 
$\{ r \rightarrow \mathcal{E} (r,x) \}_{x\in\overline{\Omega}}$ 
(See \cite[Lemma 9.4.5]{Ambrosio-Gigli-Savare}, for example). 
To show
\[
\liminf_{n\rightarrow \infty }C_{\tau}(\gamma^{n},h^{n})
\geq C_{\tau}(\gamma,h),
\]
and $(\gamma,h)\in ADM(\mu,\rho_{\tau}),$ we argue as in Lemma \ref{exists}. This  gives us the existence of a minimum as well as item $(i),$ assuming we can prove the
claim.\\
 Next, we show the claim. Arguing as in Lemma \ref{exists} and using Jensen inequality
 we obtain
\begin{multline*}
\int_{\Omega}\mathcal{E}(\rho)\hspace{1mm}dx+C_{\tau}(\gamma,h) \geq-||\Psi||_{\infty}
\bigg(\mu(\Omega)+\nu(\Omega)+\tau |h|(\Omega)\bigg)
\\ 
\hspace{1mm}+K| h|(\Omega)+(C(K)-1)|\Omega|+
\rho(\Omega) \log \bigg(\frac{\rho(\Omega)}{|\Omega|}\bigg)
-(1+||V||_{\infty})\rho(\Omega) .\\
\end{multline*}
Taking $K$ large enough, we obtain a uniform bound
on $\rho(\Omega)+\tau |h|(\Omega)$ and consequently on $|\gamma|,$ for any minimizing
sequence. (Recall we assume 
that the plans have no mass concentrated on
$\partial\Omega\times\partial\Omega$). This proves the claim.

We proceed to  the proof of $(ii)$. 
Let $\eta$ be a function with compact support in $\Omega.$
For each $\varepsilon>0,$ let 
$\rho_{\tau}^{\varepsilon}=\rho_{\tau}-\tau\varepsilon\eta.$ 
By Lemma \ref{uniform}, for sufficiently small $\varepsilon$ we can guarantee that $\rho_{\tau}^{\varepsilon}$
is non-negative. Since
$(\gamma, h+\varepsilon \eta)\in
ADM(\mu,\rho_{\tau}^{\varepsilon}),$
by minimality must have
\[
E(\rho_{\tau}^{\varepsilon})-E(\rho_{\tau})+
C_{\tau}(\gamma,h+\varepsilon \eta)-C_{\tau}(\gamma,h)\geq0.
\]
Dividing by $\varepsilon$ and letting $\varepsilon \downarrow0$, due
to \eqref{interior-mass}, Lemma \ref{uniform}, Lemma \ref{cre-mass}, the dominated convegence Theorem and
the fact that $e$ and $\mathcal{E}$ are locally Lipschitz in $D(e)$ and $D(\mathcal{E}),$
we get
\[
 \int_{\Omega} \big( e^\prime \circ h \big) \eta \hspace{1mm} \hspace{1mm}dy  
 - \int_{\Omega} \big( \log \rho_{\tau}+V \big) \eta \hspace{1mm} \hspace{1mm}dy
 \geq 0.  \]
Replacing $\eta$ by $-\eta$ gives the desired result.

Now, we show $(iii)$.  
Let $\lambda$ and $\Lambda$ be  positive numbers 
such that Proposition \ref{bar} holds. Then, for $\tau\in(0,1)$ we have that 
$\rho,\rho_{\tau}>\lambda\hspace{0.5mm}(\inf e^{-V}/\sup e^{-V})/(1+C_{0}).$ We let $\delta\in(0,1)$ have the property that
Corollary \ref{comparable} and Proposition \ref{trans-mass} hold for any $\tau\in(0,\delta).$  Now, observe
that Corollary $\ref{uniqueness}$ and the absolute continuity of
$\mu_{\tau}$ guarantee the existence of $T$. 
Then, $(iii)$ follows from  $(ii),$ Corollary \ref{uniqueness},
and Proposition \ref{phi} (Note that in Corollary \ref{uniqueness},
$T$ plays the role of $S$).

To show $(iv)$ we note that, by minimality of $\rho_{\tau},$
\[
Wb_{2}^{e,\Psi,\tau}(\mu,\rho_{\tau})\leq E(\mu)-E(\rho_{\tau}),
\]
and thus 
\begin{multline*}
\frac{1}{2\tau}\int_{\Omega}|\nabla\log\rho_{\tau}+\nabla 
V|^{2}(\rho_{\tau}+\tau h )\hspace{1mm}dy \leq\frac{1}{2\tau}
\int_{\overline{\Omega}\times\overline{\Omega}}| x-y|^{2}\hspace{1mm}d\gamma_{\tau}\\
\leq E(\mu)-E(\rho_{\tau})-\tau\int_{\Omega} e(h)\hspace{1mm}dy+\int_{\partial\Omega\times\Omega}\Psi(x)\hspace{1mm}d\gamma-\int_{\Omega\times\partial\Omega}\Psi(y)\hspace{1mm}d\gamma.
\end{multline*}
Consequently, after makign $\delta$ smaller if necessary, we get 
\[
\int_{\Omega}|\nabla\rho_{\tau}|^{2}\hspace{1mm}dy=C_{2}\int_{\Omega}|\nabla\log\rho_{\tau}|^{2}(\rho_{\tau}+\tau h )\hspace{1mm}dy <\infty.
\]
Here, $C_{2}:=C_{2}(\Psi,e,V,\rho_{0}).$
Also, we have used the fact that $\rho_{\tau}$ is bounded from below 
by $\lambda/(1 + C_{0}),$  $V$ belongs to $W^{1,2}(\Omega),$ and Corollary 
\ref{comparable} holds. Combining  \eqref{eq:-3}, Lemma \ref{bound},  and Lemma \ref{projection}, 
we can see that  
\begin{align*}
-\frac{| y-P(y)|^{2}}{2\tau}-
C_{1}| y-P(y)|-C\sqrt{\tau}
&\leq-\Psi(P(y))+\log\rho_{\tau}(y)+V(y)\\
&\leq C\sqrt{\tau}+C_{1}| y-P(y)|+\frac{| y-P(y)|^{2}}{2\tau},
\end{align*}
where $P(y)$ denotes any of the closest points in $\partial\Omega$
to $y$. Also, $C$ and $C_{1}$ depend only on $\partial\Omega$ and $\Psi$.
Finally $(iv)$ follows from the previous inequality. 
\end{proof}
\end{proposition}

\textit{Proof of Theorem \ref{equation}. }
Let $\rho_{0}$ be bounded and uniformly positive. Let $\delta\in(0,1)$ be such that Propositions \ref{step}
and \ref{inequality2} and Corollary \ref{comparable}  hold for any $\tau\in(0,\delta)$. For any $n$ in $\mathbb{N},$ let $(\gamma_{n}^{\tau},h_{n}^{\tau}$)
be the minimizing pair from $\rho_{n}^{\tau}$ to $\rho_{n+1}^{\tau}$.
Also, let $T_{n}^{\tau}$ be the map that induces $(\gamma_{n}^{\tau})_{\overline{\Omega}}^{\Omega}$
given by Proposition \ref{step} $(ii)$.

Let $t_{f}$ be a positive number larger than $\tau.$ 
Iterating Proposition $\ref{bar},$ we can see that
that there exist positive constants
$\lambda$ and $\Lambda$ such that 
\begin{equation}
\label{exp}
 \bigg((1+C_{0}\tau)^{\frac{1}{\tau}}\bigg)^{-n\tau} \frac{\lambda}{\sup e^{-V}}
 e^{-V}\leq\rho_{n}^{\tau}\leq \frac{\Lambda}{\inf e^{-V}}
 e^{-V}\hspace{1em}\forall n\in \mathbb{N}.
\end{equation}
Note
\[
 \lim_{\tau\rightarrow0}(1+C_{0}\tau)^{\frac{1}{\tau}}=e^{C_{0}}.
\]
Hence, for sufficiently small $\tau$ we obtain a uniform lower bound for $\rho_{n}^{\tau}$
whenever $n\tau \leq t_{f}+1.$ Then, Lemmas \ref{bound},
\ref{boundary}, \ref{interior}, \ref{cost bound},
Corollary \ref{comparable}, and 
Proposition \ref{inequality2}
can be iterated to hold, with uniform constants 
$C,$ $\kappa_{1},$ and $\kappa_{2},$ for all these measures. Henceforth, we assume the condition $n\tau \leq t_{f}+1.$ Fix $\zeta\in C_{c}^{\infty}(\Omega).$

Recall that given $\gamma,$ we denote by $\gamma_{A}^{B}$ its restriction
to $A\times B$. Note that since
\[
\gamma_{n}^{\tau}=\big(\gamma_{n}^{\tau}\big)_{\Omega}^{\Omega}+\big(\gamma_{n}^{\tau}\big)_{\Omega}^{\partial\Omega}+\big(\gamma_{n}^{\tau}\big)_{\partial\Omega}^{\Omega},
\]
by \eqref{interior-mass} we have
\[
\mu_{n}^{\tau}=\big(\pi_{1}\big)_{\#}\big(\gamma_{n}^{\tau}\big)_{\Omega}^{\Omega}+\big(\pi_{1}\big)_{\#}\big(\gamma_{n}^{\tau}\big)_{\Omega}^{\partial\Omega},
\]
and
\[
\mu_{n+1}^{\tau}=\big(\pi_{2}\big)_{\#}\big(\gamma_{n}^{\tau}\big)_{\Omega}^{\Omega}+\big(\pi_{2}\big)_{\#}\big(\gamma_{n}^{\tau}\big)_{\partial\Omega}^{\Omega}-\tau h^{\tau}_{n} \hspace{1mm}dy.
\]
Consequently, we obtain
\begin{multline}
\int_{\Omega}\zeta \hspace{1mm}d\mu_{n+1}^{\tau}-\int_{\Omega}\zeta \hspace{1mm}d\mu_{n}^{\tau}=\int_{\overline{\Omega}\times\overline{\Omega}}\zeta\circ\pi_{2}\hspace{1mm}d\big(\gamma_{n}^{\tau}\big)_{\Omega}^{\Omega}-\int_{\overline{\Omega}\times\overline{\Omega}}\zeta\circ\pi_{1}\hspace{1mm}d\big(\gamma_{n}^{\tau}\big)_{\Omega}^{\Omega}
\\
-\tau\int_{\Omega}\zeta \hspace{0.2mm} h^{\tau}_{n}\hspace{1mm}dy+\int_{\overline{\Omega}\times\overline{\Omega}}\zeta\circ\pi_{2}\hspace{1mm}d\big(\gamma_{n}^{\tau}\big)_{\partial\Omega}^{\Omega}-\int_{\overline{\Omega}\times\overline{\Omega}}\zeta\circ\pi_{1}\hspace{1mm}d\big(\gamma_{n}^{\tau}\big)_{\Omega}^{\partial\Omega}.\label{eq:pre weak}
\end{multline}
First, using Proposition \ref{step} and a Taylor expansion,
\begin{equation}
\begin{aligned}\int_{\overline{\Omega}\times\overline{\Omega}} & \zeta\circ\pi_{2}\hspace{1mm} \hspace{1mm}d\big(\gamma_{n}^{\tau}\big)_{\Omega}^{\Omega}-\int_{\overline{\Omega}\times\overline{\Omega}}\zeta\circ\pi_{1}\hspace{1mm}d\big(\gamma_{n}^{\tau}\big)_{\Omega}^{\Omega} \\
 &=\int_{\overline{\Omega}\times\overline{\Omega}}\big(\zeta(y)-\zeta(x)\big)\hspace{1mm}d\big(\gamma_{n}^{\tau}\big)_{\Omega}^{\Omega}\\
 & =\int_{\overline{\Omega}\times\overline{\Omega}}\big(\zeta(y)-\zeta(T_{n}^{\tau}(y))\big)1_{\{x=T^{\tau}_{n}(y)\}}\hspace{1mm}d\big(\gamma_{n}^{\tau}\big)_{\Omega}^{\Omega}\\
 & =\int_{\overline{\Omega}\times\overline{\Omega}}(\zeta-\zeta\circ T_{n}^{\tau})\circ\pi_{2}\hspace{1mm}d\big(\gamma_{n}^{\tau}\big)_{\Omega}^{\Omega}\\
 & =\int_{\overline{\Omega}\times\overline{\Omega}}(\zeta-\zeta\circ T_{n}^{\tau})1_{\{T^{\tau}_{n}\not\in\partial\Omega\}}\circ\pi_{2}\hspace{1mm}d\big(\gamma_{n}^{\tau}\big)_{\Omega}^{\Omega}+\int_{\overline{\Omega}\times\overline{\Omega}}(\zeta-\zeta\circ T_{n}^{\tau})1_{\{T^{\tau}_{n}\not\in\partial\Omega\}}\circ\pi_{2}\hspace{1mm}d\big(\gamma_{n}^{\tau}\big)_{\partial\Omega}^{\Omega}\\
 & =\int_{\Omega}(\zeta-\zeta\circ T_{n}^{\tau})1_{\{T^{\tau}_{n}\not\in\partial\Omega\}}\hspace{1mm}d\mu_{n+1}^{\tau}+R_{1}(\tau,n)\\
 & =-\int_{\Omega}\langle\nabla\zeta,T_{n}^{\tau}-Id\rangle\rho_{n+1}^{\tau}1_{\{T^{\tau}_{n}\not\in\partial\Omega\}}\hspace{1mm}dy+R_{2}(\tau,n)+R_{1}(\tau,n)\\
 & =-\tau\int_{\Omega}\langle\nabla\zeta,\nabla\rho_{n+1}^{\tau}+\rho_{n+1}^{\tau}\nabla V\rangle1_{\{T^{\tau}_{n}\not\in\partial\Omega\}}\hspace{1mm}dy+R_{2}(\tau,n)+R_{1}(\tau,n).
\end{aligned}
\end{equation}
Second, by item $(iii)$ of Proposition \ref{step}, we have
$h_{n}^{\tau}(y)=[e_{y}^{\prime}]^{-1}(\log\rho_{n+1}^{\tau}(y)+V(y))$
and consequently 
\[
-\tau\int_{\Omega}\zeta h_{n}^{\tau}\hspace{1mm}dy=-\tau\int_{\Omega}\zeta[e^{\prime}]^{-1}(\log\rho_{n+1}^{\tau}+V)\hspace{1mm}dy.
\]
Third, using Corollary \ref{uniqueness},
\begin{align*}\int_{\overline{\Omega}\times\overline{\Omega}} &\zeta\circ\pi_{2}\hspace{1mm}d\big(\gamma_{n}^{\tau}\big)_{\partial\Omega}^{\Omega} -\int_{\overline{\Omega}\times\overline{\Omega}}\zeta\circ\pi_{1}\hspace{1mm}d\big(\gamma_{n}^{\tau}\big)_{\Omega}^{\partial\Omega}\\
 & =\int_{\overline{\Omega}\times\overline{\Omega}}\zeta\circ\pi_{2}1_{\{x=T_{n}^{\tau}(y)\}}\hspace{1mm}d\big(\gamma_{n}^{\tau}\big)_{\partial\Omega}^{\Omega}-\int_{\overline{\Omega}\times\overline{\Omega}}\zeta\circ\pi_{1}1_{\{S_{n}^{\tau}(x)=y\}}\hspace{1mm}d\big(\gamma_{n}^{\tau}\big)_{\Omega}^{\partial\Omega}\\
 & =\int_{\overline{\Omega}\times\overline{\Omega}}\zeta(x)1_{\{x=T_{n}^{\tau}(y)\}}\hspace{1mm}d\big(\gamma_{n}^{\tau}\big)_{\partial\Omega}^{\Omega}-\int_{\overline{\Omega}\times\overline{\Omega}}\zeta_{n+1}(y)1_{\{S_{n}^{\tau}(x)=y\}}\hspace{1mm}d\big(\gamma_{n}^{\tau}\big)_{\Omega}^{\partial\Omega}+R_{3}(\tau,n)\\
 & =\int_{\overline{\Omega}\times\overline{\Omega}}\zeta(x)\hspace{1mm}d((\gamma_{n}^{\tau})_{\partial\Omega}^{\Omega}-\int_{\overline{\Omega}\times\overline{\Omega}}\zeta(y)\hspace{1mm}d(\gamma_{n}^{\tau})_{\Omega}^{\partial\Omega}+R_{3}(\tau,n).
\end{align*}
Here, $S^{\tau}_{n}$ is the map which induces 
$(\gamma_{n}^{\tau})_{\Omega}^{\overline{\Omega}},$ 
given by Corollary \ref{uniqueness}.
Putting the above together, we obtain 
\begin{equation}
\label{eq:stepweak}
\begin{aligned}
\int_{\Omega}\zeta \hspace{1mm}d\mu_{n+1}^{\tau}-\int_{\Omega}\zeta \hspace{1mm}d\mu_{n}^{\tau}=\tau\bigg(&-\int_{\Omega}
\langle \nabla\zeta_{n+1}, 
\nabla\rho_{n+1}^{\tau} +\rho_{n+1}^{\tau}\nabla V \rangle 1_{\{T_{n}^{\tau}\not\in\partial\Omega\}}\hspace{1mm}dx \\ 
 & -\int_{\Omega}\zeta_{n+1}[e^{\prime}]^{-1}(\log\rho_{n+1}^{\tau}+V)\hspace{1mm}dx\bigg)\\
 & +\int_{\overline{\Omega}\times\overline{\Omega}}\zeta(x)\hspace{1mm}d((\gamma_{n}^{\tau})_{\partial\Omega}^{\Omega}-\int_{\overline{\Omega}\times\overline{\Omega}}\zeta(y)\hspace{1mm}d(\gamma_{n}^{\tau})_{\Omega}^{\partial\Omega}+R(n,\tau).
\end{aligned}
\end{equation}
Here, $R(n,\tau)$ is given by
\begin{align*}R&(n,\tau) =R_{1}(n,\tau)+R_{2}(n,\tau)+R_{3}(n,\tau)\\
 & =\tau\int_{\Omega}(\zeta(y)-\zeta\circ T_{n}^{\tau}(y))h_{n}^{\tau}1_{\{T_{n}^{\tau}\not\in\partial\Omega\}}\hspace{1mm}dy\\
 & +\int_{\Omega}\bigg(\int_{0}^{1}\bigg(\langle\nabla\zeta\circ((1-s)T_{n}^{\tau}+sId),Id-T_{n}^{\tau}\rangle-\langle\nabla\zeta,Id-T_{n}^{\tau}\rangle\bigg)\rho_{n+1}^{\tau}1_{\{T_{n}^{\tau}\in\partial\Omega\}}\hspace{1mm}dy\\
 & +\int_{\overline{\Omega}\times\overline{\Omega}}\bigg(\zeta\circ\pi_{2}-\zeta\circ\pi_{1}\bigg)1_{\{x=T_{n}^{\tau}(y)}\hspace{1mm}d\big(\gamma_{n}^{\tau}\big)_{\partial\Omega}^{\Omega}-\int_{\overline{\Omega}\times\overline{\Omega}}\bigg(\zeta\circ\pi_{1}-\zeta\circ\pi_{2}\bigg)1_{\{S_{n}^{\tau}(x)=y\}}\hspace{1mm}d\big(\gamma_{n}^{\tau}\big)_{\Omega}^{\partial\Omega}.
\end{align*}
Recall $\zeta$ is compactly supported. Hence, iterating Lemma \ref{bound}, for sufficiently small $\tau$
 we have that the intersection between the sets $supp(\zeta\circ\pi^{1})$
, $supp(\zeta\circ\pi^{2}),$ and $supp\big(\big(\gamma_{n}^{\tau}\big)_{\Omega}^{\partial\Omega}+(\gamma_{n}^{\tau}\big)_{\partial\Omega}^{\Omega}\big)$
is empty. Consequently, iterating Lemmas \ref{bound} and  \ref{interior}, we deduce
\begin{equation*}
\begin{aligned}\bigg|R(n,\tau)\bigg| & \leq\tau\mbox{Lip}(\zeta)\int_{\Omega}| y-T_{n}^{\tau}(y)|| h_{n}^{\tau}| \hspace{1mm}dy\\
 & +\mbox{Lip}(\nabla\zeta)\int_{\Omega}| T_{n}^{\tau}-Id|^{2}\rho_{n+1}^{\tau}\hspace{1mm}dy\\
 & \leq C_{1}(\zeta,\Psi,e,V,\rho_{0},\Omega)\bigg[\tau^{\frac{3}{2}}+\int_{\Omega}| T_{n}^{\tau}-Id|^{2}
 (\rho_{n+1}^{\tau}+\tau h)\hspace{1mm}dy\bigg].
\end{aligned}
\end{equation*}
Here, we have used  Corollary \ref{comparable} and the fact that $\Omega$ is bounded. Now, by Proposition \ref{inequality2}
\begin{equation}
\int_{\overline{\Omega}\times\overline{\Omega}}\frac{| x-y|^{2}}{2\tau}\hspace{1mm}d\gamma_{n}^{\tau}\leq C_{2}(\Psi,e,V,\rho_{0})\bigg(E(\rho_{n}^{\tau})-\int_{\Omega}\Psi \hspace{1mm}d\mu_{n}^{\tau}-E(\rho_{n+1}^{\tau})+\int_{\Omega}\Psi \hspace{1mm}d\mu_{n+1}^{\tau}+\tau\bigg).
 \label{eq:-15}
\end{equation}
Thus, combining the above inequalities with \eqref{interior-mass}, Lemma \ref{interior},
and Corollary \ref{comparable},
we get
\[
\bigg|R(n,\tau)\bigg|\leq C_{3}(\zeta,\Psi,e,V,\rho_{0})\bigg(\tau^{3/2}+\tau\bigg[E(\rho_{n}^{\tau})-\int_{\Omega}\Psi \hspace{1mm}d\mu_{n}^{\tau}-E(\rho_{n+1})+\int_{\Omega}\Psi \hspace{1mm}d\mu_{n+1}^{\tau}\bigg]\bigg).
\]
This implies
\begin{multline*}
\bigg|\sum_{n=M}^{N-1}R(n,\tau)\bigg|\leq C_{3}(\zeta,\Psi,e,V,\rho_{0})
\bigg(\sqrt{\tau}(M-N)\tau\\
+\tau\bigg[E(\rho_{M}^{\tau})-\int_{\Omega}\Psi \hspace{1mm}d\mu_{M}^{\tau}-E(\rho^{\tau}_{N})+\int_{\Omega}\Psi \hspace{1mm}d\mu_{N}^{\tau}\bigg]\bigg),
\end{multline*}
for sufficiently small $\tau$  and all integers $N$ and $M$ such that $\tau M \leq \tau N \leq t_{f} + 1.$

Let $\tau=\tau_{k}.$ Also, define
\[
\rho^{\tau_{k}}(t)=\rho_{n+1}^{\tau_{k}}\hspace{1em}\mbox{for}\hspace{1em}t\in\big((n+1)\tau_{k},n\tau_{k}],
\]
and
\[
\theta_{h}\rho^{\tau_{k}}(t)=\rho^{\tau_{k}}(t+h),
\]
for any positive constant $h.$
Now, choose $0\leq r<s<t_{f}+1$ and add up \eqref{eq:stepweak}
from $M=[r\backslash\tau_{k}]$ to $N=[s\backslash\tau_{k}]-1$ to
get

\begin{equation}
\begin{aligned}\int_{\Omega}\zeta\rho^{\tau_{k}}(s)\hspace{1mm}dx &-\int_{\Omega}\zeta\rho^{\tau_{k}}(r)\hspace{1mm}dx\\
 &=\int_{\tau_{k}[t\backslash\tau_{k}]}^{\tau_{k}[s\backslash\tau_{k}]}\bigg(-\int_{\Omega}\langle \nabla\zeta, \nabla\rho^{\tau_{k}}(t)+\nabla V\rho^{\tau_{k}}\rangle 1_{\{T_{n}^{\tau_{k}}\not\in\partial\Omega;[t\backslash\tau_{k}]=n\}}\hspace{1mm}dx\\ 
 & -\int_{\Omega}\zeta[e^{\prime}]^{-1}(\log\rho^{\tau_{k}}(t))+V\bigg)dt+\bigg|\sum_{n=M}^{N}R(n,\tau_{k})\bigg|\\
 & =\int_{\tau_{k}[t\backslash\tau_{k}]}^{\tau_{k}[s\backslash\tau_{k}]}\bigg(\int_{\Omega}\big[\Delta\zeta-\langle\nabla\zeta,\nabla V \rangle\big] \rho^{\tau_{k}}(t)\hspace{1mm}dx-\int_{\Omega}\zeta[e^{\prime}]^{-1}(\log\rho^{\tau_{k}}(t)+V)\bigg)dt\\
 & \hspace{26em}+\sum_{n=M}^{N}R(n,\tau_{k}).
 \label{eq:discreteweak}
\end{aligned}
\end{equation}
Here, we have used the fact that by Lemma \ref{bound},
for sufficiently small $\tau$, $\{T_{n}^{\tau_{k}}\in\partial\Omega;[t\backslash\tau_{k}]=n\}$
and $supp(\zeta)$ are disjoint.\\
The strategy to pass to the limit is to use the Aubin-Lions Theorem
\cite[Theorem 5]{Aubin-Lions}.
Let $U$ be an open set with Lipschitz boundary 
whose closure is compactly contained in $\Omega$. Also, set $p>d+1.$ First, note  $L^{2}(U)$ embeds 
in the dual of $W^{2,p}(U).$ We will denote this space by $W^{-2,p}(U)$. Second, observe
$W^{1,2}(U)$ embeds compactly in $L^{2}(U)$ (recall $\Omega$ is bounded).
Thus, in order to use the Aubin-Lions Theorem, we will show $\rho^{\tau_{k}}$ is bounded
in $L^{2}(0,t_{f};L^{2}(U))\cap$
$L_{loc}^{1}(0,t_{f};W^{1,2}(U))$ and 
\[
||\theta_{h}\rho^{\tau_{k}}-\rho^{\tau_{k}}||_{L^{1}(t_{1},t_{2};
W^{-2,p}(U))}\rightarrow0 \hspace{1em}\forall 0\leq t_{1}<t_{2}<t_{f},
\]
as $h\rightarrow0$, uniformly.\\
 Given $t\in(t_{1},t_{2})$ set $N=\ceil[\bigg]{\frac{t+h}{\tau_{k}}}-1$ and
$M=\ceil[\bigg]{\frac{t}{\tau_{k}}}$. For each $\zeta\in W^{2,p}(U),$
we consider an extension to $\mathbb{R}^{n}$ (not relabeled) satisfying
$supp(\zeta)\subset \Omega$ and
\[
||\zeta||_{W^{2,p}(\mathbb{R}^{n})}\leq C_{4}||\zeta||_{W^{2,p}(U)}.
\]
Here, $C_{4}:=C_{4}(U,\Omega)$. Then,
\begin{align*}\int_{\Omega}\zeta(\theta_{h}\rho^{\tau_{k}}(t)&-\rho^{\tau_{k}}(t))
\hspace{1mm}dx\\
& =\sum_{n=M}^{N}\int_{\Omega}\zeta \hspace{1mm}d\mu_{n+1}^{\tau_{k}}-\int_{\Omega}\zeta \hspace{1mm}d\mu_{n}^{\tau_{k}}\\
&=\sum_{n=M}^{N}\int_{\Omega}\zeta(y)-\zeta(x)\hspace{1mm}d\gamma_{n}^{\tau_{k}}-\int_{\Omega}\zeta\tau_{\tau_{k}}h_{n}^{\tau_{k}}\hspace{1mm}dx\\
&=\sum_{n=M}^{N}\int_{\overline{\Omega}\times\overline{\Omega}}\int_{0}^{1}\langle\nabla\zeta(x+s(y-x)),y-x\rangle \hspace{1mm}ds \hspace{1mm}d\gamma_{n}^{\tau_{k}}-\int_{\Omega}\zeta\tau h_{n}^{\tau_{k}}\hspace{1mm}dx\\
&\leq\sum_{n=M}^{N}\int_{\overline{\Omega}\times\overline{\Omega}}\bigg(\int_{0}^{1}|\nabla\zeta(x+s(y-x))|^{2}ds\hspace{1mm}d\gamma_{n}^{\tau_{k}}\bigg)^{\frac{1}{2}}\bigg(\int_{\overline{\Omega}\times\overline{\Omega}}| y-x|^{2}\hspace{1mm}d\gamma_{n}^{\tau_{k}}\bigg)^{\frac{1}{2}}\\
& \hspace{23em}+C_{5}\tau_{k}||\zeta||_{W^{2,p}(U)}\\
&\hspace{10em}\leq C_{6} ||\zeta||_{W^{2,p}(U)}\sum_{n=M}^{N}\bigg[\bigg(\int_{\overline{\Omega}\times\overline{\Omega}}| y-x|^{2}\hspace{1mm}d\gamma_{n}^{\tau_{k}}\bigg)^{\frac{1}{2}}+\tau_{k}\bigg].
\end{align*}
Here, we used Lemma \ref{interior} and the
embedding of $W^{2,p}(U)$ into $C^{1}(U)$.
Also, $C_{5}:=C_{5}(t_{f},\Psi,e,V,\rho_{0})$ and  $C_{6}:=C_{6}(\Omega,U,t_{f},\Psi,e,V,\rho_{0}).$
Consequently, it follows:
\begin{equation}
\begin{aligned}||\theta_{h}\rho^{\tau_{k}}(t)&-\rho^{\tau_{k}}(t)||_{W^{-2,p}(U)}\\
 & =\sup_{||\zeta||_{W^{2,p}(U)}=1}\int_{\Omega}\zeta\big(\theta_{h}\rho^{\tau_{k}}(t)-\rho^{\tau_{k}}(t)\big)\hspace{1mm}dy\\
 & \leq C_{6}\sum_{n=M}^{N}\bigg[\bigg(\int_{\overline{\Omega}\times\overline{\Omega}}| y-x|^{2}\hspace{1mm}d\gamma_{n}^{\tau}\bigg)^{\frac{1}{2}}+\tau_{k}\bigg]\\
 & \leq C_{6}\bigg(\tau_{k}(N-M)+\big(\tau_{k}(N-M)\big)^{\frac{1}{2}}\bigg(\sum_{n=M}^{N}\bigg[\bigg(\int_{\overline{\Omega}\times\overline{\Omega}}\frac{| y-x|^{2}}{\tau}\hspace{1mm}d\gamma_{n}^{\tau_{k}}\bigg)\bigg)^{\frac{1}{2}}\\
 & \leq C_{7}\bigg(h+\sqrt{h}\bigg[\sum_{n=M}^{N}E(\rho_{n}^{\tau_{k}})-\int_{\Omega}\Psi\rho_{n}^{\tau_{k}}\hspace{1mm}dy-E(\rho_{n}^{\tau_{k}})+\int_{\Omega}\Psi\rho_{n}^{\tau_{k}}+\tau \hspace{1mm}dy\bigg]^{\frac{1}{2}}\bigg)\\
 & \leq C_{7}\bigg(h+\sqrt{h}\bigg[E(\rho_{M}^{\tau_{k}})-\int_{\Omega}\Psi\rho_{M}^{\tau_{k}}\hspace{1mm}dy-E(\rho_{N+1}^{\tau_{k}})+\int_{\Omega}\Psi\rho_{N+1}^{\tau_{k}}+h\hspace{1mm}dy)\bigg]^{1/2}\bigg).
\label{equicont}
 \end{aligned}
\end{equation}

Here, we used the Jensen inequality and Proposition \ref{inequality2}. Also, $C_{7}:=C_{7}(t_{f},\Omega,U,\Psi,e,V,\rho_{0}).$
This shows $||\theta_{h}\rho^{\tau_{k}}-\rho^{\tau_{k}}||_{L^{1}(t_{1},t_{2};W^{-2,p}(U))}\rightarrow0$
as $h\rightarrow0,$ uniformly in $k$.
In order to show that $\rho^{\tau_{k}}$ is bounded in $L^{1}(0,t_{f};W^{1,2}(U))$
we use \eqref{interior-mass}, Proposition \ref{step}, Lemma \ref{interior}, Corollary \ref{comparable} and Proposition
\ref{inequality2}
to obtain
\begin{equation}\label{H1}
\begin{aligned}\int_{\Omega}\bigg|\nabla\log\rho_{n+1}^{\tau_{k}}&+\nabla V\bigg|^{2}\big(\rho_{n+1}^{\tau_{k}}+\tau_{k}h_{n}^{\tau_{k}}\big)\hspace{1mm}dx\hspace{1mm}\tau_{k} \leq \int_{\overline{\Omega}\times\overline{\Omega}} \frac{| x-y|^{2}}{2\tau_{k}}\hspace{1mm}d\gamma_{n}^{\tau_{k}}\\
 &\hspace{8em} \leq C_{8}\bigg(E(\rho_{n}^{\tau_{k}})-\int_{\Omega}\Psi \hspace{1mm}d\mu_{n}^{\tau_{k}}-E(\rho_{n+1}^{\tau_{k}})+\int_{\Omega}\Psi \hspace{1mm}d\mu_{n+1}^{\tau_{k}}+\tau_{k}\bigg),
\end{aligned}
\end{equation}
for every $n$ in $[0,t_{f}/\tau].$ Here, $C_{8}:=C_{8}(\Psi,e,V,\rho_{0}).$ 
By Proposition \eqref{exp}, we have that 
$\tilde{\Lambda}\geq\rho_{n+1}^{\tau_{k}}
\geq
\tilde{\lambda},$ for some positive constants $\tilde{\lambda}
:=\tilde{\lambda}(t_{f})$ and $\tilde{\Lambda}$ and  every $n$ in $[0,t_{f}/\tau]$. 
Then, using \eqref{H1}, the Young inequality, Corollary \ref{comparable}, and the fact that $V$ is in
$W^{1,2}(\Omega),$ we get
\begin{equation}
\int_{0}^{t_{f}}\bigg(\int_{\Omega}|\nabla\rho^{\tau_{k}}(t)
|^{2}\hspace{1mm}dx\bigg)dt<C_{9}(t_{f},\Psi,e,V,\rho_{0})(1+t_{f}).\label{eq:fisher}
\end{equation}
Hence, we conclude that $\{\rho^{\tau_{k}}\}_{k=1}^{\infty}$ is equibounded in $L^{2}(0,t_{f};W^{1,2}(\overline{\Omega}))$.
Also, from Proposition \ref{bar}, we  have that $\{\rho^{\tau_{k}}\}_{k=1}^{\infty}$ is equibounded in $L^{2}(0,t_{f};L^{2}(\overline{\Omega}))$ as well.

This shows that the hypotheses of the Aubin-Lions Theorem are satisfied. Thus,
we obtain a map $\rho\in L^{2}(0,t_{f};L^{2}(U))$ and a subsequence (not
relabeled). Such a subsequence satisfies that  $\rho^{\tau_{k}}\rightarrow\rho$
in $L^{2}(0,t_{f};L^{2}(U))$ as $k\rightarrow\infty$.
By \eqref{equicont} and the Arzela Ascoli Theorem,
this subsequence converges to $\rho$ in
$C^{1/2}(0,t_{f};W^{-2,p}(U)).$

 The final step is to use a diagonal argument
along a sequence of sets $U$ increasing to $\Omega$. By doing this we obtain 
a further subsequence converging in $L^{2}(0,t_{f};L_{loc}^{2}(\Omega))$ 
and in 
$C^{1/2}(0,t_{f};W^{-2,p}_{loc}(\Omega))$
to a
map $\rho\in L^{2}(0,t_{f};L_{loc}^{2}(\Omega))$, which we have not relabeled.\\
Consequently, for any $\zeta\in C_{c}^{\infty}(\Omega),$
\[
\int_{\Omega}\zeta\rho^{\tau_{k}}(s)\hspace{1mm}dx-\int_{\Omega}\zeta\rho^{\tau_{k}}(r)
\hspace{1mm}dx\rightarrow\int_{\Omega}\zeta\rho(s)\hspace{1mm}dx-
\int_{\Omega}\zeta\rho(r)\hspace{1mm}dx.
\]
Let $U$ be an open set such that $supp(\zeta)\subset U$ and
$\overline{U}$ is compactly contained in $\Omega$. By Proposition \ref{bar}, there exists
$C_{10}:=C_{10}(\lambda,\Lambda,t_{f})$ such that
\[
 \int_{\Omega}|\zeta e^{\prime}(\log\rho^{\tau_{k}}(t)+V)|\hspace{1mm}dx
 \leq C_{10}\int_{\Omega}||\zeta||_{L^{\infty}(\Omega)}\hspace{1mm}dx<\infty, 
\]
and 
\[
 \int_{\Omega}|\big[\Delta\zeta-\langle\nabla\zeta,\nabla V\rangle\big]\rho^{\tau_{k}}(t)|\hspace{1mm}dx
 \leq C_{10}\int_{\Omega}\big[||\Delta\zeta||_{L^{\infty}(\Omega)}+
 |\nabla\zeta|^{2}+|\nabla V|^{2}\big]\hspace{1mm}dx<\infty.
\]
Recall that $V$ is in $W^{1,2}(\Omega).$
Using the fact that 
$\rho^{\tau_{k}}(t)\rightarrow \rho(t)$ 
in $L^{2}(U)$ for almost every $t$ and the 
dominated convergence Theorem, we get
\[\int_{\Omega}\zeta e^{\prime}(\log\rho^{\tau_{k}}(t)+V)\hspace{1mm}dx\rightarrow\int_{\Omega}\zeta e^{\prime}(\log\rho(t)+V)\hspace{1mm}dx,\]
and
\[
 \int_{\Omega}\big[\Delta\zeta-\langle\nabla\zeta,\nabla V\rangle\big]\rho^{\tau_{k}}(t)\hspace{1mm}dx
 \rightarrow\int_{\Omega}\big[\Delta\zeta
 -\langle\nabla\zeta,\nabla V\rangle\big]\rho(t)\hspace{1mm}dx.
\]
for almost every $t$ in $[0,t_{f}].$
Then, a second application of the dominated convergence Theorem gives us 
\begin{multline*}
\int_{\tau_{k}[r\backslash\tau_{k}]}^{\tau_{k}[s\backslash\tau_{k}]}\bigg(\int_{\Omega}\big[\Delta\zeta-\langle\nabla\zeta,\nabla V\rangle\big]\rho^{\tau_{k}}(t)\hspace{1mm}dx-\int_{\Omega}\zeta[e^{\prime}]^{-1}(\log\rho^{\tau_{k}}(t)+V)\bigg)dt\\
\rightarrow\int_{r}^{s}\bigg(\int_{\Omega}\big[\Delta\zeta-\langle\nabla\zeta,\nabla V\rangle\big]\rho(t)\hspace{1mm}dx-\int_{\Omega}\zeta[e^{\prime}]^{-1}(\log\rho(t)+V)\bigg)dt  
\end{multline*}
Moreover, \eqref{eq:fisher} and the Fatou Lemma yield
\[
\int_{0}^{t_{f}}\liminf_{k\rightarrow\infty}\bigg(\int_{\Omega}|\nabla\rho^{\tau_{k}}(t)|^{2}\hspace{1mm}dx\bigg)dt<\infty.
\]
This gives
\[
\liminf_{k\rightarrow\infty}\bigg(\int_{\Omega}|\nabla\rho^{\tau_{k}}(t)|^{2}\hspace{1mm}dx\bigg)<\infty\hspace{1em}\mbox{for a.e \ensuremath{t\geq0.}}
\]
Now, for any $t$ such that the above $\liminf$ is finite, consider
a subsequence $k_{n}$ (depending on $t$) such that
\[
\sup_{n\in\mathcal{\mathbb{N}}}\int_{\Omega}|\nabla\rho^{\tau_{k_{n}}}(t)|^{2}\hspace{1mm}dx<\infty.
\]
This implies that $\rho^{\tau_{k_{n}}}(t)$ is uniformly bounded in $W^{1,2}(\Omega)$.
Recall that $\rho^{\tau_{k}}(t)\rightarrow\rho(t)$ in $L^{2}(0,t_{f};L_{loc}^{2}(\Omega))$. Hence, $\rho^{\tau_{k_{n}}}(t)\rightharpoonup\rho(t)$
in $W^{1,2}(\Omega)$. Then, by Proposition \ref{step} we get that 
$\rho(t)-e^{\Psi-V}$ is in $W_{0}^{1,2}(\Omega)$. Hence, we  have shown that
the map $t\rightarrow \rho(t)$ is a weak solution of \eqref{T2}. Finally, to show \eqref{Z3},
we  use the fact that   $\rho^{\tau_{k}}(t)\rightarrow\rho(t)$ in 
$C^{1/2}(0,t_{f};W^{-2,p}_{loc}(\Omega))$
and  \eqref{exp}.\\
 $\square$

\section{Properties of Minimizers}
In this section, we let $\tau>0$ be a fixed time step. We set 
$\mu=\rho\hspace{1mm}\mathcal{L}^{d}_{|\Omega}$ and
denote by $\mu_{\tau}=\rho_{\tau}\hspace{1mm}\mathcal{L}_{|\Omega}^{d}$ a minimizer
of \eqref{minmov}. The density $\rho$ is assumed to be strictly positive. Additionally, we let
$(\gamma,h)$ be the associated optimal pair for $(\rho,\rho_{\tau})$. 
The objective of the section is to show some properties of $\mu_{\tau}$
that are necessary to prove the main result, Theorem \ref{equation}.

A priori, it is not immediate that one can obtain
a $\tau$ independent positive lower bound
for  $\mu_{\tau}$; this
is studied in Proposition \ref{bar}. Consequently, we cannot
use Proposition \ref{phi}.  However, since $\mu$ and $\mu_{\tau}$ are absolutely continuous, 
Lemma \ref{cre-mass} 
guarantees the existence of maps $T$ and $S$ with the property that
$(Id,T)_{\#}\mu=\gamma_{\Omega}^{\overline{\Omega}}$ and $(S,Id)_{\#}\mu_{\tau}=\gamma_{\overline{\Omega}}^{\Omega}.$
We will use this maps throughout this section.\\
(We remark that in the proof of Proposition 4.1, we get existence of 
$\rho_{\tau}$ without using Proposition \ref{phi} or any result 
from this section.)
\begin{lemma}\textbf{(Boundedness and Uniform positivity)}
\label{uniform} The minimizer
$\rho_{\tau},$ defined above,
is bounded and uniformly positive.
\begin{proof}
Let $r$ and $R$ be positive constants such that
\[
\log r+V<-\frac{\mbox{diam}(\Omega)^{2}}{2\tau}-||\Psi||_{\infty},
\] 
\[
\log R+V>\frac{\mbox{diam}(\Omega)^{2}}{\tau}+2||\Psi||_{\infty},
\]
and
\[
 R+1>\tau ||h||_{\infty}.
\]
(Recall that due to Lemma \ref{cre-mass}, $h$ is bounded).
Now, set $A_{R}^{\kappa}=\{\rho_{\tau}>R+1\}\cap\{\rho_{\tau}+\tau h<\kappa\}$
and $A_{r}=\{\rho_{\tau}+1<r\}.$ Define $\tilde{\gamma}$ by
\[
\tilde{\gamma}=\gamma+\varepsilon(P_{-\Psi,\tau,}Id)_{\#}1_{A_{r}}
-\varepsilon\gamma_{\overline{\Omega}}^{A_{R}^{\kappa}}+
\varepsilon(Id,P_{\Psi,\tau,})_{\#}\gamma_{\Omega}^{A_{R}^{\kappa}}
\]
If we set $\tilde{\rho}=\pi_{2}\tilde{\gamma}-\tau h,$ then
\[
\tilde{\rho}=\begin{cases}
\rho_{\tau} & \mbox{{in}}\hspace{1em}\Omega\backslash A_{R}^{\kappa}\cup A_{r},\\
\rho_{\tau}+\varepsilon & \mbox{{in}}\hspace{1em}A_{r},\\
\rho_{\tau}-\varepsilon(\rho_{\tau}+\tau h) & \mbox{{in}}\hspace{1em}A_{R}^{\kappa}.
\end{cases}
\]
Hence if $\frac{R+1}{\kappa}>\varepsilon,$ then $\tilde{\rho}\in\mathcal{M}(\Omega)$
and $(\tilde{\gamma},h)\in ADM(\mu,\tilde{\rho}).$ By optimality,
\begin{multline*}
0\leq E(\tilde{\rho})-E(\rho_{\tau})+C_{\tau}(\tilde{\gamma},h)-C_{\tau}(\gamma,h)
\\
\hspace{1em}\leq \int_{A_{r}}\big[\mathcal{E}(\rho_{\tau}+\varepsilon)-\mathcal{E}(\rho_{\tau})\big]\hspace{1mm}dy+\varepsilon\int_{A_{r}}\bigg(\frac{| P_{-\Psi,\tau,}(y)-y|^{2}}{2\tau}-\Psi(P_{-\Psi,\tau,}(y))\bigg)\hspace{1mm}dy
\\
\hspace{1.2em}+\int_{A_{R}^{\kappa}}\big[\mathcal{E}(\rho_{\tau}-\varepsilon(\rho_{\tau}+\tau h))-\mathcal{E}(\rho_{\tau})\big]\hspace{1mm}dy+\varepsilon\int_{A_{R}^{\kappa}}\bigg(\frac{\mbox{diam}^{2}(\Omega)}{2\tau}+||\Psi||_{\infty}
\\
\hspace{20em}-\frac{|S(y)-y|^{2}}{2\tau}+||\Psi||_{\infty} \bigg)(\rho+\tau h)\hspace{1mm}dy.
\end{multline*}
Then, by convexity of $\mathcal{E}$ with respect to its first variable
\begin{multline*}
0\leq\varepsilon\int_{A_{r}}\bigg[\mathcal{E}^{\prime}(\rho_{\tau}+\varepsilon)+\frac{| P_{-\Psi,\tau,}(y)-y|^{2}}{2\tau}-\Psi(P_{-\Psi,\tau,}(y))\bigg]\hspace{1mm}dy
\\
+\varepsilon\int_{A_{R}^{\kappa}}\bigg[-\mathcal{E}^{\prime}(\rho_{\tau}-\varepsilon(\rho_{\tau}+\tau h))+\frac{\mbox{diam}^{2}(\Omega)}{2\tau}+||\Psi||_{\infty}
\\
-\frac{|S(y)-y|^{2}}{2\tau}+||\Psi||_{\infty}\bigg](\rho+\tau h)\hspace{1mm}dy.
\end{multline*}
Now if we let $\varepsilon<\mbox{min}\big(1/\kappa,1\big)$,
then $\mathcal{E}^{\prime}(\rho_{\tau}-\varepsilon(\rho_{\tau}+\tau h)>\log R+V$
in $A_{R}^{\kappa}$ and $\mathcal{E}^{\prime}(\rho_{\tau}+\varepsilon)<\log r+V$ in $A_{r}.$ Hence,
by construction both integrands are strictly negative. Thus, we conclude
$|A_{r}|=|A_{R}^{\kappa}|=0.$ Since $\kappa$ was arbitrary and $\rho_{\tau}+\tau h\in L^{1}(\Omega)$,
we obtain the desired result.
\end{proof}
\end{lemma}
In the next Proposition, we will say that a point in $\Omega$ is a density point for 
$\rho_{\tau}+\tau h$ if it is a point of with density $1$ for the set 
$\{\rho_{\tau}+\tau h>0 \}$ and it is a Lebesgue point for $\rho_{\tau}$ and $h.$
As before, the interior of the set of points where $\mathcal{E}$ is finite will
be denoted by $D(\mathcal{E}).$ 
The Proposition is 
the analogue in our context of
\cite[Proposition 3.7]{Figalli-Gigli}.
\begin{proposition}\label{perturba2}
With the notation introduced at the beginning of this section, the following inequalities hold:  
\begin{itemize}
\item Let $y_{1}$ and $y_{2}$ be points in $\Omega.$
Assume that $y_{1}$ is a  density point for $\rho_{\tau}+\tau h$ 
and Lebesgue point for $S$ and that $y_{2}$ is a Lebesgue point for $\rho_{\tau}$
. Then 
\end{itemize}
\begin{equation}
\log(\rho_{\tau}(y_{1}))+V(y_{1})+\frac{| y_{1}-S(y_{1})|^{2}}{2\tau}
\leq\log(\rho_{\tau}(y_{2}))+V(y_{2})+\frac{| y_{2}-S(y_{1})|^{2}}
{2\tau}.\label{eq:-4}
\end{equation}
\begin{itemize}
\item Let $x\in\Omega$ be a Lebesgue point for $\rho,$ and $T$ and
assume that $T(x)\in\partial\Omega.$ Assume further that $y\in\Omega$
is a Lebesgue point for $\rho_{\tau}$ and $h$. Then 
\end{itemize}
\begin{equation}
\frac{| x-T(x)|^{2}}{2\tau}+\Psi(T(x))\leq\log(\rho_{\tau}(y)
)+V(y)+\frac{| x-y|^{2}}{2\tau}.\label{eq:-5}
\end{equation}
\begin{itemize}
\item Let $y_{1}$ and $S(y_{1})$ be points in $\Omega.$
Assume that $y_{1}$ is a density point for $\rho+\tau h$
and a Lebesgue point for $S$.
Then for any $y_{2}\in\partial\Omega,$ we have
\begin{equation}
\log(\rho_{\tau}(y_{1}))+V(y_{1})+\frac{| y_{1}-S(y_{1})|^{2}}{2\tau}
\leq\frac{| y_{2}-S(y_{1})|^{2}}{2\tau}+\Psi(y_{2}).\label{eq:-6}
\end{equation}
\item Let $y\in\Omega$ be a density point for $\rho_{\tau}+\tau h$ and 
a Lebesgue point for $P_{-\Psi,\tau}$. Then 
\end{itemize}
\begin{multline}
-\frac{| y-P_{-\Psi,\tau}(y)|^{2}}{2\tau}\leq-\Psi(P_{-\Psi,\tau}(y))
+\log\rho_{\tau}(y)+V(y)
\\ \leq\frac{| y-P_{-\Psi,\tau}(y)|| y-S(y)
|}{\tau}+\frac{| y-P_{-\Psi,\tau}(y)|^{2}}{2\tau}.\label{eq:-3}
\end{multline}
\begin{itemize}
\item Let $y\in\Omega$ be a density point for $\rho_{\tau}+\tau h$
and a Lebesgue point for  
$S$ and $P_{-\Psi,\tau}$  and assume that $S(y)=P_{-\Psi,\tau}(y)\in\partial\Omega$.
Then 
\end{itemize}
\begin{equation}
\log\big(\rho_{\tau}(y)\big)+V(y)+\frac{| y-P_{-\Psi,\tau}(y)|^{2}}{2\tau}-\Psi\big(P_{-\Psi,\tau}(y)\big)=0.\label{eq:-7}
\end{equation}
\begin{proof}
\textbf{-Heuristic argument. }
First we start with \eqref{eq:-5}. Consider a point $y\in\Omega$ and suppose
that we take some mass from $x\in\Omega$ and instead of sending it
to $T(x)\in\partial\Omega,$ we send it to $y$. Then, we are paying
$\log\rho_{\tau}(y)+V(y)$ in terms of the entropy and $\frac{| x-y|^{2}}{2\tau}$
in terms of the cost. We are also saving $\frac{| x-T(x)|^{2}}{2\tau}+\Psi(T(x))$
in terms of the cost. Hence, by minimality, we must have
\[
\frac{| x-T(x)|^{2}}{2\tau}+\Psi(T(x))\leq\log(\rho_{\tau}(y))+V(y)+\frac{| x-y|^{2}}{2\tau}.
\]
Now we proceed with \eqref{eq:-6}. We take some mass from $S(y_{1})$
and instead of sending it to $y_{1}$, we send it to $y_{2}.$ Then,
we pay $\frac{| y_{2}-S(y_{1})|^{2}}{2\tau}+\Psi(y_{2})$ in
terms of the cost. We also save $\frac{| y_{1}-S(y_{1})|^{2}}{2\tau}$
in terms of the cost and $\log(\rho_{\tau}(y_{1}))+V(y_{1})$ in terms
of the entropy. Thus, the desired result follows by minimality; (\ref{eq:-4}) is analogous.

To show (\ref{eq:-3}), we argue as follows. Pick a point $y\in\Omega$
and perturb $\rho_{\tau}$ by taking some small mass from a point
in $P_{-\Psi,\tau}(y)\in\mathcal{P}_{-\Psi,\tau}(y)$ and putting it onto $y$. 
(In the case
that $S(y)\in\partial\Omega$ we choose the point to be $S(y)=P_{-\Psi,\tau}(y)\in\mathcal{P}_{-\Psi,\tau}(y)$. It is
easy to verify that the minimality of the pair allows us to do this almost everywhere in $\Omega$.)
In this way, we pay $\log(\rho_{\tau}(y))+V(y)$
in terms of the entropy and $\frac{| y-P_{-\Psi,\tau}(y)|^{2}}{2\tau}-\Psi(P_{-\Psi,\tau}(y))$
in term of the cost. Consequently, by minimality we must have

\begin{equation}
\log(\rho_{\tau}(y))+V(y)-\Psi(P_{-\Psi,\tau}(y))\geq-\frac{| y-P_{-\Psi,\tau}(y)|^{2}}{2\tau}.\label{eq:-8}
\end{equation}

Now consider two cases. First, if $S(y)\in\Omega$, we stop sending
some mass from $S(y)$ to $y$. Instead, we send it to $P_{-\Psi,\tau}(y)\in\partial\Omega$.
By doing this, we earn $\log(\rho_{\tau}(y))+V(y)$ in terms of
the entropy and $\frac{| S(y)-y|^{2}}{2\tau}$in terms of the cost.
On the other hand, we pay $\frac{| S(y)-P_{-\Psi,\tau}(y)|^{2}}{2\tau}+\Psi(P_{-\Psi,\tau}(y))$
in terms of the cost. Thus, 

\begin{align*}
-\Psi(P_{-\Psi\tau}(y))+\log(\rho_{\tau}(y))+V(y)+
\frac{| S(y)-(y)|^{2}}{2\tau}&\leq\frac{| S(y)
-P_{-\Psi,\tau}(y)|^{2}}{2\tau}\\
&\leq\frac{\big(| y-S(y)|+| y-P_{-\Psi,\tau}(y)|\big)^{2}}{2\tau}.
\end{align*}
Consequently, when we combine this with \eqref{eq:-8}, we obtain 
\begin{multline}
-\frac{| y-P_{-\Psi,\tau}(y)|^{2}}{2\tau}\leq-\Psi(P_{-\Psi,\tau}(y)
)+\log\rho_{\tau}(y)+V(y)
\\ \leq\frac{| y-P_{-\Psi,\tau}(y)|| y-S(y)
|}{\tau}+\frac{| y-P_{-\Psi,\tau}(y)|^{2}}{2\tau}.\label{pre-trace}
\end{multline}
Second, if $S(y)\in\partial\Omega$ the above inequality is obtained
as a consequence of \eqref{eq:-7} and \eqref{eq:-8}.\\
The proof of \eqref{eq:-7} is a sort of converse of \eqref{eq:-8}. Indeed,
as $S(y)=P_{-\Psi,\tau}(y)\in\partial\Omega,$ we know that the mass
at $y$  comes from the boundary. Hence, we can perturb $\rho_{\tau}$
by taking a bit less mass from the boundary, so that there is  less mass in $y$. In this way, we save $\log(\rho_{\tau}(y))+V(y)$
in terms of the entropy and $\frac{| y-P_{-\Psi,\tau}(y)|^{2}}
{2\tau}-\Psi(P_{-\Psi,\tau}(y))$
in terms of the cost. Hence,
\[
-\bigg(\log(\rho_{\tau}(y))+V(y)+\frac{| y-P_{-\Psi,\tau}(y)|^{2}}{2\tau}-\Psi(P_{-\Psi,\tau}(y))\bigg)\geq0.
\]
From \eqref{eq:-8}, we get the opposite inequality and thus we conclude
the argument.\\
\textbf{-Rigorous proof}.  We only
prove \eqref{eq:-6}; the proofs of the other inequalities are analogous.

Let $\mathcal{T}_{y_{1}}^{y_{2}} 
\hspace{1mm}:\hspace{1mm}B_{r}(y_{1})\rightarrow \partial\Omega$
be identically  equal to $y_{2}$ in $B_{r}(y_{1}) $ and 
let $r$ be positive constant such that $B_{r}(y_{1})$
is contained in $\Omega$.
Define the plan $\gamma^{r,\varepsilon}$ by
\[
\gamma_{r}^{\varepsilon}=\gamma_{\overline{\Omega}}^{B_{r}(y_{1})^{c}}
+(1-\varepsilon)\gamma_{\overline{\Omega}}^{B_{r}(y_{1})}+\varepsilon
\big(\pi^{1},\mathcal{T}_{y_{1}}^{y_{2}})_{\#}\gamma_{\overline{\Omega}}
^{B_{r}(y_{1})}\big),
\]
and set
\[
\mu_{\tau}^{r,\varepsilon}:=\pi_{\#}^{2}\gamma^{r,\varepsilon}
-\tau h \hspace{1mm}dy.
\]
Observe that $\pi_{\#}^{1}\gamma^{r,\varepsilon}=\pi_{\#}^{1}\gamma,$
$(\gamma^{r,\varepsilon},h)\in ADM(\rho,\mu_{r}^{r,\varepsilon}),$
$(\gamma^{r,\varepsilon})_{\partial\Omega}^{\Omega}=\gamma_{\partial\Omega}^{\Omega}-\varepsilon\gamma_{\partial\Omega}^{B_{r}(y_{1})}$,
$(\gamma^{r,\varepsilon})_{\Omega}^{\partial\Omega}=\gamma_{\Omega}^{\partial\Omega}+\varepsilon\big(\pi^{1},\mathcal{T}_{y_{1}}^{y_{2}})_{\#}\gamma_{\Omega}^{B_{r}(y_{1})},$
and $\mu_{\tau}^{r,\varepsilon}=\rho_{\tau}^{r,\varepsilon}\mathcal{L}^{d}.$
Here, $\rho_{\tau}^{r,\varepsilon}$ is given by
\[\rho_{\tau}^{r,\varepsilon}=
\begin{cases}
\rho_{\tau}(y) & \mbox{if }y\in B_{r}(y_{1})^{c},\\
(1-\varepsilon)(\rho_{\tau}(y)+\tau h)-\tau h & \mbox{if }y\in B_{r}(y_{1}).
\end{cases}
\]
(We remark that by Lemma \ref{cre-mass} $h$ is in $L^{\infty}(\Omega)$ and by Lemma \ref{uniform} 
we that $\rho_{\tau}$ is bounded and uniformly positive. Hence, we
can guarantee that for sufficiently small $\varepsilon,$ $\rho^{r,\varepsilon}_{\tau}$ is strictly positive.)

From the minimality of $\rho_{\tau}$ and the relationship between $\gamma,$ $S,$ and $T,$ we get
\begin{align*}
0 & \leq\int_{\Omega}\mathcal{E}(\rho_{\tau}^{r,\varepsilon})\hspace{1mm}dx+C_{\tau}(\gamma^{r,\varepsilon},h)-\int_{\Omega}\mathcal{E}(\rho_{\tau})\hspace{1mm}dx-C_{\tau}(\gamma,h)\\
 & =\int_{B_{r}(y_{1})}
 \mathcal{E}((1-\varepsilon)(\rho_{\tau}(y)+\tau h)-
 \tau h)
 -\mathcal{E}(\rho_{\tau})
 \hspace{1mm}dy\\
 & +\varepsilon\int_{B_{r}(y_{1})}
 \bigg (\frac{| \mathcal{T}_{y_{1}}^{y_{2}}(y)
 -S(y)|^{2}}{2\tau}1_{\{S(y)\in \Omega\}} 
 -\frac{| y-S(y)|^{2}}{2\tau}\bigg)(\rho_{\tau}(y)+\tau h)\hspace{1mm}dy\\
 & +\varepsilon\int_{B_{r}(y_{1})}\big[
 \Psi(\mathcal{T}_{y_{1}}^{y_{2}}(y))1_{\{S(y)\in \Omega\}}
 +\Psi(S(y))1_{\{S(y)\in \partial\Omega\}} \big](\rho_{\tau}(y)+\tau h)\hspace{1mm}dy.\end{align*}
Dividing by $\varepsilon$ and letting $\varepsilon\downarrow0,$ using 
Lemma \ref{uniform}, the dominated convergence Theorem, and
the fact that $\mathcal{E}$ in Lipschitz any compact subset of
$D(\mathcal{E}),$
we
obtain
\begin{align*}
\int_{B_{r}(y_{1})}\mathcal{E^{\prime}}(\rho_{\tau}(y))(\rho_{\tau}(y)& +\tau h)\hspace{1mm}dy\\
\leq & \int_{B_{r}(y_{1})}\bigg(\frac{| \mathcal{T}_{y_{1}}^{y_{2}}(y)-S(y)|^{2}}{2\tau}1_{\{S(y)\in \Omega\}}-\frac{| y-S(y)|^{2}}{2\tau}\bigg)(\rho_{\tau}(y)+\tau h)\hspace{1mm}dy\\
 & +\int_{B_{r}(y_{1})}\big[\Psi(\mathcal{T}_{y_{1}}^{y_{2}}(y))1_{\{S(y)\in \Omega\}}
 +\Psi(S(y)) 1_{\{S(y)\in \partial\Omega\}} \big](\rho_{\tau}(y)+\tau h)\hspace{1mm}dy.
\end{align*}
Recall that by assumption $S(y_{1})\not\in\partial\Omega.$ Now, since $y_{1}$ is a 
density point for $\rho_{\tau}+\tau h,$ and a Lebesgue point for $S$ when we
divide both sides by $\mathcal{L}^{d}(B_{r}(0))$ and we let $r\downarrow0,$
we obtain \eqref{eq:-6}.
\end{proof} \end{proposition}
Henceforth, we will omit the proof of these kinds of perturbation
arguments. They can be made rigorous using the ideas contained in the previous
Proposition. In the next Proposition,
the constants $C_{0}$ and $s$ are the ones described in the introduction.
\begin{proposition} \textbf{ ($L^{\infty}$Barriers) } 
With the notation introduced at the beginning of this section, the following holds:
There exists $\varepsilon\in(0,1),$ such that if  
$\lambda$ and $\Lambda$ satisfy 
$0<\lambda<\varepsilon<\frac{1}{\varepsilon}<\Lambda$ and
\[
\frac{\lambda}{\sup\{e^{-V}\}}e^{-V}\leq
\rho\leq\frac{\Lambda}{\inf\{e^{-V}\}}e^{-V},
\]
then
\[
\bigg(\frac{1}{1+C_{0}\tau}\bigg)\frac{\lambda}{\sup\{e^{-V}\}}e^{-V}\leq \rho_{\tau}\leq\frac{\Lambda}{\inf\{e^{-V}\}}e^{-V}.
\]
Here, $\varepsilon$ depends only on $ e,$
$||\Psi||_{\infty},$ and $||V||_{\infty}.$  
\begin{proof}
We first prove the lower bound.\\
Assumption (C8) allows us to choose 
$\varepsilon\in(0,s)$ 
so that 
\begin{equation}
[e_{x}^{\prime}]^{-1}(\log r+V(x))\leq C_{0}\hspace{1mm}r
\hspace{1em} {\rm{in}} \hspace{1em} \Omega,
\label{universal}
\end{equation}
and
\[\Psi>\log r+V
\hspace{1em} {\rm{on}} \hspace{1em} \partial\Omega
,\] 
for all $r\in(0,\varepsilon).$\\ 
Let $\lambda\in(0,\varepsilon)$ such that 
$\lambda \hspace{1mm} e^{-V}/ \sup\{e^{-V}\}\leq \rho$ and set 
\[A_{\lambda}=\bigg\{\rho_{\tau}<\bigg(\frac{1}{1+C_{0}\tau}\bigg)
\frac{\lambda}{\sup\{e^{-V}\}}e^{-V}\bigg\}.\]
For a contradiction suppose $\rho_{\tau}(A_{\lambda})>0$ (Note that by Lemma \ref{uniform}, $\rho_{\tau}$
is uniformly positive).

For each $x\in A_{\lambda},$ we perturb $(\gamma,h)$ by decreasing $h(x)$
and thus increasing the mass created at $x$. By optimality, we get
\begin{equation}
\log\rho_{\tau}(x)+V(x)-e^{\prime}(h(x))\geq0.
\label{pcreate}
\end{equation}
Since 
\[\bigg(\frac{1}{1+C_{0}\tau}\bigg)\frac{\lambda}{\sup\{e^{-V}\}}e^{-V}<s,
\]
when we combine \eqref{universal} and \eqref{pcreate} we conclude
\[h<\bigg(\frac{C_{0}}{1+C_{0}\tau}\bigg)\frac{\lambda}{\sup\{e^{-V}\}}e^{-V}
\hspace{1em} {\rm{in}} \hspace{1em} A_{\lambda}. 
\]
Let $C_{\lambda}=\{x\in A_{\lambda}\hspace{1mm}:\hspace{1mm}T(x)\not\in A_{\lambda}\}$
and note that $\rho(C_{\lambda})>0$. Otherwise by \eqref{interior-mass} and the previous inequality
\begin{align*}
\int_{A_{\lambda}}\bigg(\frac{1}{1+C_{0}\tau}\bigg)\frac{\lambda}{\sup\{e^{-V}\}}e^{-V}\hspace{1mm}dx&>
\rho_{\tau}(A_{\lambda})\geq\rho_{\tau}(T(A_{\lambda}))
\\ &\geq\rho(T^{-1}(T(A_{\lambda})))-\tau\int_{TA_{\lambda}}h\hspace{1mm}dx\\
&\geq \int_{A_{\lambda}}\frac{\lambda}{\sup\{e^{-V}\}}e^{-V}\hspace{1mm}dx-\bigg(\frac{C_{0}\tau}{1+C_{0}\tau}\bigg)\frac{\lambda}{\sup\{e^{-V}\}}e^{-V}\hspace{1mm}dx\\
&=\int_{A_{\lambda}}\bigg(\frac{1}{1+C_{0}\tau}\bigg)\frac{\lambda}{\sup\{e^{-V}\}}e^{-V}\hspace{1mm}dx.
\end{align*}
Define the sets
\[
C_{\lambda}^{1}:=\bigg\{ x\in C_{\lambda}\hspace{1mm}:\hspace{1mm}T(x)\in\Omega\bigg\}
\hspace{1em}{\rm{and}}\hspace{1em}
C_{\lambda}^{2}:=\bigg\{ x\in C_{\lambda}\hspace{1mm}:\hspace{1mm}T(x)
\in\partial\Omega\bigg\}.
\]
Since $C_{\lambda}=C_{\lambda}^{1}\cup C_{\lambda}^{2}$, we
have that either $\rho(C_{\lambda}^{1})>0$
or $\rho(C_{\lambda}^{2})>0$. Suppose we are in the first case. Then,
we can find a point $x$ which is a Lebesgue point for $T$ such that $T(x)$ is a
Lebesgue point for $\rho_{\tau}$. If we stop sending some mass
from $x$ to $T(x)$, then, by optimality we obtain
\[
\log\rho_{\tau}(x)+V(x)-\log(\rho_{\tau}(T(x))-V(T(x))-\frac{| x-T(x)|^{2}}{2\tau}\geq0.
\]
Since  
\[
\log\rho_{\tau}(T(x))\geq\log
\bigg[\bigg(\frac{1}{1+C_{0}\tau}\bigg)\frac{\lambda}{\sup\{e^{-V}\}}e^{-V(T(x))}
\bigg],\]
and
\[
\log\bigg[\bigg(\frac{1}{1+C_{0}\tau}\bigg)\frac{\lambda}
{\sup\{e^{-V}\}}e^{-V(x)}\bigg]>\log\rho_{\tau}(x),\]
we get a contradiction.

Now, suppose $\rho(C_{\lambda}^{2})>0$. We perturb $(\gamma,h)$
by not moving some mass from $x$ to the boundary. By optimality
we must have
\[
\log\rho_{\tau}(x)+V(x)-\Psi(T(x))-\frac{| x-T(x)|^{2}}{2\tau}\geq0.
\]
Since $\lambda>\rho_{\tau}(x)$
and $\Psi>\log\lambda+V(x),$ we get a contradiction.

Second, we prove the upper bound.\\
By assumptions (C1), (C7), (B1), and (B2), after making $\varepsilon$ smaller, 
we can guarantee that  
\[[e_{x}^{\prime}]^{-1}(\log r+V)>0 
\hspace{1em} {\rm{in}} \hspace{1em} \Omega,\]
and
\[\Psi<\log r+V
\hspace{1em} {\rm{on}} \hspace{1em} \partial\Omega
,\]
for all $r>1/\varepsilon.$\\
Let $\Lambda>1/\varepsilon$ satisfy  
$ \hspace{1mm} \frac{\Lambda}{\inf\{e^{-V}\}}e^{-V} \geq \rho$
and set $A_{\Lambda}=\bigg\{\rho_{\tau}>\frac{\Lambda}{\inf\{e^{-V}\}}e^{-V}\bigg\}.$\\
In order to get a contradiction, suppose $\rho_{\tau}(A_{\Lambda})>0.$\\
For each $x\in A_{\Lambda},$ we perturb $(\gamma,h)$ by increasing
$h(x)$ and hence decreasing the amount of mass created in $x$. By optimality we get
\[
e^{\prime}(h(x),x)-\log\rho_{\tau}(x)-V(x)\geq0.
\]
Since $\frac{\Lambda}{\inf\{e^{-V}\}}e^{-V}\geq\Lambda,$ we deduce that
$h$ is non-negative in $A_{\Lambda}$. Now, we consider the following cases:

\textit{Case 1:} the mass of \textbf{$\rho_{\tau}$ }in\textbf{ $A_{\Lambda}$
}does not come from\textbf{ $\partial\Omega.$} Let $B_{\Lambda}=T^{-1}(A_{\Lambda})$
and observe that due to \eqref{interior-mass},
\[
\int_{A_{\Lambda}}\frac{\Lambda}{\inf\{e^{-V}\}}e^{-V}\hspace{1mm}dx<\rho_{\tau}(A_{\Lambda})\leq\rho(B_{\Lambda})-\tau\int_{A_{\Lambda}}h\hspace{1mm}dx<\int_{B_{\Lambda}}\frac{\Lambda}{\inf\{e^{-V}\}}e^{-V}\hspace{1mm}dx,
\]
which implies
\[
| A_{\Lambda}|<| B_{\Lambda}|.
\]
Hence, we can find a Lebesgue point $x\in B_{\Lambda}\backslash A_{\Lambda}$.
If we stop transporting some mass from $x$ to $T(x),$ then  by
optimality, we obtain
\[
-\frac{| x-T(x)|^{2}}{2\tau}+\log\rho_{\tau}(x)+V(x)-\log\rho_{\tau}(T(x))-V(T(x))\geq0.
\]
Now by construction,
\[
\log\frac{\Lambda}{\sup\{e^{-V}\}}e^{-V(x)}>\log\rho_{\tau}(x),\hspace{1em}
{\rm{and}}
\hspace{1em}\log\rho_{\tau}(T(x))>\log\frac{\Lambda}{\sup\{e^{-V}\}}e^{-V(T(x))}.
\]
When we combine this with the previous inequality we reach a contradiction.

\textit{Case 2:} the mass of\textbf{ $\rho_{\tau}$ }comes partially from\textbf{
$\partial\Omega$. }Let $D_{\Lambda}\subset A_{\Lambda}$
be the set of points $y$ such that the mass $\rho_{\tau}(y)$ comes
from the boundary;
i.e., $D_{\Lambda}:=\{y\in A\hspace{1mm}:\hspace{1mm}S(y)\in\partial\Omega\}.$
Also, let  $y\in D_{\Lambda}$ be a Lebesgue point for $S.$ Then, if we stop
moving some mass from $S(y)$ to $y,$ by optimality we obtain
\[
\frac{-| S(y)-y|^{2}}{2\tau}+\Psi(S(y))-\log\rho_{\tau}(y)-V(y)\geq0.
\]
Since $\rho_{\tau}(x)>\Lambda$
and $\Psi<\log\Lambda+V(x),$ we get a contradiction. This concludes 
the proof.
\end{proof} \label{bar} \end{proposition}
For the next Lemma we recall that we have assumed that 
$\gamma_{\partial\Omega}^{\partial\Omega}=0.$
\begin{lemma}\textbf{(Transportation bound)} Let $\varepsilon$,
$\rho$, $\lambda,$ and $\Lambda$ be as in Proposition \ref{bar}.
Then, there exists $C>0$ such that
\[
| y-x|\leq C\sqrt{\tau}\hspace{1em}\forall(x,y)\in supp(\gamma). 
\]
Here, $C$ depends only on $\varepsilon,$  $\lambda$, $\Lambda,$ $||\Psi||_{\infty},$
and $||V||_{\infty}.$ 
\begin{proof} Let $(x,y)$ be a point in $supp(\gamma).$ Then, we perturb the plan $\gamma$ by not moving some mass from $x$
to $y$. By optimality,
\[
\Psi(y)1_{\Omega\times\partial\Omega}-\Psi(x)1_{\partial\Omega\times
\Omega}+
\big(\log\rho_{\tau}(x)+V(x)\big)1_{\Omega}(x)-
\big(\log\rho_{\tau}(y)-V(y)\big)1_{\Omega}(y)-\frac{| 
x-y|^{2}}{2\tau}\geq0.
\]
Thus, the result follows (B1), (B2), (C1), and Proposition \ref{bar}. 
 \label{bound} \end{proof} \end{lemma}
\begin{lemma} \textbf{ (Boundary Mass Flux estimate) }Let
$\varepsilon$, $\rho$, $\lambda,$ and $\Lambda$ be as in Proposition
\ref{bar}. Then, there exists $C>0$ such that 
\[
\gamma(\partial\Omega\times\Omega\cup\Omega\times\partial\Omega)\leq C\sqrt{\tau}.
\]
Here, $C$ depends only on $\varepsilon,$ $\lambda,$ $\Lambda$,  $||\Psi||_{\infty},$
and $||V||_{\infty}.$ 
\begin{proof}
By Lemma \ref{bound}, no mass either sent or taken from
the boundary travels more than $C\sqrt{\tau}$. Then,
 at most a $C\sqrt{\tau}$ neighborhood
of $\partial\Omega$ can be sent to the boundary. The mass taken from
the boundary can fill at most a $C\sqrt{\tau}$ neighborhood of $\partial\Omega$. Hence, 
the desired result follows from (B2) and Proposition \ref{bar}.
\label{boundary} \end{proof} \end{lemma}
\begin{lemma} \textbf{(Interior Mass Creation estimate) } Let
$\varepsilon$, $\rho$, $\lambda,$ and $\Lambda$ be as in Proposition
\ref{bar}. Then, there exists $C>0$  such that 
\[
\int_{\Omega}| h| \hspace{1mm} \hspace{1mm}dx\leq C
\hspace{1em}
{\rm{and}}
\hspace{1em}
| h|\leq C.
\]
Here, C  depends only on  $\varepsilon,$ $\Omega$, $\lambda,\Lambda,$ and 
$||V||_{\infty}.$
\begin{proof} By item $(iii)$ of Proposition \ref{step}, we know
\[
e^{\prime}(h)=\log\rho_{\tau}+V.
\]
Consequently, by Proposition \ref{bar},
\[
\log\bigg[\bigg(\frac{1}{1+C_{0}\tau}\bigg)\lambda\bigg]-||V||_{\infty}
\leq e_{x}^{\prime}(h(x))\leq\log(\Lambda)+||V||_{\infty},
\hspace{1em} \forall x\in \Omega.
\]
Using assumptions (C2), (C7), (B1) and (B2), we get that $h$ is bounded.
Thus, since $\Omega$ is bounded, the result follows.
\label{interior} \end{proof} \end{lemma}
\begin{corollary}\label{comparable} Let
$\varepsilon$, $\rho$, $\lambda,$ and $\Lambda$ be as in Proposition
\ref{bar}. Then, there exist  positive constants $\kappa_{1},$ $\kappa_{2},$ and $\delta$  such that
\[
\kappa_{1}<\frac{\rho_{\tau}}{\rho_{\tau}+\tau h} <\kappa_{2}.
\]
for every $\tau\in(0,\delta).$ Here, $\kappa_{1}$ and $\kappa_{2}$  depend only on $\varepsilon, \lambda, \Lambda,||\Psi||_{\infty},$ and 
$||V||_{\infty}.$
\begin{proof}
This follows directly from Proposition \ref{bar} and
Lemma \ref{interior}. 
\end{proof}

\end{corollary}

\begin{lemma} \textbf{(Boundary cost bound)} Let $\varepsilon$,
$\rho$, $\lambda,$ and $\Lambda$ be as in Proposition \ref{bar}.
Then, for every $\epsilon>0$, there exists $C>0$ such that 
\begin{equation*}
\int_{\overline{\Omega}\times\overline{\Omega}}\Psi(y)1_{\Omega\times\partial\Omega}-\Psi(x)1_{\partial\Omega\times\Omega}\hspace{1mm}d\gamma\geq\int_{\Omega}\Psi \hspace{1mm}d\mu-\int_{\Omega}\Psi \hspace{1mm}d\mu_{\tau}-\epsilon \int_{\overline{\Omega}\times\overline{\Omega}}\frac{| x-y|^{2}}{2\tau}\hspace{1mm}d\gamma\hspace{1mm}
-C\bigg( 1+\frac{1}{\epsilon} \bigg)\tau.
\end{equation*}
Here,  $C$ depends only on
$\varepsilon, {\rm{Lip}}\Psi,$ $\lambda,$ $\Lambda,$ and $||V||_{\infty}$. 
\begin{proof}
Set
$\zeta=\Psi$ in \eqref{eq:pre weak}. By doing this and rearranging terms, we get
\begin{align*}
\int_{\overline{\Omega}\times\overline{\Omega}}\Psi(y)&1_{\Omega\times\partial\Omega}-\Psi(x)1_{\partial\Omega\times\Omega}\hspace{1mm}d\gamma=\int_{\Omega}\Psi \hspace{1mm}d\mu-\int_{\Omega}\Psi \hspace{1mm}d\mu_{\tau}\\
& - \bigg(\int_{\Omega\times\Omega}\Psi(x)-\Psi(y) \hspace{1mm} d\gamma +\tau \int_{\Omega}\Psi h\hspace{1mm}dx\bigg)+R(\Psi,\tau),
\end{align*}
where, 
\begin{multline*}
R(\Psi,\tau)=\int_{\overline{\Omega}\times\overline{\Omega}}\bigg(\Psi\circ\pi_{2}-\Psi\circ\pi_{1}\bigg)1_{\{x=S(y)\}}\hspace{1mm}d\gamma{}_{\partial\Omega}^{\Omega}\\
-\int_{\overline{\Omega}\times\overline{\Omega}}\bigg(\Psi\circ\pi_{1}-\Psi\circ\pi_{2}\bigg)1_{\{T(x)=y\}}\hspace{1mm}d\gamma{}_{\Omega}^{\partial\Omega}. 
\end{multline*}
First, by Lemma \ref{bound} and Lemma \ref{boundary},
\[
| R(\Psi,\tau)|\leq C_{1}( {\rm Lip}\Psi,\lambda,\Lambda,\Omega,||V||_{\infty}) \tau.
\]
Second, by Lemma  \ref{interior} we have
\[
-\tau \int_{\Omega}\Psi h\hspace{1mm}dx\hspace{1mm}\geq-C_{3}(||\Psi||_{\infty},V,\lambda,\Lambda)|\Omega|\tau.
\]
Finally, by the  Young inequality, Proposition \ref{step}, Proposition \ref{bar}, and Lemma \ref{interior},
\begin{align*}
 -\int_{\Omega\times\Omega}\Psi(x)-\Psi(y)\hspace{1mm} d\gamma &\geq-\int_{\Omega} {\rm Lip} \Psi \hspace{0.1mm} |x-y|\hspace{1mm}d\gamma\\
 & -\frac{\tau}{2\epsilon}\int ({\rm Lip}\Psi)^{2}d\gamma-\epsilon\int_{\overline{\Omega}\times\overline{\Omega}}\frac{|x-y|^{2}}{2\tau}\hspace{1mm}d\gamma\\ 
 &\hspace{1em}\geq -C_{4}(\epsilon,\Psi,e,V,\lambda,\Lambda,\Omega)\frac{\tau}{\epsilon}-\epsilon\int_{\overline{\Omega}\times\overline{\Omega}}\frac{| x-y|^{2}}{2\tau}\hspace{1mm}d\gamma.\\&
\end{align*}
Thus, the desired result follows.
\end{proof} \label{cost bound} \end{lemma}

\begin{proposition} \textbf{(Energy Inequality) } Let
$\varepsilon$, $\rho$, $\lambda,$ and $\Lambda$ be as in Proposition
\ref{bar}. Then, there exist  positive constants $C$ and $\delta$  such that
\[
\int_{\overline{\Omega}\times\overline{\Omega}} \frac{| x-y|^{2}}{2\tau}\hspace{1mm}d\gamma\leq C\bigg(E(\rho)-\int_{\Omega}\Psi \hspace{1mm}d\mu-E(\rho_{\tau})+\int_{\Omega}\Psi \hspace{1mm}d\mu_{\tau}+\tau\bigg),
\]
for every $\tau\in(0,\delta).$ Here, $C$ depends only on $\varepsilon, \lambda, \Lambda, {\rm{Lip}}\Psi,$ and 
$||V||_{\infty}.$
\label{inequality2}

\begin{proof}

By minimality of $\rho_{\tau},$ we obtain
\[
\int_{\overline{\Omega}\times\overline{\Omega}}\frac{| x-y|^{2}}{2\tau}+\Psi(y)1_{\Omega\times\partial\Omega}-\Psi(x)1_{\partial\Omega\times\Omega}\hspace{1mm}d\gamma+\tau\int_{\Omega} e(h)\hspace{1mm}dx+E(\rho_{\tau})\leq E(\rho).
\]
Also, by Lemma \ref{interior} and the above inequality,
\[
\int_{\overline{\Omega}\times\overline{\Omega}}\frac{| x-y|^{2}}{2\tau}+\Psi(y)1_{\Omega\times\partial\Omega}-\Psi(x)1_{\partial\Omega\times\Omega}\hspace{1mm}d\gamma\leq E(\rho)-E(\rho_{\tau})+C_{1}(\Psi,e,V,\mu,\Omega)\tau.\]
Now, using the above inequality and Lemma \ref{cost bound}, we obtain
\begin{multline*}
\int_{\overline{\Omega}\times\overline{\Omega}}\frac{| x-y|^{2}}{2\tau}\hspace{1mm}d\gamma+\int_{\Omega}\Psi \hspace{1mm}d\mu-\int_{\Omega}\Psi \hspace{1mm}d\mu_{\tau}-\epsilon \int_{\overline{\Omega}\times\overline{\Omega}}\frac{| x-y|^{2}}{2\tau}\hspace{1mm}d\gamma\hspace{1mm}\tau
-C_{2}\bigg( 1+\frac{1}{\epsilon} \bigg)\tau\\
\leq E(\rho)-E(\rho_{\tau})+C_{1}(\Psi,e,V,\lambda,\Lambda,\Omega)\tau.
\end{multline*}
Here, $C_{2}:=C_2(\Psi,e,V,\lambda,\Lambda,\Omega).$ Then, the result
follows by first choosing $\epsilon$ and then $\delta$ 
appropriately in the above inequality.
\end{proof} \end{proposition}
For the next proposition, we will need the map 
$P\hspace{1mm}:\hspace{1mm}\rightarrow\mathbb{R}^{d}$, which 
was defined in 
Section 3. Such a map satisfies
\[
 |x-P(x)|=d(x,\partial\Omega)
 \hspace{1em}
 \forall{x\in\Omega}.
\]
\begin{lemma} \textbf{(Projection estimate)} \label{projection}
Assume $\Omega$ satisfies the interior ball condition with radius $r>0$.
Then, for all $x$ with $d(x,\partial\Omega)<\frac{r}{2}$, we have  
\[
| P(x)-P_{\Psi,\tau}(x)|\leq 4\tau{\rm{Lip}}\Psi
\hspace{1em}
{\rm{and}}
\hspace{1em}
| P(x)-P_{-\Psi,\tau}(x)|\leq 4\tau{\rm{Lip}}\Psi.
\]
\end{lemma} \begin{proof} Let
$x\in\Omega$ such that $\hspace{1mm}d(x,\partial\Omega)<\frac{r}{2}.$ By the interior ball condition, $P(x)$ is unique. 
For a contradiction, suppose 
\[
| P(x)-P_{\Psi,\tau}(x)|>4\tau\mbox{Lip\ensuremath{\Psi}}.
\]
Denote by $Q$ the center of the circle of radius $r$ that is tangent
to $\partial\Omega$ at $P(x)$ and is contained in $\Omega$. Using the cosine law and the
fact that $| Q-P_{\Psi,\tau}(x)|\geq r,$
we can see that 
\begin{align*}
| x-P_{\Psi,\tau}(x)|^{2}& -| x-P(x)|^{2}  \\
 &\geq |P(x)-P_{\Psi,\tau}(x)|^{2}\bigg(1-\frac{| x-P(x)|}{r}\bigg)
  \geq\frac{| P(x)-P_{\Psi,\tau}(x)|^{2}}{2}.
\end{align*}
Hence, 
\[
\frac{| x-P_{\Psi,\tau}(x)|^{2}}{2\tau}-\frac{| x-P(x)|^{2}}{2\tau}+\Psi(P_{\Psi,\tau}(x))-\Psi(P(x)),
\]
is bounded from below by 
\[
\frac{| P(x)-P_{\Psi,\tau}(x)|^{2}}{4\tau}-\mbox{Lip}\Psi| P(x)-P_{\Psi,\tau}(x)|.
\]
Our assumption implies that the above quantity is strictly positive. This
contradicts the minimality of $P_{\Psi,\tau}(x)$. Thus, we get the first inequality of the Lemma. The second inequality can
be shown using the same argument.
\end{proof}

\appendix
\section{Minimizers of problem 1.1}
In this section, we study properties of the minimizers
of Problem 1.1 that are needed for Section 3. For this 
purpose, we let $\mu$ and
$\rho\hspace{1mm} dx$ be absolutely continuous measures 
in $\mathcal{M}(\Omega$) and let $\tau$
be a fixed positive number. Additionally, we
define
$m_{r}\hspace{1mm}:\hspace{1mm}\overline{\Omega}\rightarrow\mathbb{R}$
by 
\[
m_{r}(x):=[e_{x}^{\prime}]^{-1}(r),
\]
for any $r$ in $\mathbb{R}.$ 

Henceforth, we will say that a plan is optimal in the classical
sense if it is an optimal plan for the cost $d(x,y)=|x-y|^2.$ Whenever
$\gamma$ is an optimal plan in the classical sense 
and $\mu=\pi_{1\#}\gamma$ is absolutely continuous, we can guarantee the existence of a map
$T$ such that $(Id,T)_{\#}\mu=\gamma$ (see, for example, \cite[Theorem 6.2.4 and Remark 6.2.11]
{Ambrosio-Gigli-Savare}). Any map satisfying the previous property will be
called optimal in the classical sense.
\begin{lemma}
\label{refinement}
\textbf{(Refinement of pairs) }Let $\mu$ and $\rho \hspace{1mm} dx$ be
absolutely continuous measures in $\mathcal{M}(\Omega)$ and let $\tau$
be a positive constant.  Then, for any $(\gamma,h)$ in $ADM(\mu,\rho)$
there exists $(\gamma^{\prime},h)$ and $(\gamma^{\prime\prime},h^{\prime})$ in $ADM(\mu,\rho)$
with the following properties:
\begin{enumerate}[(i)]
\item The plans 
$(\gamma^{\prime})_{\Omega}^{\overline{\Omega}}$
and 
$(\gamma^{\prime})_{\overline{\Omega}}^{\Omega}$
are optimal in the classical sense, $(\gamma^{\prime})_{\partial\Omega}^{\partial\Omega}=0$
and
\[
C_{\tau}(\gamma^{\prime},h)-
C_{\tau}(\gamma,h)=
\int_{\overline{\Omega}\times\overline
{\Omega}\backslash\partial\Omega
\times\partial\Omega}\frac{|
x-y|^{2}}{2\tau}\hspace{1mm}d\gamma^{\prime}-\int_{\overline{\Omega}\times\overline{\Omega}\backslash\partial\Omega\times\partial\Omega}\frac{| x-y|^{2}}{2\tau}
\hspace{1mm}d\gamma.
\]
\item We have
\[
h^{\prime}(x)>[e_{x}^{\prime}]^{-1}\bigg(-\frac{\mbox{diam(\ensuremath{\Omega})}^{2}}{2\tau}-||\Psi||_{\infty}\bigg),
\]
for almost every $x$ in $\Omega$ and
\begin{multline}
\label{comp}
C_{\tau}(\gamma^{\prime\prime},h^{\prime})-C_{\tau}(\gamma^{\prime},h)\leq\tau\int_{A_{r}}\bigg[e^{\prime}(m_{r})+
\frac{| P_{-\Psi,\tau,}(y)-y|^{2}}{2\tau}-
\Psi(P_{-\Psi,\tau,}(y))\bigg](m_{r}-h)\hspace{1mm}dy\\
+\tau\varepsilon\int_{A_{R}^{\kappa}}\bigg[-e^{\prime}(h-
\varepsilon(\rho+\tau h))+\frac{\mbox{diam}^{2}(\Omega)}{2\tau}\\
-\frac{|S(y)-y|^{2}}{2\tau}+2||\Psi||_{\infty} \bigg](\rho+\tau h)\hspace{1mm}dy
\leq0 .
\end{multline}
\end{enumerate}
Here, $S$ is an optimal map, in the classical sense, such that $(S,Id)_{\#}(\rho+\tau h)=\gamma_{\overline{\Omega}}^{\prime\Omega}$
(this exists by the absolute continuity of $\rho+\tau h$), $A_{r}=\{h<m_{r}\},$
$A_{R}^{\kappa}=\{h>m_{R}+1\}\cap\{\rho+\tau h<\kappa\},$
$\kappa$ is a positive constant, and 
$\varepsilon<{\rm{min}}
(1\backslash\kappa, 1\backslash \tau) .$
Also, $r$ and $R$ are constants satisfying
\[
r<-\frac{\mbox{diam}(\Omega)^{2}}{2\tau}-||\Psi||_{\infty},
\]
\[
R>\frac{\mbox{diam}(\Omega)^{2}}{\tau}+2||\Psi||_{\infty},
\]
and
\[
 m_{R}>0 \hspace{1em} {\rm{in}} \hspace{1em} \Omega.
\]
\begin{proof}
It is easy to verify that if $\tilde{\gamma}$ satisfies 
$\pi_{i}\gamma_{\Omega}^{\overline{\Omega}}=
\pi_{i}(\tilde{\gamma})_{\Omega}^{\overline{\Omega}}$
and
$\pi_{i}\gamma_{\overline{\Omega}}^{\Omega}=
\pi_{i}(\tilde{\gamma})_{\overline{\Omega}}^{\Omega}$
for $i=1$ and $i=2,$ then $(\tilde{\gamma},h)\in ADM(\mu,\rho),$
and
\[
C_{\tau}(\tilde{\gamma},h)-C_{\tau}(\gamma,h)=\int_{\overline{\Omega}\times\overline{\Omega}\backslash\partial\Omega\times\partial\Omega}\frac{| x-y|^{2}}{2\tau}\hspace{1mm}d\tilde{\gamma}-\int_{\overline{\Omega}\times\overline{\Omega}\backslash\partial\Omega\times\partial\Omega}\frac{| x-y|^{2}}{2\tau}
\hspace{1mm}d\gamma.
\]
Consequently, if 
$(\tilde{\gamma})_{\overline{\Omega}}^{\Omega}$
and
$(\tilde{\gamma})_{\Omega}^{\overline{\Omega}}$
are optimal plans in the classical
sense, then 
\[
C_{\tau}(\tilde{\gamma},h)\leq C_{\tau}(\gamma,h).
\]
Now, $(i)$ follows from the observation that $(\tilde{\gamma}-\tilde{\gamma}_{\partial\Omega}^{\partial\Omega},h)\in ADM(\mu,\rho)$
and 
\[
C_{\tau}(\tilde{\gamma},h)=C_{\tau}(\tilde{\gamma}-\tilde{\gamma}_{\partial\Omega}^{\partial\Omega},h).
\]
We proceed to the proof of $(ii)$. Let $\gamma^{\prime}$ be the plan given by item $(i)$.
Define $h^{\prime}$ and $\gamma^{\prime\prime}$ by 
\[
h^{\prime}=h+(m_{r}-h)1_{A_{r}}-\varepsilon\pi_{2}^{\#}(\gamma^{\prime})_{\overline{\Omega}}^{A_{R}^{\kappa}},
\]
\[
\gamma^{\prime\prime}=\gamma^{\prime}+\tau(P_{-\Psi,\tau,}Id)_{\#}(m_{r}-h)1_{A_{r}}-\tau\varepsilon(\gamma^{\prime})_{\overline{\Omega}}^{A_{R}^{\kappa}}+\tau\varepsilon(Id,P_{\Psi,\tau,})_{\#}(\gamma^{\prime})_{\Omega}^{A_{R}^{\kappa}}.
\]
Here, we are using same notation as in the statement of the Lemma.
Observe that by \eqref{interior-mass},
$(\gamma^{\prime\prime},h^{\prime})\in ADM(\mu,\rho),$
\[
h^{\prime}=\begin{cases}
h & \mbox{{in}}\hspace{1em}\Omega\backslash A_{R}^{\kappa}\cup A_{r},\\
m_{r} & \mbox{{in}}\hspace{1em}A_{r},\\
h-\varepsilon(\rho+\tau h) & \mbox{{in}}\hspace{1em}A_{R}^{\kappa},
\end{cases}
\]
and 
\[
\pi_{2\#}\gamma^{\prime\prime}_{|\Omega}=\begin{cases}
\rho+\tau h & \mbox{{in}}\hspace{1em}\Omega\backslash A_{R}^{\kappa}\cup A_{r},\\
\rho+\tau m_{r} & \mbox{{in}}\hspace{1em}A_{r},\\
(1-\tau \varepsilon)(\rho+\tau h) & \mbox{{in}}\hspace{1em}A_{R}^{\kappa}.
\end{cases}
\]
Hence,
\begin{multline*}
C_{\tau}(\gamma^{\prime\prime},h^{\prime})-C_{\tau}(\gamma^{\prime},h)\\
\hspace{1.5em}
\leq \tau\int_{A_{r}}\big[e(m_{r})-e(h)\big]\hspace{1mm}dy+\tau\int_{A_{r}}
\bigg(\frac{| P_{-\Psi,\tau,}(y)-y|^{2}}{2\tau}-\Psi(P_{-\Psi,\tau,}
(y))\bigg)(m_{r}-h)\hspace{1mm}dy\\
\hspace{1.6em}
+\tau\int_{A_{R}^{\kappa}}e(h-\varepsilon(\rho+\tau h))-e(h)
\hspace{1mm}dy+\varepsilon\tau\int_{A_{R}^{\kappa}}\bigg
(\frac{\mbox{diam}^{2}(\Omega)}{2\tau}+||\Psi||_{\infty}\\
\hspace{21em} -\frac{|S(y)-y|^{2}}{2\tau}+||\Psi||_{\infty})\bigg)(\rho+\tau h)\hspace{1mm}dy.
\end{multline*}
Then, the desired result follows by the convexity of $e$ with respect
to its first variable
and the definition of $r,$ $R,$ and $\varepsilon.$
\end{proof}
\end{lemma}
For the next lemma, we will need the set  $D(e),$ which was previously defined
to be the  interior of the set of points such that
$e$ is finite.
\begin{lemma}\textbf{(Optimal maps and Bounds on the created mass)}
\label{cre-mass} Let $\mu$ and $\rho \hspace{1mm} dx$ be
absolutely continuous measures in $\mathcal{M}(\Omega)$ and let $\tau$
be a positive constant. Additionally, let $(\gamma,h)$
be a pair in $Opt(\mu,\rho)$. Then
\begin{enumerate}[(i)]
 \item The plans $\gamma_{\Omega}^{\overline{\Omega}}$ and
 $\gamma_{\overline{\Omega}}^{\Omega}$ are optimal in the classical sense.
 \item There exist maps $T$ and $S$ from $\Omega$ to $\overline{\Omega}$ such that
 \[
  (Id,T)_{\#}\mu=\gamma_{\Omega}^{\overline{\Omega}}, 
 \]
 and
 \[
  (S,Id)_{\#}(\rho+\tau h)=\gamma_{\overline{\Omega}}^{\Omega}. 
 \]
 \item There exists a compact
set $K\subset\mathbb{R}\times\overline{\Omega}$ contained in 
$D(e)$
such that
\[
(x,h(x))\in K,
\]
for a.e $x$ in $\Omega.$
\end{enumerate}
\begin{proof}
By Lemma \ref{refinement}, $(i)$ follows by optimality. Since
$\mu$ and $\rho+\tau h$ are absolutely continuous, $(ii)$
follows from the classical optimal transportation theory
(see for example \cite[Theorem 6.2.4 and Remark 6.2.11]
{Ambrosio-Gigli-Savare}).

Now, we proceed to the proof of  $(iii).$
Let $r,$ $R,$ and $(\gamma^{\prime\prime},h^{\prime})$ be defined as in the previous Lemma and set $K=\{(q,x)\in\mathbb{R}\times \overline{\Omega} \hspace{1mm}:\hspace{1mm}m_{r}(x)\leq q\leq m_{R}(x)+1\}.$
By (C7) and (C8), $K$ is compact. By  construction, both integrands
in \eqref{comp} are
strictly negative. Thus, from the minimality of $(\gamma,h)$, we conclude that $|A_{r}|=|A_{r}^{\kappa}|=0.$
Since $\kappa$ was arbitrary and $h+\tau\rho\in L^{1}(\Omega)$, we obtain
the desired result. 
\end{proof}
\end{lemma}
In the next proposition, we will say that a point in $\Omega$ is a density point for 
$\rho+\tau h$ if it is a point of with density $1$ for the set 
$\{\rho+\tau h>0 \}$ and it is a Lebesgue point for $\rho$ and $h.$
\begin{proposition} \label{perturba}
Let $\mu$ and $\rho \hspace{1mm} dx$ be
absolutely continuous measures in $\mathcal{M}(\Omega)$ and let $\tau$
be a positive constant. Additionally, let $(\gamma,h)$
be a pair in $Opt(\mu,\rho)$. If $T$
and $S$ are the maps given by Lemma \ref{cre-mass}, 
then the following inequalities hold: 
\begin{itemize}
\item Let $y_{1}$ and $y_{2}$ be points in $\Omega.$ Assume that $y_{1}$
is a density point for $\rho+\tau h$ and  a Lebesgue point for $S$
and that $y_{2}$ is a Lebesgue point
for $h$. Then 
\end{itemize}
\begin{equation}
e^{\prime}(h(y_{1}))+\frac{|y_{1}-S(y_{1})|^{2}}{2\tau}\leq e^{\prime}(h(y_{2}))+\frac{|y_{2}-S(y_{1})|^{2}}{2\tau}.\label{eq:-14}
\end{equation}
\begin{itemize}
\item Let $y_{1}$ and $S(y_{1})$ be points in $\Omega.$ Assume that $y_{1}$
is a density point for $\rho+\tau h$ and a Lebesgue point for $S$.
Then for any $y_{2}\in\partial\Omega,$ we have 
\begin{equation}
e^{\prime}\circ h(y_{1})+\frac{|y_{1}-S(y_{1})|^{2}}{2\tau}\leq\frac{|y_{2}-S(y_{1})|^{2}}{2\tau}+\Psi(y_{2}).\label{eq:-16}
\end{equation}
\item Let $x_{1}\in\Omega$ be a Lebesgue point for the density of $\mu,$ and
$T$ and assume that $T(x_{1})\in\partial\Omega.$ Assume further that
$y_{1}\in\Omega$ is a Lebesgue point for $h$. Then 
\end{itemize}
\begin{equation}
\frac{|x_{1}-T(x_{1})|^{2}}{2\tau}+\Psi(T(x_{1}))\leq e^{\prime}(h(y_{1}))+\frac{|x_{1}-y_{1}|^{2}}{2\tau}.\label{eq:-25}
\end{equation}
\begin{itemize}
\item Let $x_{1}\in\Omega$ be a Lebesgue point for the density of $\mu,$ and
$T$ and assume that $T(x_{1})\in\partial\Omega.$ Then for any $y_{1}$ in
$\partial\Omega,$  
\end{itemize}
\begin{equation}
\frac{|x_{1}-T(x_{1})|^{2}}{2\tau}+\Psi(T(x_{1}))\leq\frac{|x_{1}-y_{1}|^{2}}{2\tau}+\Psi(y_{1}).\label{eq:1-5-1}
\end{equation}
\begin{itemize}
\item Let $y_{1}\in\Omega$ be a Lebesgue point for $h.$ Then for any $x_{1}\in\partial\Omega,$
\begin{equation}
0\leq e^{\prime}\circ h(y_{1})+\frac{|y_{1}-x_{1}|^{2}}{2\tau}-\Psi(x_{1}).\label{eq:-16-1}
\end{equation}
\end{itemize}
\begin{itemize}
\item Let $y_{1}\in\Omega$ be a density point for $\rho+\tau h$
such that $S(y_{1})\in\partial\Omega.$ Then, for any $x_{1}$ in $\partial\Omega,$
\end{itemize}
\begin{equation}
 \frac{|y_{1}-S(y_{1})|^{2}}{2\tau}-
 \Psi(S(y_{1}))\leq\frac{|y_{1}-x_{1}|^{2}}{2\tau}
 -\Psi(x_{1}).
 \label{37}
\end{equation}
\begin{itemize}
\item Let $y_{1}\in\Omega$ be a density point for $\rho+\tau h$
such that $S(y_{1})\in\partial\Omega.$ Then
\end{itemize}
\begin{equation}
\frac{|y_{1}-S(y_{1})|^{2}}{2\tau}-\Psi(S(y_{1}))\leq-e^{\prime}(h(y_{1})).  
\end{equation}

\begin{proof}
We only prove \eqref{eq:-14}; the proofs of the other inequalities are
analogous. Also, Proposition \ref{perturba2} provides heuristic arguments
that illustrate the method used to prove those inequalities. This 
Proposition is the analogue of 
\cite[Proposition 3.7]{Figalli-Gigli} in our
context. We have decided to include this proof since this is the first times we explain how to make these kinds of arguments rigorous with perturbations that involve mass creation.\\
\textbf{-Heuristic argument}
We provide the idea to show \eqref{eq:-14}. First suppose $S(y_{1})\in\Omega.$
Then we can take some mass from $S(y_{1})$ and instead of sending it to $y_{1},$ we
send it to $y_{2}.$ In order to end up with we an admisible pair, we then have to create 
the missing mass at $y_{1}$ and remove the extra mass at $y_{2}.$ In order to do this,
we have to decrease $h(y_{1})$ and increase $h(y_{2}).$ By doing this we save
$\frac{|y_{1}-S(y_{1})|^2}{2\tau}$ and we pay
\[
\frac{|y_{2}-S(y_{1})|^2}{2\tau}-e^{\prime}(h_{1})+e^{\prime}(h_{2}). 
\]
Hence, \eqref{eq:-14} follows by minimality. If $S(y_{1})\in \partial\Omega,$ when we do 
the previous perturbation we save $\frac{|y_{1}-S(y_{1})|^2}{2\tau}+\Psi(S(y_{1}))$ and we pay
\[
\frac{|y_{2}-S(y_{1})|^2}{2\tau}+\Psi(S(y_{1}))-e^{\prime}(h_{1})+e^{\prime}(h_{2}),
\]
Thus, we get the same conclusion.\\
\textbf{-Rigorous proof}
We define $\gamma^{r,\varepsilon}$ and $h^{r,\varepsilon}$ by
\[
\gamma_{r}^{\varepsilon}=\gamma_{\overline{\Omega}}^{B_{r}(y_{1})^{c}}+(1-\varepsilon)\gamma_{\overline{\Omega}}^{B_{r}(y_{1})}+\varepsilon(\pi_{1},\mathcal{T}_{y_{1}}^{y_{2}})_{\#}\gamma_{\overline{\Omega}}^{B_{r}(y_{1})},\]
and
\[
h_{r}^{\varepsilon}=h-\frac{\varepsilon}{\tau}(\pi_{2\#}\gamma_{\overline{\Omega}}^{B_{r}(y_{1})})+\frac{\varepsilon}{\tau}(\mathcal{T}_{y_{1}}^{y_{2}}\circ\pi_{2\#}\gamma_{\overline{\Omega}}^{B_{r}(y_{1})}).\]
Here, $\mathcal{T}_{y_{1}}^{y_{2}}(y)=y-y_{1}+y_{2},$  and $r$ is small enough so that $B_{r}(y_{1})$ and $B_{r}(y_{2})$ are disjoint
and contained in $\Omega$. \\
Note that $\pi_{\#}^{1}(\gamma^{r,\varepsilon})=\pi_{\#}^{1}\gamma$
and $\pi_{\#}^{2}\gamma-\tau h=\pi_{\#}^{2}\gamma^{r,\varepsilon}-\tau h^{r,\varepsilon}$.
Hence, $(\gamma^{r,\varepsilon},h^{r,\varepsilon})\in ADM(\mu,\nu)$. By optimality, we must have
\[
0\leq C(h^{r,\varepsilon},\gamma^{r,\varepsilon})-C_{\tau}(h,\gamma).
\]
Thus,
\begin{align*}0 & \leq\varepsilon\int_{B_{r}(y_{1})}\bigg[\frac{|\mathcal{T}_{y_{1}}^{y_{2}}-S|^{2}}{2\tau}-\frac{|Id-S|^{2}}{2\tau}\bigg](\rho+\tau h)\hspace{1mm}dy\\
 & \hspace{10em}+\tau\int_{B_{r}(y_{1})}\bigg[e\bigg(h-\frac{\varepsilon}{\tau}(\rho_{\tau}+\tau h)\bigg)-e(h)\hspace{1mm}\bigg]\hspace{1mm}dy\hspace{1mm}\\
 & \hspace{15em}+\tau\int_{B_{r}(y_{1})}\bigg[e\bigg(h\circ\mathcal{T}_{y_{1}}^{y_{2}}+\frac{\varepsilon}{\tau}(\rho+\tau h)\bigg)-e(h\circ\mathcal{T}_{y_{1}}^{y_{2}})\hspace{1mm}\bigg]\hspace{1mm}dx.
\end{align*}
If we divide by $\varepsilon$ and let $\varepsilon\downarrow0,$
using Lemma \ref{cre-mass}, the fact that $e$ is locally Lipschitz
in $D(e),$ and the dominated convergence Theorem,
we obtain
\begin{align*}0 & \leq\int_{B_{r}(y_{1})}\bigg[\frac{|\mathcal{T}_{y_{1}}^{y_{2}}-S|^{2}}{2\tau}-\frac{|y-S|^{2}}{2\tau}\bigg](\rho+\tau h)\hspace{1mm}dy-\int_{B_{r}(y_{1})}e^{\prime}(h)(\rho+\tau h)\hspace{1mm}dy\\
 & \hspace{79mm}+\int_{B_{r}(y_{1})}e^{\prime}(h\circ\mathcal{T}_{y_{1}}^{y_{2}})(\rho+\tau h)\hspace{1mm}dy.
\end{align*}
Recall $y_{1}$ is a density point for $\rho+\tau h$ and a Lebesgue point for $S$ and $y_{1}$
and $y_{2}$ are Lebesgue points for $h.$ Hence, when we divide by $\mathcal{L}^{d}(B_{r}(0))$ and
we let $r\downarrow0,$ we obtain \eqref{eq:-14}.
\end{proof}
\end{proposition}
Henceforth, as we did in Section 5, we will omit the proof of these kinds of perturbation
arguments. They can be made rigorous using the ideas contained in the previous
Proposition. For the next proposition, we will need the sets  $D(e_{x}),$ which were previously defined
to be the  interior of the set of points $z$ such that $e(z,x)$ is finite.

\begin{proposition}\textbf{(Bounds on the transported mass)}
\label{trans-mass}  
With the notations and assumptions from Proposition \ref{perturba},
the following implication holds:
If there there exists a positive constant $\lambda_{0}$ 
such that
\[
\lambda_{0}\hspace{1mm} dx\leq \mu,
\]
and 
\[
\lambda_{0}\leq\rho,
\]
then there exists a positive number $\delta$ such that
\[
\frac{\lambda_{0}}{4}\leq \rho+\tau h,
\]
for all $\tau$ in $(0,\delta).$ Here, $\delta$ depends only on
$\lambda_{0},$ $\Psi,$ and $e.$
\begin{proof}
Let $\tilde{\rho}=\rho+\tau h.$ If the sets $D(e_{x})$ are
of the form $(a,\infty)$ with $a$ finite, then since $h(x)\in D(e_{x})$
and (C2), (C5), and (C8) hold, the lower bound follows easily by choosing $\delta$ sufficiently
small. Hence, we assume that the sets 
$D(e_{x})$ are of the form $(-\infty,\infty).$ (We remark that due to
(C8) the two conditions are mutually exclusive). 

By (C6) and (C8), there exits $\delta$ such that for every $\tau<\delta$
there exists $r$ such that 
\[
r<-||\Psi||_{\infty}
\]
and 
\[
\frac{1}{\tau}\bigg(\frac{\lambda_{0}}{4}-\rho\bigg)\leq m_{r}\leq\frac{1}{\tau}\bigg(\frac{\lambda_{0}}{2}-\rho\bigg)\hspace{1em}
{\rm{in}}\hspace{1em}\mbox{\ensuremath{\Omega}.}
\]
Set $A_{r}=\{\tilde{\rho}<\rho+\tau m_{r}(x)\}$ and $C_{r}=\{x\in A_{r}\hspace{1mm}:\hspace{1mm}T(x)\not\in A_{r}\}.$
For a contradiction, suppose 
$|A_{r}|>0.$ Note that $|C_{r}|\geq0.$
Otherwise, by \eqref{interior-mass}
\[
\frac{\lambda_{0}}{2}| A_{r}|>\tilde{\rho}(A_{r})\geq\tilde{\rho}(T(A_{r}))=\mu(T^{-1}T(A_{r}))\geq\lambda_{0}| A_{r}|.
\]
Define the sets 
\[
C_{r}^{1}:=\bigg\{ x\in C_{r}\hspace{1mm}:\hspace{1mm}T(x)\in\Omega\bigg\}\hspace{1em}{\rm {and}}\hspace{1em}C_{r}^{2}:=\bigg\{ x\in C_{r}\hspace{1mm}:\hspace{1mm}T(x)\in\partial\Omega\bigg\}.
\]
Since $C_{r}=C_{r}^{1}\cup C_{r}^{2},$ we have that either $| C_{r}^{1}|>0$
or $| C_{r}^{2}|>0.$ Suppose we are in the first case. Then,
we can find a point $x$ which is a Lebesgue point for $T$ such that
$T(x)$ is a Lebesgue point for $\tilde{\rho}.$ If we stop sending
some mass from $x$ to $T(x)$ then we can create the missing mass
at $T(x)$ and remove the extra mass at $x.$ To do this we have to
increase $h(x)$ and decrease $h(T(x)).$ By doing this, we produce
a pair in $ADM(\mu,\rho).$ Thus, by optimality, we must have
\[
e^{\prime}(h(x))-e^{\prime}(h(T(x))-\frac{| x-T(x)|^{2}}{2\tau}\geq0.
\]
By construction, if we use \eqref{interior-mass}, we obtain
\[
e^{\prime}(h(x))=e^{\prime}\bigg(\frac{\tilde{\rho}(x)-\rho(x)}{\tau}\bigg)<e^{\prime}(m_{r}(x))=r,
\]
and 
\[
e^{\prime}(h(T(x))=e^{\prime}\bigg(\frac{\tilde{\rho}(T(x))-\rho(T(x))}{\tau}\bigg)\geq e^{\prime}(m_{r}(x))=r.
\]
This  gives us  a contradiction. \\
Now, suppose $| C_{r}^{2}|>0.$
Then, we can find a point $x\in C_{r}^{2}$ such that $x$
is a Lebesgue point for $T$ and $h.$
If we stop moving some mass
from $x$ to the boundary, then we can remove the extra mass at $x$.
To do this we have to increase $h(x).$ By doing this, we produce
a pair in $ADM(\mu,\rho).$ By optimality, we must have
\[
e^{\prime}(h(x))-\Psi(T(x))-\frac{| x-T(x)|^{2}}{2\tau}\geq0.
\]
As before $e^{\prime}(h(x))<r$ and by construction $r-\Psi<0.$ This
gives us a contradiction. Hence, we conclude that
\[
\rho+\tau  h =\tilde{\rho}\geq \rho+\tau m_{r}\geq \frac{\lambda_{0}}{4}\hspace{1em}{\rm{a.e}}\hspace{3mm}{\rm{in}}\hspace{1em}\Omega.
\]
\end{proof}
\end{proposition}
\bibliographystyle{plain}
\bibliography{Preprint166}

\end{document}